\documentclass[12pt]{amsart}
\usepackage[T1]{fontenc}
\usepackage[utf8]{inputenc}
\usepackage{amsmath,amssymb,amsthm,mathrsfs}
\usepackage{graphicx}
\usepackage{enumerate} 
\usepackage{lmodern}
\usepackage{mathrsfs} 
\usepackage{dsfont}
\usepackage{color}
\usepackage[all]{xy}
\usepackage{hyperref}
\usepackage[margin=3.5cm]{geometry}
\usepackage{caption}
\usepackage{subcaption}

\DeclareMathOperator{\PH}{PH}
\DeclareMathOperator{\Mod}{Mod}
\DeclareMathOperator{\Hn}{H}
\DeclareMathOperator{\Ob}{Ob}
\DeclareMathOperator{\Supp}{Supp}
\DeclareMathOperator{\Diam}{diam}
\DeclareMathOperator{\ori}{{or\mspace{2mu}}}
\DeclareMathOperator{\vol}{Vol}

\begin{document}
	

\newcommand{\Lip}{\mathrm{Lip}}

\newcommand{\maxgen}[1]{\mathfrak{R}_{#1 !}}
\newcommand{\maxgeno}[1]{\mathfrak{R}_{#1 *}}

\newcommand{\sliced}{\mathcal{S}}
\newcommand{\V}{\mathbb{V}}
\newcommand{\W}{\mathbb{W}}
\newcommand{\E}{\mathbb{E}}
\newcommand{\RR}{\mathrm{R}}
\newcommand{\C}{\mathbb{C}}
\newcommand{\R}{\mathbb{R}}
\newcommand{\N}{\mathbb{N}}
\newcommand{\Z}{\mathbb{Z}}
\newcommand{\B}{\mathbb{B}}
\newcommand{\Proj}{\mathbb{P}}
\newcommand{\Cont}[1]{\mathscr{C}^{{\,#1}}}	
\newcommand{\dist}{{\mathrm{dist}}}	
\newcommand{\Der}[1][]{\mathsf{D}^{#1}}
\newcommand{\Derb}{\Der[\mathrm{b}]}
\newcommand{\cor}{{\mathbf{k}}}
\newcommand{\ism}[1]{{\mathrm{d}_{#1}}}
\newcommand{\isme}[1]{{\mathrm{\delta}_{#1}}}		\newcommand{\RHom}[1][]{\RR\mathrm{Hom}_{\raise1.5ex\hbox to.1em{}#1}}
\newcommand{\Hom}[1][]{\mathrm{Hom}_{\raise1.5ex\hbox to.1em{}#1}}

\newcommand{\oim}[1]{{#1}_*}
\newcommand{\eim}[1]{{#1}_!}
\newcommand{\roim}[1]{\RR{#1}_*}
\newcommand{\reim}[1]{\RR{#1}_!}
\newcommand{\reiim}[1]{\RR{#1}_{\mspace{1mu}!!}}
\newcommand{\reeim}[1]{\RR{#1}_{\mspace{1mu}!!}}
\newcommand{\opb}[1]{#1^{-1}}
\newcommand{\popb}[1]{#1^{\text{``}-1\text{''}}}
\newcommand{\epb}[1]{#1^{\,!}\,}
\newcommand{\spb}[1]{#1^{*}}
\newcommand{\eoim}[1]{{#1}_{!*}}
\newcommand{\lspb}[1]{\LL{#1}^{*}}
\newcommand{\sect}{\Gamma}
\newcommand{\rsect}{\mathrm{R}\Gamma}
\newcommand{\BBD}{{\mathbb D}}
\newcommand{\omA}[1][M]{\omega_{#1}}
\newcommand{\dual}{\mathrm{D}}
\newcommand{\rhom}[1][]{{\RR\mathscr{H}\mspace{-3mu}om}_{\raise1.5ex\hbox to.1em{}#1}}
\newcommand{\projbar}{\mathscr{P}}
\newcommand{\npprojbar}{\mathsf{P}}
\newcommand{\pt}{\mathrm{pt}}

\newcommand{\opnorm}[1]{{\left\vert\kern-0.25ex\left\vert\kern-0.25ex\left\vert #1 
	\right\vert\kern-0.25ex\right\vert\kern-0.25ex\right\vert}}
\newcommand{\norm}[1]{\left\| #1 \right\|}
\newcommand{\comp}{\mathrm{comp}}

\newcommand{\conv}[1][]{\mathop{\circ}\limits_{#1}}
\newcommand{\npconv}[1][]{\mathop{\circ}\limits^{\rm np}\limits_{#1}}
\newcommand{\npsconv}[1][]{\mathop{\star}\limits^{\rm np}\limits_{#1}}
\newcommand{\SSi}{\mathrm{SS}}

\newcommand{\fF}{\mathfrak{F}}
\newcommand{\fL}{\mathfrak{L}}
\newcommand{\fD}{\mathfrak{D}}
\newcommand{\rc}{{\R\rm{c}}}
\newcommand{\id}{\mathrm{id}}
\newcommand{\Int}{\mathrm{Int}}
\newcommand{\gcg}{{\operatorname{\gamma-cg}}}

	
	\theoremstyle{plain} 
	\newtheorem{theorem}{Theorem}[section]
	\newtheorem{corollary}[theorem]{Corollary}
	\newtheorem{proposition}[theorem]{Proposition}
	\newtheorem{lemma}[theorem]{Lemma}
	\theoremstyle{definition} 
	\newtheorem{definition}[theorem]{Definition}
	\newtheorem{example}[theorem]{Example}
	\newtheorem{remark}[theorem]{Remark}
	\newtheorem{examples}[theorem]{Examples}
	\newtheorem{question}[theorem]{Question}
	\newtheorem{Rem}[theorem]{Remark}
	\newtheorem{Notation}[theorem]{Notations}
	\newtheorem{conjecture}[theorem]{Conjecture}

\numberwithin{equation}{section}

\author{Nicolas Berkouk}
\address{Laboratory for Topology and Neuroscience, EPFL, Lausanne, Switzerland.}
\email{nicolas.berkouk@epfl.ch}

\author{Fran\c{c}ois Petit}
\address{Université Paris Cité and Université Sorbonne Paris Nord, Inserm, INRAE, Center for Research in Epidemiology and StatisticS (CRESS), F-75004 Paris, France}
\email{francois.petit@inserm.fr}



\thanks{N. B. was supported by the Swiss Innovation Agency (Innosuisse project 41665.1 IP-ICT)}
\thanks{F. P. was supported by the French Agence Nationale de la Recherche through the project reference ANR-22-CPJ1-0047-01.}

\title[Projected distances]{Projected distances for multi-parameter persistence modules}

\begin{abstract}

Relying on sheaf theory, we introduce the notions of  \emph{projected barcodes} and \emph{projected distances}  for multi-parameter persistence modules.
Projected barcodes are defined as derived pushforward of persistence modules onto $\R$. Projected distances come in two flavors: the integral sheaf metrics (ISM) and the sliced convolution distances (SCD).
We conduct a systematic study of the stability of projected barcodes and show that the fibered barcode is a particular instance of projected barcodes. 
We prove that the ISM and the SCD provide lower bounds for the convolution distance.
Furthermore, we show that the  $\gamma$-linear ISM and the $\gamma$-linear SCD which are projected distances tailored for $\gamma$-sheaves can be computed using TDA software dedicated to one-parameter persistence modules. Moreover, the time and memory complexity required to compute these two metrics are advantageous since our approach does not require computing nor storing an entire $n$-persistence module.

\end{abstract}

\maketitle

\tableofcontents

\section{Introduction}

The theory of persistence appeared in the 2000s as an algebraic framework for studying the presence of topological features in data. Its main objects of interest are $n$-parameters persistence modules that are functors from the poset category $(\R^n,\leq)$ to the category $\Mod(\cor)$ of $\cor$-vector spaces over a fixed field $\cor$, that can be compared in a meaningful way using a distance defined algebraically: the interleaving distance $d_I$. When $n = 1$, the theory is well-understood. One-parameter persistence modules are entirely determined by a discrete summary called \emph{barcode}, which can be efficiently computed. The interleaving distance can also be computed from the barcodes, thanks to the \emph{bottleneck distance} $d_B$. We refer to \cite{BL22, dey2022computational} for general introductions to $n$-parameter persistence.

One-parameter persistence is often restrictive, in particular in applications where there is no canonical choice of filtering function on the data. Furthermore, there are many situations in which being able to perform machine learning on multi-parameters persistence modules is anticipated to be fruitful \cite{lesnick2015interactive,CarriereBlumberg}. Nevertheless, the theory of $n$-parameters persistence modules when $n \geq 2$, is far more intricate. Indeed, it has been proven that there cannot exist, in a precise sense, an analogue of barcodes in this situation \cite{CZ07}, and that the interleaving distance is NP-hard to compute \cite{bjerkevik2019computing}. In this paper, relying on microlocal sheaf theory, we introduce new invariants and distances for multi-parameter persistence modules, prove that they enjoy several stability properties and are efficiently computable, which are essential requirements for such notions to be useful in practice.

So far, the main invariant that has been developed for multi-parameter persistent modules, efficiently implemented, and which enjoys the desired stability properties, is the fibered barcode \cite{Landi2018} (which is equivalent to the rank-invariant \cite{CZ07}). Roughly speaking, the fibered barcode of a $n$-parameters persistence module $M$ corresponds to the collection of barcodes obtained by restricting $M$ along each affine line of positive slope in $\R^n$. One can compare fibered barcodes by taking the supremum over all lines of positive slopes of the line-wise bottleneck distance (corrected by a specific coefficient) between barcodes of the restrictions of the persistence modules. This distance is called the \emph{matching distance}, usually denoted $d_M$. The matching distance between the fibered barcodes associated with two $n$-parameter persistence modules is bounded above by their interleaving distance \cite{Landi2018}, hence ensuring stability. Moreover, the fibered barcode of 2-parameters persistence modules originating from point-cloud data can be computed and visualized by the software RIVET \cite{lesnick2015interactive}. 

Nevertheless, there are several bottlenecks to the fibered barcode approach. First, it is easy to exhibit two persistence modules at matching distance zero but having arbitrarily large interleaving distance (see section \ref{s:cexamplefibered}). Second, computing and storing an entire $n$-parameters persistence module is time and memory costly. RIVET can currently only handle $2$-parameter persistence modules. Moreover, it cannot deal with persistence modules originating from sublevel sets filtrations of functions (such as images or PL maps on simplicial complexes).

In this paper, we approach the study of multi-parameter persistence modules through the lens of derived, and microlocal sheaf theory of Masaki Kashiwara and Pierre Schapira \cite{KS90} as initiated in \cite{KS18}. In this setting, building upon \cite{KS18}, the authors have proved, in a previous work \cite{BP21}, that a multi-parameter persistence module can be identified isometrically to a $\gamma$-sheaf (see equation \eqref{def:gammasheaf}). Therefore, this allows us to tackle the study of multi-parameter persistence modules via microlocal techniques without leaving apart computational considerations. 

One of the key aspects of sheaf theory is that any continuous map $f : Y \longrightarrow X$ between topological spaces induces a pair of adjoint functors $(f^{-1}, \roim{f})$ (inverse image and derived pushforward) between the associated derived categories of sheaves. In this language, the restriction of a persistence module $M$ to a line of positive slope $\mathcal{L} \subset \R^n$ is the persistence module $i_{\mathcal{L}}^{-1}M$ obtained by applying the inverse image functor of the inclusion of $\mathcal{L}$ into $\R^n$ to $M$. Therefore, the fibered barcode construction can be understood as a dimension reduction technique obtained by using inverse image functors along inclusion of one-dimensional sub-spaces. A natural question follows: what can we say about derived direct images of $n$-parameters persistence modules (or, more generally, of constructible sheaves) along projections onto one-dimensional spaces? Indeed, the pushforward operation does not have an easy description in the language of persistence modules, though it is natural in the sheaf setting. 

In this work, we provide a detailed study of the pushforward operation on sheaves and persistence modules, both from a theoretical and computational perspective. Following the same strategy of reducing the study of multi-parameter persistence modules to the study of families of one-dimensional persistence modules, we introduce the notions of $\fF$-\textit{projected barcodes} and $\fF$-\textit{integral sheaf metric} (see equations \eqref{def:ism} and \eqref{def:ismnp}). The distance between two sheaves is the supremum of the distances between the pushforwards of the sheaves by morphisms belonging to a family $\fF$. Again, in a similar spirit, we introduce sliced convolution distances (see equations \eqref{def:sliced} and \eqref{def:gammasliced}).

Implementing such an approach requires relating precisely the classical theory of persistence modules and sheaf theory. This has largely been achieved thanks to several works: \cite{ BG18, BGO19, BP21, Gui16, KS18, Mil20d}. Heuristically, this correspondence goes as follows.
\begin{center}
\begin{tabular}{c|c}
Persistence theory & Sheaf theory\\

\hline
level sets persistence modules & sheaves for the usual topology\\

\hline
        sublevel sets persistence  modules              & sheaves for the $\gamma$-topology\\

        \hline

interleaving distances & convolution distances\\
\hline
        tame             & constructible at infinity
\end{tabular}
\end{center}

On our way, we prove several results and provide examples that may be of independent interest to the TDA community. In particular, we show that increasing the number of parameters of two filtrations of topological spaces can only increase the interleaving/convolution distance between their associated persistence modules/sheaves. In particular, this partially invalidates the classical saying that "multi-parameter persistence is more robust to outliers than one-parameter persistence."

\bigskip

\subsection*{Structure of the paper}

\subsubsection*{Section 2} We review classical constructions of sheaf theory, such as integral transforms and kernel compositions. We recall the definition of the convolution distance between (derived) sheaves of $\cor$-vector spaces on a finite-dimensional real vector space, as developed by Kashiwara-Schapira \cite{KS18} and provide proof about some properties of the convolution distance that are well-known to the experts but do not appear anywhere (to the best of our knowledge) in the literature. We recall the lemma, proved by Petit and Schapira, stating that the pushforward by a $C$-Lipschitz map is again $C$-Lipschitz for the convolution distance \cite{PS20}. We end this section by exposing the notion of constructible sheaves up to infinity, recently defined by Schapira in \cite{schapira2021constructible}.

\subsubsection*{Section 3} We review the notion of $\gamma$-sheaves, and recall the precise relationship between this type of sheaves and persistence modules \cite{BP21}. We then strengthen one of our previous results, asserting that the interleaving distance between persistence modules equals the convolution distance between their associated $\gamma$-sheaves. Next, we recall the notion of graded-barcodes for constructible sheaves on $\R$ and how we can compute the convolution distance between two such sheaves from their graded-barcode, thanks to the derived isometry theorem \cite{BG18}. We end the section by providing a purely sheaf-theoretic formulation of the fibered barcode.

 \subsubsection*{Section 4} This section is devoted to the study of linear dimensionality reduction of sheaves on a finite-dimensional real vector space $\V$ through pushforwards along linear forms. Using projective duality, we prove that a sheaf $F$ on $\V$ is zero if and only if its pushforwards with compact support by all linear forms on $\V$ are all zero (Proposition \ref{prop:vanishproj}). We then restrict our attention to pushforwards of $\gamma$-sheaves, which corresponds in our setting to persistence modules. We recall the definition of Kashiwara-Schapira of the sublevel sets persistence sheaf $\PH(f)$, associated with a continuous map valued in $\V$. Importantly, we prove in Lemma \ref{lem:comsublevel} that under a positivity assumption on the linear form $u$ and mild hypothesis on the continuous function $f : S \longrightarrow \V$, one has 
 \begin{equation*}\label{eq:funda}
 \roim{u} \PH (f) \simeq \PH(u\circ f).
 \end{equation*}
 We emphasize that the right-hand side is nothing but the sublevel sets persistence of the real-valued function $u\circ f$, which can be computed with already existing software packages dedicated to TDA. Then, we provide a counter-example to the above isomorphism when the positivity assumption is not met (Proposition \ref{p:cexpositivity}). This counter-example shows that post-composing a sublevel set filtration with a linear map is not, in general, a stable operation. The identification of the necessity of the positivity hypothesis is made transparent, again, thanks to the sheaf formalism.
 
 Finally, we carefully study the pushforward of $\gamma$-sheaves along linear forms close to the boundary of the polar cone. This allows us to prove an unexpected result (Corollary \ref{cor:sensitivitymultipers}), unknown to the best of our knowledge, stating that increasing the number of parameters in a sublevel sets filtration can only increase the interleaving distance. More precisely, given $(f_1,...,f_n) : X \to \R^n$ and $(g_1,...,g_n) : Y \to \R^n$ continuous maps from compact good topological spaces to $\R^n$, one has for every $1 \leq i \leq n$:
 
\begin{equation*}
    \dist_\R(\PH(f_i),\PH(g_i)) \leq \dist_{\R^n}(\PH(f),\PH(g)).
\end{equation*}

This result raises several questions regarding the behavior of multi-parameter persistence modules with respects to outlier.

 \subsubsection*{Section 5} In this section, we elaborate on our study of the pushforward operation and introduce in Definition \ref{def:projectedbarcode} the notion of $\fF$-projected barcodes, associated to a family $\fF$ of subanalytic functions up to infinity from $\V$ to $\R$. We motivate the introduction of this concept by studying the classic example of the two persistence modules having the same fibered barcode but being at a strictly positive interleaving distance. We show that these two modules can be distinguished through their pushforward via a linear form onto $\R$. We then study several fundamental continuity properties of the linear and $\gamma$-linear projected barcodes (Proposition \ref{P:nullitytest} and \ref{P:gammalinearnullity}). Finally, we prove in Proposition \ref{P:directinverse} that the fibered barcode can be expressed as a projected barcode. 
 
\subsubsection*{Section 6} We develop the theory of $\fF$-Integral Sheaf Metrics ($\fF$-ISM), which are well-behaved distances between $\fF$-projected barcodes, obtained by taking the supremum over each function in $f \in \fF$ of the pushforward (possibly with proper support) by $f$ of two sheaves. When all functions in $\fF$ are $1$-Lipschitz, we prove in Proposition \ref{prop:lowerbound} that the $\fF$-ISM provides lower bounds for the convolution distance. We then give two detailed examples of $\fF$-ISM: the distance kernel ISM and the linear ISM. Finally, we introduce the sliced convolution distances, obtained by integrating the $p$-th power of the distance between the pushforwards of two sheaves over the unit dual sphere.

\subsubsection*{Section 7} We apply our previous results to the case of $n$-parameters persistence modules (seen as $\gamma$-sheaves over $\R^n$). In particular, we show that the $\gamma$-linear ISM can be obtained by optimizing an almost everywhere differentiable functional whose values and gradient can be evaluated using only one-parameter persistence software packages. We end the section by showing some concrete computations of ISM for multi-parameter persistence modules.\\

\noindent \textbf{Acknowledgement} The authors would like to thank the anonymous referee for his comments
and suggestions which greatly improved the paper.

\section{Sheaves}
This section introduces the necessary background on sheaf theory, convolution distance, and its links with persistence. The main reference for general results on sheaves is \cite{KS90}. The convolution distance for sheaves has been introduced in \cite{KS18} and generalized in \cite{PS20}. A useful notion for our purpose is the one of constructible sheaves up to infinity, which was introduced recently by Schapira in \cite{schapira2021constructible}.

Recall that a topological space is good if it is Hausdorff, locally compact, countable at infinity and of finite flabby dimension. 

Let $\cor$ be a field and let $X$ be a good topological space. We denote by $\Der(\cor_X)$ the derived category of sheaves of $\cor$-vector space on $X$, by $\Derb(\cor_X)$ its bounded counterpart, that is, the  full subcategory of $\Der(\cor_X)$ whose objects are the $F \in \Der(\cor_X)$ such that there exists $n \in \N$ such that for every $k\in \Z$ with $|k|\geq n$, $\Hn^k(F)=0$. We write $\Derb_\comp(\cor_X)$ for the full subcategory of $\Derb(\cor_X)$  spanned by the objects with compact support

In this text, we will freely make use of techniques from micro-local sheaf theory for which, we refer the reader to \cite{KS90}. Nonetheless, micro-local techniques will mostly appears in proofs and not in the results themselves. Hence a reader only interested in the results may ignore them. Let $M$ be a smooth manifold.

\begin{itemize}

\item We denote by $\ori_M$ its orientation sheaf and by $\omA$ its dualizing sheaf. Recall that $\omA = \ori_M[\dim M]$. We will also need the duality functors
\begin{align*}
    &&\dual^\prime_M(\cdot) := \rhom[\cor_M](\cdot,\cor_M), && \dual_M(\cdot) := \rhom[\cor_M](\cdot,\omA).
\end{align*}
\item If $Z$ is a locally closed subset of $M$, we denote by $\cor_Z$ the sheaf associated to the locally closed subset $Z$.

\item We write $T^\ast M$ for the cotangent bundle of $M$ and set $\dot{T}^\ast M= T^\ast M \setminus 0_M$, with $0_M$ the zero section of $T^\ast M$.

\item For $F \in \Derb(\cor_M)$, we denote by $\SSi(F)$ the \textit{micro-support} of $F$. It is a closed conical co-isotropic subset of $T^\ast M$. We refer the reader to \cite[Chapter V]{KS90} for a detailed presentation of this notion.

\item Following \cite[\S 6.1]{KS90}, let $V$ be a subset of $T^\ast M$. We define the full subcategory $\Derb_V(\cor_M)$ of $\Derb(\cor_M)$ by setting:
\begin{equation*}
	\Ob(\Derb_V(\cor_M))= \{F \in \Ob(\Derb(\cor_M)) \mid \; \SSi(F) \subset V\}
\end{equation*}

\item Let $\Omega = T^\ast M \setminus V$, we set 
\begin{equation*}
	\Derb(\cor_M; \Omega)=\Der(\cor_M)/\Ob(\Derb_V(\cor_M))
\end{equation*}
for the localization of $\Der(\cor_M)$ with respects to the null system generated by the object of $\Derb_V(\cor_M)$. The category $\Derb(\cor_M; \Omega)$ is a triangulated category.

\item We will encounter the technical notion of \textit{cohomologically constructible} sheaf for which we refer the reader to \cite[\S 3.4]{KS90}.
\end{itemize}

\subsection{Composition of kernels and integral transforms}
	
In this section, we set up a few notations and present an associativity criterion for non-proper composition of kernels.
	
Given topological spaces $X_i$ $(i=1, \; 2, \;3)$, we write $X_{ij}$ for $X_i \times X_j$, $X_{123}$ for $X_1 \times X_2 \times X_3$, $p_i \colon X_{ij} \to X_i$ and $p_{ij} \colon X_{123} \to X_{ij}$ for the projections. 
One defines the composition of kernels for $K_{ij} \in \Derb(\cor_{X_{ij}})$ as
\begin{align*}		
K_{12} \conv[2] K_{23}:= \reim{p_{13}} (\opb{p_{12}} K_{12} \otimes \opb{p_{23}} K_{23}),\\
K_{12} \npconv[2] K_{23}:= \roim{p_{13}} (\opb{p_{12}} K_{12} \otimes \opb{p_{23}} K_{23}).
\end{align*}

To a sheaf $K \in \Derb(\cor_{X_{12}})$, one associate the following functor
\begin{equation*}
	\Phi_K \colon \Derb(\cor_{X_1})\to \Derb(\cor_{X_2}), \; F \mapsto \reim{p_2}( K \otimes \opb{p_1}F). 
\end{equation*}

\begin{example}\label{ex:projduality}
	Let $\V$ be a real vector space of dimension $n+1$. We set $\dot{\V}:=\V \setminus \{0\}$ and  $\R^\times$ for the multiplicative group $\R \setminus \{0\}$. We 
	consider  $\mathbb{P}^n:=\dot{\V}/ \R^\times$ the projective space of dimension $n$. 
	The dual projective space $\mathbb{P}^{\ast n}$ is defined similarly with $\V$ replaced by $\V^\ast := \Hom[\cor](\V,\cor)$. We consider the subset 
	\begin{equation}\label{eq:projdualkernel}
		A=\{(x,y) \in \mathbb{P}^n \times \mathbb{P}^{\ast n} \mid \, \langle x, y\rangle=0 \}
	\end{equation}
and the sheaf $\cor_A$. The integral transform associated to the kernel $\cor_A$
\begin{equation*}
	\Phi_{\cor_A}\colon \Derb(\mathbb{P}^n) \to \Derb(\mathbb{P}^{\ast n}), \; F \mapsto \cor_A \circ F
\end{equation*}
induces an equivalence of categories
\begin{equation*}
	\widetilde{\Phi}_{\cor_A}\colon \Derb(\mathbb{P}^n, \dot{T}^\ast \mathbb{P}^n ) \to \Derb(\mathbb{P}^{\ast n},\dot{T}^\ast \mathbb{P}^{\ast n})
\end{equation*}
This is a consequence of \cite[Thm~7.2.1]{KS90} and we refer the reader to \cite{Gao17} for a detailed study.
\end{example}

	The proper composition of kernel $ - \, \conv[2] \, -$ is associative. This is not the case of the non-proper one $ - \, \npconv[2] \, -$. Nonetheless, we have the following result.
	
	\begin{theorem}[{\cite[Thm~2.1.8]{PS20}}]\label{th:assocnp}
		Let $X_i$ ($i=1,2,3,4$), be four $C^\infty$-manifolds and let $K_i\in\Derb(\cor_{X_{i,i+1}})$, ($i=1,2,3$). 
		Assume that
		$K_1$ is cohomologically constructible,  $q_1$ is proper on $\Supp(K_1)$ and 
		$\SSi(K_1)\cap (T^*_{X_1}X_1\times T^*X_2)\subset T^*_{X_{12}}X_{12}$. Then
		\begin{equation*}
			K_1\npconv[2](K_2\npconv[3]K_3)\simeq (K_1\npconv[2]K_2)\npconv[3] K_3.
		\end{equation*}
	\end{theorem}

\subsection{Convolution distance}

 It is possible to equip the derived categories of sheaves on a good metric space $(X,d_X)$ with a pseudo-metric \cite{PS20}. This pseudo-metric generalizes the convolution distance of \cite{KS18} from normed finite dimensional real vector spaces to good metric spaces. Hence, we will also refer to this extension as the convolution distance. The definition of a good metric space and the construction of this pseudo-metric on $\Derb(\cor_X)$ are involved and we do not need them explicitly. Hence, we do not recall them here and only review the definition of the convolution distance in the special case of sheaves on a normed finite dimensional real vector space. Nonetheless, we will state and prove some of the results at the level of generality of sheaves on a good metric spaces.

We consider a finite dimensional real vector space $\V$ endowed with a norm $\|\cdot\|$. We equip $\V$ with the topology induced by the norm $\|\cdot\|$. Following \cite{KS18}, we briefly present the convolution distance. We introduce the following notations: 

\begin{equation*}
	s : \V \times \V \to \V, ~~~s(x,y) = x + y
\end{equation*}
\begin{equation*}
	p_i : \V \times \V \to \V ~~(i=1,2) ~~~p_1(x,y) = x,~p_2(x,y) = y.
\end{equation*} 

The convolution bifunctor $\star \colon \Derb(\cor_\V)\times \Derb(\cor_\V) \to \Derb(\cor_\V)$ and the non-proper convolution bifunctor $\npsconv \colon \Derb(\cor_\V)\times \Derb(\cor_\V) \to \Derb(\cor_\V)$ are defined as follows. For $F, \; G \in \Derb(\cor_\V)$, we set
	\begin{align*}
		F \star G := \reim{s} (F \boxtimes G),\\
		F \npsconv F :=\roim{s}(F \boxtimes G).
	\end{align*}
	We consider the morphism 
	\begin{equation*}\label{mor:u}
		u: \V \times \V \to \V , \quad (x,y) \mapsto x-y.
	\end{equation*}
	We will need the following elementary formula relating non-proper convolution and non-proper composition.
	
	\begin{lemma}\label{lem:convtocomp} Let $F, \; G \in \Derb(\cor_\V)$. Then
		\begin{enumerate}[(i)]
			\item $(\opb{u}F) \npconv G \simeq F \npsconv G$
			\item $\opb{u}F \npconv \opb{u}G \simeq \opb{u}(F \npsconv G)$
		\end{enumerate}	
	\end{lemma}
	\begin{proof}
		We will only prove $(i)$, the proof of $(ii)$ being similar. We define the maps $(u,p_2) : \V \times \V \to \V \times \V$ (resp. $(s,p_2)$) by $(u,p_2) (x,y) = (u(x,y),y)$ (resp. $(s,p_2) (x,y) = (s(x,y),y)$). These two continuous maps are invertible, inverse of each other.
		\begin{align*}
			\opb{u}F \npconv G &\simeq \roim{p_1}(\opb{u}F \otimes \opb{p_2}G)\\
			&\simeq \roim{p_1}\left ( \opb{(u,p_2)} \opb{p_1} F \otimes \opb{(u,p_2)}\opb{p_2}G \right)\\
			& \simeq \roim{p_1} \opb{(u,p_2)} (F \boxtimes G)\\
			& \simeq \roim{p_1} \roim{(s,p_2)} (F \boxtimes G)\\
			& \simeq F \npsconv G.
		\end{align*}	
	\end{proof}

	For $r \geq 0$, we set $B_r = \{x \in \V \mid  \| x \| \leq r \}$, and $\Int{(B_r)} = \{x \in \V \mid  \| x \| < r \}$. For all $r \in \R$, we define the following sheaf:
	
	\[K_r := \begin{cases} \cor_{B_r} \textnormal{~if~} r \geq 0 \\  \cor_{\Int{(B_{-r})}}[\dim(\V)] ~\textnormal{otherwise}\end{cases}. \]

	The following proposition is proved in \cite{KS18}.
	
	\begin{proposition}
		
		\label{P:propertiesofconvolution} Let $r, r'\in \R$ and $F \in \Derb(\cor_\V)$. There are functorial isomorphisms 
		
		\begin{equation*}
			( K_{r}  \star K_{r'}) \star F \simeq K_{r + r'} \star F ~~~ and ~~~ K_0 \star F\simeq F.
		\end{equation*}
	\end{proposition}

	If $r \geq r' \geq 0 $, there is a canonical morphism 
	$\chi_{r, r'} \colon K_{r}\to K_{r'}$ in $\Derb(\cor_\V)$. It induces a canonical morphism $\chi_{r, r'} \star F \colon K_{r} \star F \to  K_{r'} \star F $. 
	In particular when $r' = 0$, we get
	\begin{equation}
		\chi_{r,0} \star F \colon K_{r} \star F \to  F.
	\end{equation}
	
	Following \cite{KS18}, we recall the notion of $c$-isomorphic sheaves. 
	
	\begin{definition}
		Let $F,G \in \Derb(\cor_\V)$ and let $c \geq  0$. The sheaves $F$ and $G$ are $c$-isomorphic if there are morphisms $f : K_r  \star F \to G$ and $g : K_r \star G \to F$ such that the  diagrams
		
		\begin{align*}
			\xymatrix{ K_{2r} \star F \ar[rr]^-{{K_{2r}} \star f} \ar@/_2pc/[rrrr]_{{\chi_{2r,0}} \star F} && K_{r} \star G \ar[rr]^-{g} &&   F
			},\\
			\xymatrix{ K_{2r} \star G \ar[rr]^-{{K_{2r}} \star g} \ar@/_2pc/[rrrr]_{{\chi_{2r,0}} \star G} && K_{r} \star F \ar[rr]^-{f} &&   G
			}.
		\end{align*}
		are commutative. The pair of morphisms $f,g$ is called a pair of $r$-isomorphisms.	
	\end{definition}

\begin{definition}For $F,G \in \Derb(\cor_\V) $, their \textit{convolution distance} is 
	\begin{equation*}
		\dist_\V(F,G) := \inf(\{r \geq 0 \mid F ~\text{and}~G~\text{are}~r -\text{isomorphic}\} \cup \{ \infty \}).
	\end{equation*}

\end{definition} 
	
It is proved in \cite{KS18} that the convolution distance is, indeed, a pseudo-extended metric, that is, it satisfies the triangular inequality. The following properties of the convolution distance are well-known to the specialists. We include them with proofs for the convenience of the reader, as we do not know any references for them.	

\begin{lemma} \label{lem:transaffine} Let $(\V, \norm{\cdot})$ be a real finite dimensional normed vector space. Let $F, G \in \Derb(\cor_\V)$. 
	\begin{enumerate}[(i)]
		\item Let $v \in \V$, and $\tau_v \colon \V \to \V$, $x \mapsto x-v$. Then
		\begin{equation*}
			\dist_\V(\oim{\tau_v}F,\oim{\tau_v}G)=\dist_\V(F,G).
		\end{equation*}
		
		\item let $\lambda \in \R$ and $h_\lambda \colon \V \to \V$, $x \mapsto \lambda  x$. Then
		\begin{equation*}
			\dist_\V(\oim{h_\lambda}F,\oim{h_\lambda}G)= \vert \lambda \vert \, \dist_\V(F,G).
		\end{equation*}
	\end{enumerate}
The functors $\oim{\tau_v}$ and $\oim{h_\lambda}$ are exact, hence we do not need to derive them.
\end{lemma}

\begin{proof}
	\noindent (i) For $F \in \Derb(\cor_X)$, there are the following natural isomorphisms

	\begin{align*}
		K_{r} \star \oim{\tau_v} F &\simeq \reim{s}\opb{(\id \times \tau_{-v})} (K_{r} \boxtimes F)\\ 
		& \simeq \oim{\tau_{v}} (K_{r} \star F).
	\end{align*}
	Let $F, G \in \Derb(\cor_\V)$ and assume that they are $c$-isomorphic. Let $f \colon K_{r} \star F \to G$ and $g \colon K_{r} \star G \to F$ be a pair of $c$-isomorphisms. Then, we have the morphisms
	\begin{align*}
       f^\prime \colon K_{r} \star \oim{\tau_v} F \simeq & \oim{\tau_{v}} (K_{r} \star F) \stackrel{\oim{\tau_v}f}{\longrightarrow} \oim{\tau_v} G,\\
       g^\prime \colon K_{r} \star \oim{\tau_v} G \simeq & \oim{\tau_{v}} (K_{r} \star G) \stackrel{\oim{\tau_v}f}{\longrightarrow} \oim{\tau_v} F
	\end{align*}
It is straightforward to verify that $(f^\prime,g^\prime)$ is a pair of $c$-isomorphisms. One shows similarly that if $\oim{\tau_v}F$ and $\oim{\tau_v}G$ are $c$-isomorphic, then $F$ and $G$ are $c$-ismorphic. \\

\noindent (ii) We first observe that 
\begin{align*}
		\oim{h_{\lambda}}(K_{r} \star  F) & \simeq \opb{h_{1/\lambda}}K_{r} \star \oim{h_{\lambda}} F\\
		& \simeq  K_{ \vert \lambda \vert r} \star \oim{h_\lambda} F.\\
	\end{align*}
Then, we notice that $\oim{h_\lambda} (\chi_{r,0}) \simeq \chi_{\lambda r,0}$ which concludes the proof.
\end{proof}	
	
We now recall several key inequalities for the convolution distance.

\begin{proposition}[{\cite[Prop.~2.6]{KS18}}]\label{prop:dualdis}
\begin{enumerate}[(i)]
\item Let $F, G \in \Derb(\cor_\V)$, then 
\begin{equation*}
     \dist_\V(\dual_\V(F),\dual_\V(G)) \leq \dist_\V(F,G).
\end{equation*}
\item Assume that $\dist_\V(F_i,G_i) \leq a_i$ ($i=1,2$) then one has  
\begin{equation*}
\dist_\V(F_1 \star F_2,G_1 \star G_2) \leq a_1 + a_2.
\end{equation*}
\end{enumerate}
\end{proposition}

\begin{lemma}\label{lem:dualiso}
For $F, G \in \Derb_{\rc}(\cor_\V)$,
$\dist_\V(\dual_\V(F),\dual_\V(G))=\dist_\V(F,G)$.
\end{lemma}

\begin{proof}
    By Proposition \ref{prop:dualdis}, we have that $\dist_\V(\dual_\V(F),\dual_\V(G)) \leq \dist_\V(F,G)$. As $F$ and $G$ are constructible, it follows from \cite[Prop.~3.4.3]{KS90}, that $\dual_\V(\dual_\V(F))\simeq F$ and similarly for $G$. It follows that
    \begin{align*}
        \dist_\V(F,G) &= \dist_\V(\dual_\V(\dual_\V((F)),\dual_\V(\dual_\V((G)))\\
        &\leq \dist_\V(\dual_\V(F),\dual_\V(G)).
    \end{align*}
    Hence, $\dist_\V(F,G)=\dist_\V(\dual_\V(F),\dual_\V(G))$.
\end{proof}

Let X be a topological space, $(\V, \norm{\cdot})$ be a normed finite dimensional real vector space and $f_1, \; f_2 \colon X \to \V$ be two continuous maps. We set
\begin{equation*}
	\norm{f_1-f_2}_\infty=\sup_{x \in X} \, \norm{f_1(x)-f_2(x)}.
\end{equation*}

\begin{theorem}[{\cite[Thm.~2.7]{KS18}}]\label{thm:stability}
Let $X$ be a good topological space and $f_1, \;f_2 :  X \rightarrow \V$ be two continuous maps. Then for any $F \in \Derb(\cor_X)$, we have
\begin{align*}
    \dist_\V(\reim{f_1} F,\reim{f_2} F) \leq \norm{f_1-f_2}_\infty && \dist_\V(\roim{f_1} F,\roim{f_2} F) \leq \norm{f_1-f_2}_\infty.
\end{align*}
\end{theorem}

\begin{theorem}[{\cite[Corollary 2.5.9.]{PS20}}]\label{thm:lipstability}
Let $(X,d_X)$ and $(Y,d_Y)$ be two good metric spaces and $f\colon X \rightarrow Y$ a $K$-Lipschitz map. Then, for every $F, G \in \Derb(\cor_X)$
\begin{equation*}
    \dist_X(\reim{f}F,\reim{f}G) \leq K \, \dist_Y(F,G),
\end{equation*}
If moreover $X$ and $Y$ are finite dimensional vector spaces and $d_X$ and $d_Y$ are euclidean distances, then one also has
\begin{equation*}
    \dist_Y(\roim{f}F,\roim{f}G) \leq K \, \dist_X(F,G).
\end{equation*}
\end{theorem}

\subsection{Constructible sheaves up to infinity}

We review in this section the notion of sheaves constructible up to infinity, which was recently introduced by Schapira in \cite{schapira2021constructible}. The geometric setting is the following. Let $(\V, \norm{\cdot})$ be a real  vector space of dimension $n$ endowed with a norm. We denote by $\Proj$ the projectivisation of $\V$. That is, we consider the $n$ dimensional projective space $\Proj^n(\V \oplus \R)$. We also write $\Proj^\ast$ for the projectivization of $\V^\ast$.  We have the open immersion: 
\begin{align*}
	j \colon \V \to \Proj, \quad x \mapsto [x,1]. 
\end{align*}
In particular, $j$ is an open embedding of real analytic manifolds, whose image is relatively compact.

\begin{proposition}[{\cite[Lemma 2.7]{schapira2021constructible}}]
  Let $F \in \Derb_{\rc }(\cor_\V)$. The following are equivalent.
  
  \begin{enumerate}[(i)]
      \item The micro-support $\SSi(F)$ is subanalytic in $T^\ast \Proj$,
      
      \item The micro-support $\SSi(F)$ is contained in a locally closed $\R^+$-conic subanalytic isotropic subset of $T^\ast \Proj$,
      
      \item $j_! F \in \Derb_{\rc }(\cor_\Proj)$,
      
      \item $\text{R} j_\ast F \in \Derb_{\rc }(\cor_\Proj)$.
  \end{enumerate}
\end{proposition}

\begin{definition}
    If $F \in \Derb_{\rc }(\cor_\V)$ satisfies any (hence all) of the conditions above, we say that $F$ is constructible up to infinity. We denote by $\Derb_{\rc }(\cor_{\V_\infty})$ the full triangulated subcategory of $\Derb_{\rc }(\cor_\V)$ whose objects are sheaves constructible up to infinity.
\end{definition}

The following is a special case of \cite[Definition 2.5]{schapira2021constructible}.

\begin{definition}
Let $\V$ and $\mathbb{W}$ be two finite dimensional real vector spaces, and $\Proj$ and $\Proj'$ their respective projectivization. A map $f : \V \to \W$ is subanalytic up to infinity if its graph $\Gamma_f \subset \V \times \mathbb{W}$ is subanalytic in $\Proj \times \Proj'$. 
\end{definition}

\begin{Notation}
We write 
\begin{enumerate}[(i)]
\item $\mathcal{A}(\V_\infty)$ for the set of morphism of analytic manifolds from $\V$ to $\R$ which are subanalytic up to infinity.
\item  $\mathcal{SC}(\V_\infty)$ for the set of continuous maps from $\V$ to $\R$ which are subanalytic up to infinity.
\end{enumerate}
\end{Notation}

Let $\V$ and $\mathbb{W}$ be two finite dimensional real vector spaces, $\Proj$ and $\Proj'$ their respective projectivization and $f : \V \to \W$ an element of $\mathcal{SC}(\V_\infty)$.

\begin{proposition}\label{prop:opconstructinf}
    Let $F \in \Derb_{\rc}(\cor_{\V_\infty})$ and $G \in \Derb_{\rc}(\cor_{\mathbb{W}_\infty})$. Then, the sheaves $f^{-1} G, f^! G, \text{R} f_\ast F$, and $\text{R}f_! F$ are constructible up to infinity.
\end{proposition}
\begin{proof}
This follows immediately from \cite[Lem.~1.1]{schapira2021constructible}  and \cite[Cor.~2.1']{schapira2021constructible}
\end{proof}

\begin{corollary}[{\cite[Corollary 2.13]{schapira2021constructible}}]\label{cor:roimtoreim}
With the same notations, one has $\text{R}f_\ast F \simeq \dual_\mathbb{W} \text{R}f_! \dual_\V F$ and $f^! G \simeq \dual_\V f^{-1} \dual_\mathbb{W} G$.
\end{corollary}

\section{\texorpdfstring{$\gamma$}{gamma}-sheaves} \label{sec:gammasheaves}

The interplay between sheaves on a real vector space and persistence theory necessitates the use of a topology on a vector space introduced by Kashiwara and Schapira \cite{KS90}, called the $\gamma$-\emph{topology}. In this section, we first recall the basic definitions associated to the $\gamma$-topology. There is a notion of interleaving distance between sheaves on the $\gamma$-topology, which the authors proved to coincide under some properness hypothesis with the convolution distance \cite[Corollary 5.9]{BP21}. Here, we strengthen this result by removing the properness assumption. We then recall some results specific to the dimension one case, where $\gamma$-sheaves have a graded-barcode \cite{BG18}. We end the section with the notion of fibered barcode which was introduced in \cite{Cer13} for persistence modules, that we translate in our $\gamma$-sheaf setting.

Let $\V$ be a finite dimensional real vector space. We write $a \colon x \mapsto -x$ for the antipodal map. If $A$ is a subset of $\V$, we write $A^a$ for the image of $A$ by the antipodal map.
A subset $\gamma$ of the vector space $\V$ is a cone if
\begin{enumerate}[(i)]
\item $0 \in \gamma$,
\item $\R_{>0} \cdot \gamma \subset \gamma$.
\end{enumerate}
A convex cone $\gamma$ is proper if $\gamma^a \cap \gamma =\lbrace 0 \rbrace$. The polar cone $\gamma^\circ$ of a cone $\gamma \subset \V$ is the cone of $\V^\ast$
\begin{equation*}
\gamma^\circ=\lbrace \xi \in \V^\ast \, \mid \, \forall v \in \gamma, \langle \xi, v \rangle \geq 0 \rbrace.
\end{equation*} 
From now on, we assume that $\gamma$ is a
\begin{equation}\label{hyp:cone}
\textnormal{\textit{closed proper convex cone with non-empty interior.}}
\end{equation}

We say that a subset $A$ of $\V$ is \textit{$\gamma$-invariant} if $A=A+ \gamma$. The set of $\gamma$-invariant open subset of $\V$ is a topology on $\V$ called the $\gamma$-topology. We will generically designate topology of this type by \textit{cone topology}. We denote by $\V_\gamma$ the vector space $\V$ endowed with the $\gamma$-topology. We write $\phi_\gamma \colon \V \to \V_{\gamma}$ for the continuous map whose underlying function is the identity.

We set, following \cite{KS18},
\begin{equation} \label{def:gammasheaf}
\Derb_{\gamma^{\circ,a}}(\cor_\V)=\lbrace F \in \Derb(\cor_\V) \, \vert \, \SSi(F) \subset \V \times \gamma^{\circ,a} \rbrace,
\end{equation}

We call an object of $\Derb_{\gamma^{\circ,a}}(\cor_\V)$ a \emph{$\gamma$-sheaf}. This terminology is motivated by the following result.

\begin{theorem}[{\cite[Thm. 1.5]{KS18}}]
Let $\gamma$ be a proper closed convex cone in $\V$. The functor $\roim{\phi_\gamma} \colon \Derb_{\gamma^{\circ,a}}(\cor_\V) \to \Derb(\cor_{\V_\gamma})$ is an equivalence of triangulated categories with quasi-inverse $\opb{\phi_\gamma}$.
\end{theorem}

The canonical map $\cor_{\gamma^a} \to \cor_{\lbrace 0 \rbrace}$ induces a morphism
\begin{equation}\label{mor:gammacar}
F \npsconv \cor_{\gamma^a} \to F.
\end{equation}

\begin{proposition}[{\cite[Prop.~3.9]{GS14}}]\label{prop:gammaloc}
Let $F \in \Derb(\cor_\V)$. Then $F \in \Derb_{\gamma^{\circ,a}}(\cor_\V)$ if and only if the morphism \eqref{mor:gammacar} is an isomorphism.
\end{proposition}

The following lemma is closely related.

\begin{lemma}[{\cite[Prop. 3.5.4]{KS90}}]
The endofunctors $ \cor_{\gamma^a} \npsconv  (\cdot)$ of $\Derb(\cor_\V)$ factors through $\Derb_{\gamma^{\circ,a}}(\cor_\V)$ and defines a projector $\Derb(\cor_\V) \to \Derb_{\gamma^{\circ,a}}(\cor_\V)$.
\end{lemma}

This functor is called the ``gammaification'' functor.\\

Let $v\in \V$. Recall that
\begin{equation*}
\tau_v \colon \V \to \V, \; x \mapsto x - v.
\end{equation*}
For $v,w \in \V$ such that $w + \gamma \subset v + \gamma$, we can use proposition \ref{prop:gammaloc} to construct a morphism of functors from $\Derb_{\gamma^{\circ ,a}}(\cor_{\V})$ to $\Derb_{\gamma^{\circ ,a}}(\cor_{\V})$

\begin{equation}\label{mor:bismoothmu}
\chi^\mu_{v,w} \colon \oim{\tau_v} \to \oim{\tau_w}.
\end{equation}

We refer the reader to \cite[\S 4.1.2]{BP21} for details.

\begin{definition}
Let $F$, $G \in \Derb_{\gamma^{\circ,a}}(\cor_\V)$, and $v \in \gamma^a$. We say that $F$ and $G$ are $v$-interleaved if there exists $f : \oim{\tau_v}F \to  G$ and $g: \oim{\tau_v}G \to F$ such that the following diagram commutes.

\begin{equation*}\xymatrix{
\oim{\tau_{2v}}F \ar[rrd]\ar@/^0.75cm/[rrr]^{\chi_{2v,0}(F)} \ar[r]^-{\sim}
& \oim{\tau_{v}} \oim{\tau_{v}} F \ar[r]^{\oim{\tau_v} f} 
& \oim{\tau_v} G  \ar[rd] \ar[r]^{g} 
& F \\
\oim{\tau_{2v}}G  \ar[urr]\ar@/_0.75cm/[rrr]_{\chi_{2v,0}(G)} \ar[r]^-{\sim}
& \oim{\tau_{v}}  \oim{\tau_{v}} G \ar[r]_-{\oim{\tau_v} g}
& \oim{\tau_v} F  \ar[ur] \ar[r]_{f\;\;\,} 
& G.
}
\end{equation*}
\end{definition}

\begin{definition}
The interleaving distance between $F$ and $G$ with respect to $v\in \gamma^a$ is 
\begin{equation*}
d_I^v(F,G) := \inf(\{r \geq 0 \mid F ~\text{and}~G~\text{are}~r \cdot v-\text{interleaved}\} \cup \{ \infty \}).
\end{equation*}
\end{definition}

The interleaving distance was introduced in \cite{Chazal2009}. We refer the reader to  \cite{Silva2018} for details.

\subsection{Distances comparison}

In this subsection, we compare the interleaving distance on $\Derb_{\gamma^{\circ,\,a}}(\cor_{\V})$ with the convolution distance on $\Derb_{\gamma^{\circ,\,a}}(\cor_{\V})$. We sharpen Proposition 5.8 and Corollary 5.9 of \cite{BP21} by removing the $\gamma$-properness assumption. The architecture of the proof is the same as in \cite{BP21}. The $\gamma$-properness hypothesis is removed thanks to Theorem \ref{th:assocnp}.

Here, $\V$ is endowed  with a closed proper convex cone $\gamma$ with non-empty interior. Let $v \in \Int(\gamma^a)$ and consider the set 
\begin{equation*}
B_v := (v+\gamma) \cap (-v+\gamma^a).
\end{equation*}
The set $B_v$ is a symmetric closed bounded convex subset of $\V$ such that $0 \in \Int B_v$. It follows that the gauge 
\begin{equation}\label{eq:normcool}
g_{B_v}(x)=\inf \lbrace \lambda > 0 \mid   x \in \lambda B_v \rbrace
\end{equation}
is a norm, the unit ball of which is $B_v$. We denote this norm by $\| \cdot \|_v$. We assume that $\V$ is equipped with this norm. Recall the map

\begin{equation*}
 u: \V \times \V \to \V , \quad (x,y) \mapsto x-y.
\end{equation*}
Consider 
\begin{equation*}
	\Theta=\{(x,y) \in \V \times \V \mid x-y \in \gamma^a \}=\opb{u}(\gamma^a).
\end{equation*}
and notice that
\begin{equation*}
	\Delta_r = \{(x,y) \in \V \times \V \mid  \|x-y \|_v \leq r\}=u^{-1}(B_r)
\end{equation*}
where $B_r$ is the closed ball of center $0$ and radius $r$ in $\V$ for the norm $\| \cdot \|_v$.

The following formulas follow from  Lemma \ref{lem:convtocomp}.
\begin{align}
    \cor_\Theta \npconv F \simeq \cor_{\gamma^a} \npsconv F. \label{eq:gamtheta}\\
	\cor_{\Delta_r} \npconv \cor_\Theta \simeq \cor_{\Delta_r + \Theta}. \label{eq:lintheta}
\end{align}

\begin{lemma}\label{lem:assoboule} Let $F \in \Derb(\cor_\V)$. Then
 \begin{equation*}
 (\cor_{\Delta_r} \npconv \cor_{\Theta}) \npconv F \simeq \cor_{\Delta_r} \npconv (\cor_{\Theta} \npconv F).
 \end{equation*}
\end{lemma}

\begin{proof}
	 Let $\pi \colon T^\ast \V \to \V$ be the cotangent bundle of $\V$. We denote by $u_d$ the dual of the tangent morphism of $u$. It fits in the following commutative diagram
	 \begin{equation*}
	 \xymatrix{
	 T^\ast (\V \times \V) \ar[rd]^{\pi} & (\V \times \V) \times_{\V} \times T^\ast \V \ar[l]_-{u_d} \ar[r]^-{u_\pi} \ar[d]^{\pi}& T^\ast \V \ar[d]^-{\pi} \\
	 & \V \times \V \ar[r]^-{u} & \V. 
     }	
	 \end{equation*}
	 Since, we have the isomorphism $\cor_{\Delta_r} \simeq \opb{u} \cor_{B_r}$, it follows from \cite[Prop.~5.4.5]{KS90} that 
	 $\SSi(\cor_{\Delta_r})=u_d \, \opb{u_\pi}(\SSi(B_r))$. A direct computation shows that $u_d \, \opb{u_\pi}(T^\ast \V) \cap T^\ast_{\V}  \V \times T^\ast \V \subset 
	 T^\ast_{\V \times \V} (\V \times \V)$. Hence $\SSi(\cor_{\Delta_r})  \cap T^\ast_{\V}  \V \times T^\ast \V  \subset T^\ast_{\V \times \V} (\V \times \V)$. Since $\cor_{\Delta_r}$ is constructible and the properness assumption is clearly satisfied, we apply Theorem \ref{th:assocnp} and get the desired isomorphism.
\end{proof}
Finally, there is also the following isomorphism
\begin{equation}\label{eq:isoboule}
	\cor_{r \cdot v + \gamma^a} \simeq \cor_{B_r + \gamma^a}.
\end{equation}

We now state and prove the sharpen version of \cite[Prop.~5.8]{BP21}.

\begin{theorem}\label{thm:isomBP21}
	Let $v \in \Int(\gamma^a)$, $r \in \R_{\geq 0}$ and $F, G \in \Derb_{\gamma^{\circ,\,a}}(\cor_{\V})$. Then $F$ and $G$ are $r \cdot v$-interleaved if and only if they are $r$-isomorphic.
\end{theorem}

\begin{proof}
	Let $F, G \in \Derb_{\gamma^{\circ,\,a}}(\cor_{\V})$. Assume they are $r \cdot v$-interleaved. We set $w=r \cdot v$. Hence, we have the maps 
	\begin{align*}
	\alpha \colon \tau_{w \ast} F \to G && \beta \colon \tau_{w \ast} G \to  F
	\end{align*}
	such that the below diagrams commute
	\begin{align*}
	\xymatrix{\tau_{2w \ast} F \ar[r]^-{\oim{\tau_w} \alpha} \ar@/_2pc/[rr]_{{\chi^\mu_{0,2w}}(F)} & \oim{\tau_{w}} G \ar[r]^-{\oim{\tau_{w}} \beta} &  F 
		&&
		\tau_{2w \ast} G \ar[r]^-{\oim{\tau_w} \beta} \ar@/_2pc/[rr]_{{\chi^\mu_{2w,0}}(G)} & \oim{\tau_{w}} F \ar[r]^-{\oim{\tau_{w}} \alpha} &  G. 
		\\
	}
	\end{align*}
	Using \cite[Lem.~4.3]{BG18}, we obtain 
	\begin{align*}
	\xymatrix{\cor_{2w+\gamma^a} \npsconv F \ar[rr]^-{\cor_{2w+\gamma^a} \npsconv \alpha} \ar@/_2pc/[rrrr]_{{\chi_{2w,0}} \npsconv F} && \cor_{w+\gamma^a} \npsconv G \ar[rr]^-{\cor_{w+\gamma^a} \npsconv \beta} &&   F.
	}
	\end{align*}
Moreover for every $r \geq 0$, we have the following isomorphisms
\begin{align*}
	\cor_{r v + \gamma^a} \npsconv F &\simeq \cor_{B_r + \gamma^a} \npsconv F \quad \quad \textnormal{by Equation \eqref{eq:isoboule}}\\
	& \simeq \cor_{\Delta_r + \Theta} \npconv F \quad \quad \textnormal{by Lemma \ref{lem:convtocomp} (i)}\\
	&\simeq (\cor_{\Delta_r} \npconv \cor_{\Theta}) \npconv F \quad \quad \textnormal{by Equation \eqref{eq:lintheta}}\\
	&\simeq \cor_{\Delta_r} \npconv (\cor_{\Theta} \npconv F) \quad \quad \textnormal{by Lemma \ref{lem:assoboule}}\\
	& \simeq \cor_{\Delta_r} \npconv (\cor_{\gamma^a} \npsconv F) \quad \quad \textnormal{by Equation \eqref{eq:gamtheta}}\\
	&\simeq \cor_{\Delta_r} \npconv F \quad \quad \textnormal{by Proposition \ref{prop:gammaloc}}\\
	&\simeq \cor_{B_r} \npsconv F \quad \quad \textnormal{by Lemma \ref{lem:convtocomp} (i)}\\
	&\simeq \cor_{B_r} \star F \quad \quad \textnormal{(compacity of $B_r$)}.
\end{align*}
Hence, we obtain the commutative diagram

	\begin{align*}
	\xymatrix{\cor_{B_{2c}} \star F \ar[rr]^-{\cor_{B_c} \star \alpha} \ar@/_2pc/[rrrr]_{{\rho_{0,2c}} \star F} && \cor_{B_c} \star G \ar[rr]^-{\beta} &&   F
	}.
	\end{align*}
	Similarly we obtain the following commutative diagram
	\begin{align*}
	\xymatrix{\cor_{B_{2r}} \star G \ar[rr]^-{\cor_{B_{r}} \star \beta} \ar@/_2pc/[rrrr]_{{\rho_{0,2r}} \star G} && \cor_{B_r} \star F \ar[rr]^-{\alpha} &&   G
	}.
	\end{align*}Hence, $F$ and $G$ are $r$-isomorphic. 
	
	A similar argument proves that if $F$ and $G$ are $r$-isomorphic then they are $r\cdot v$-interleaved.
\end{proof}

\begin{corollary}\label{cor:gamconv}
	Let $v \in \Int(\gamma^a)$, $F, G \in \Derb_{\gamma^{\circ,\,a}}(\cor_{\V})$. Then
	\begin{equation*}
	\dist_\V^v(F,G)=d^v_{I^\mu}(F,G)
	\end{equation*}
	where $\dist_\V^v$ is the convolution distance associated with the norm $\|\cdot \|_v$ defined in  equation \eqref{eq:normcool}.
\end{corollary}

\begin{proposition}\label{P:cohomologyinterleaving}
Let $v \in \Int(\gamma^a)$, $F, G \in \Derb_{\gamma^{\circ,\,a}}(\cor_{\V})$ and $c\in \R_{\geq 0}$. Assume that $F$ and $G$ are $c\cdot v$-interleaved. Then for all $j \in \Z$, $\Hn^j(F)$ and $\Hn^j(G)$ are $c\cdot v$-interleaved. 
\end{proposition}

\begin{proof}
For all $w \in \Int{\gamma^a}$, the functor $\tau_{w \ast} : \Mod(\cor_\V) \to \Mod(\cor_\V)$ is exact. Therefore, applying the functor $\Hn^j$ to a $c\cdot v $-interleaving diagram between $F$ and $G$ produces a  $c\cdot v $-interleaving diagram between $\Hn^j(F)$ and $\Hn^j(G)$.
\end{proof}

\begin{corollary}
For $v \in \Int{\gamma^a}$ and $F, G \in \Derb_{\gamma^{\circ,\,a}}(\cor_{\V})$, one has: \begin{enumerate}
    \item $\max_j d^v_{I^\mu}(\Hn^j(F),\Hn^j(G)) \leq  d^v_{I^\mu}(F,G)$,
    
    \item $\max_j \dist_\V^v(\Hn^j(F),\Hn^j(G)) \leq  \dist_\V^v(F,G)$.
\end{enumerate}
\end{corollary}

\begin{proof}
The first inequality is a direct consequence of Proposition \ref{P:cohomologyinterleaving}. The second one is a consequence of the first one, together with Corollary \ref{cor:gamconv}.
\end{proof}

\subsection{The dimension one case}

When $\V$ is a one-dimensional real vector space, the category $\Derb_{\rc}(\cor_\V)$ enjoys a structure theorem which ensures the existence of a graded-barcode for its objects, and allow to derive explicit computations for the convolution distance. This was studied in detail in \cite{KS18, BG18}. These results extends the key theorem of Crawley-Boevey \cite{CB14} to constructible sheaves on the real line. In this section, we recall the main results that will be useful in the following of the article.

\begin{theorem}[{\cite[Thm 1.17]{KS18}}]\label{T:Decomposition}

Let $F \in \Mod_{\rc}(\cor_\R)$, then there exists a unique locally finite multi-set of intervals of $\R$ noted $\mathbb{B}(F)$ such that $$F \simeq \bigoplus_{I \in \mathbb{B}(F)} \cor_{I}.$$ Moreover, this decomposition is unique up to isomorphism.

\end{theorem}

Since for $I$ and $J$ some intervals of $\R$, one has $\text{Ext}^j(\cor_I,\cor_J) \simeq 0$ for all $j > 1$, Theorem \ref{T:Decomposition} has the following useful corollary. 

\begin{corollary}\label{C:structure}
Let $F\in \Derb_{\rc}(\cor_\R)$. Then there exists an isomorphism in $\Derb_{\rc}(\cor_\R)$: $$F \simeq \bigoplus_{j\in\Z} \Hn^j(F)[-j],$$
where $\Hn^j(F)$ is seen as a complex concentrated in degree $0$.
\end{corollary}

\begin{definition}[{\cite[Definiton 2.13]{BG18}}]
\label{D:GradedBarcodes}
Let $F\in \Derb_{\rc}(\cor_\R)$, we define its \emph{graded-barcode} $\B(F)$ as the collection $(\mathbb{B}^j(F))_{j\in \Z}$ where $\B^j(F) := \B(\text{H}^j(F))$. Furthermore, to indicate that an interval $I\subset \R$ appears in degree $j\in \Z$ in the graded-barcode of $F$, we write $I^j \in \B(F)$. The element $I^j$ is called a graded-interval.
    
\end{definition}

We recall here the construction of the category $\textbf{Barcode}$ of \cite{BG18}, which is an explicit skeleton of $\Derb_{\rc}(\cor_\R)$. 
Let $\text{Inter}(\R)$ be the set of intervals of $\R$ and $p_1$,  $p_2$ be the projections on the first two coordinates of  $\text{Inter}(\R) \times \Z \times \Z_{\geq 0}$. Let $\mathbb{B}$ be a subset of $ \text{Inter}(\R) \times \Z \times \Z_{\geq 0}$. Then  $\mathbb{B}$ is said to be 
\begin{itemize}

\item \emph{locally finite} if $p_1(\mathbb{B})\cap K$ is finite for all compact subsets $K \subset \R$; 

\item \emph{bounded} if $p_2(\mathbb{B}) \subset \Z$ is bounded;

\item \emph{well-defined} if the fibers of the projection $(p_1,p_2)$ have cardinality at most $1$.
\end{itemize}
In a triple $(I,j,n) \in\mathbb{B}$, the first integer stands for the degree in which the interval $I$ is seen and the second non-negative integer $n$ stands for its multiplicity.
\begin{definition}[{\cite[Definition 6.11]{BG18}}]\label{D:CatBarcode}
The category $\textbf{Barcode}$ has as objects the elements of the set
\begin{align*}
\Ob(\textbf{Barcode}) &=\{\mathbb{B} \subset \text{Int}(\R) \times \Z\times \Z_{\geq 0}, \mathbb{B} ~\text{is bounded, locally} \\
&\hspace{0.7cm}\text{finite and well-defined} \}.
\end{align*}
\noindent For any $\mathbb{B}$ and $\mathbb{B}' \in \textbf{Barcode}$,  the set of their morphisms is
\begin{equation*}
 \Hom[\textbf{Barcode}](\mathbb{B},\mathbb{B}') =  \prod_{\substack{(I,j,n)\in \mathbb{B} \\ (I',j',n') \in \mathbb{B}'}} \Hom[\Derb_{\rc}(\cor_\R)] \left (\cor_I^n[-j] , \cor_{I'}^{n'}[-j'] \right ).
\end{equation*}
\end{definition}
\noindent We define the composition in \textbf{Barcode} so that the mapping : $$\iota : \Ob(\textbf{Barcode})\ni \mathbb{B}  \mapsto \bigoplus_{(I,j,n)\in \mathbb{B}} \cor_I^n[-j] \in \Ob(\Derb_{\rc}(\cor_\R))  $$ becomes a fully faithful functor : 
\begin{equation*}
\iota :  \textbf{Barcode} \longrightarrow \Derb_{\rc}(\cor_\R).
\end{equation*}
\noindent Note that this is possible only because the objects of $\textbf{Barcode}$ are locally finite (hence products and co-products coincide). Theorems \ref{T:Decomposition} and \ref{C:structure} assert that $\iota$ is essentially surjective, therefore is an equivalence. We also deduce from these theorems that $\textbf{Barcode}$ is a skeletal category: it satisfies for any $\mathbb{B}_1,\mathbb{B}_2 \in \textbf{Barcode}$, $$\mathbb{B}_1 \simeq \mathbb{B}_2 \text{~if and only if~} \mathbb{B}_1 = \mathbb{B}_2. $$

Since $\iota$ is an equivalence, let us denote by $\mathbb{B}$ a quasi-inverse of $\iota$. In \cite{BG18}, the authors define a matching distance between the objects of \textbf{Barcode} called the bottleneck distance, and denoted $d_B$. They prove the following isometry theorem, where $\R$ is endowed with the usual absolute value norm $| \cdot |$, and we denote by $\dist_\R$ the associated convolution distance on $\Derb(\cor_{\R})$.

\begin{theorem}[{\cite[Thm.~5.10]{BG18}}]
The functor $$\mathbb{B} : (\Derb_{\rc}(\cor_\R), \text{dist}_\R) \longrightarrow  (\textbf{Barcode}, d_B)$$is an isometric equivalence.
\end{theorem}

In this article, we will mostly be interested in $\gamma$-sheaves, since they are the sheaf theoretic analogue of persistence modules in a precise sense \cite{BP21}. Therefore, we unwrap the derived isometry theorem of \cite{BG18} in the simpler setting of $\gamma$-sheaves that will be useful for us in the following of the article. For the rest of this section, we set $\gamma = (-\infty, 0]$. Intervals appearing in the graded-barcodes of $\gamma$-sheaves are of the form $[a,b)$, with $a,b \in \R \cup \{\pm \infty\}$. Note that all results translate readily for the cone $\gamma' = [0, +\infty)$.

We recall that, given two (multi-)sets $X$ and $Y$, a partial matching between $X$ and $Y$ is the data $(\sigma, \mathcal{X}, \mathcal{Y})$ of two subsets $\mathcal{X} \subset X$ and $\mathcal{Y}\subset Y$, together with a bijection $\sigma : \mathcal{X} \to \mathcal{Y}$. In this situation, we use the notation $(\sigma, \mathcal{X}, \mathcal{Y}) :  X  \to Y$. 

\begin{definition}\label{D:matching}
    Let $F,G \in \Derb_{\rc,\gamma^{\circ,\,a}}(\cor_{\R})$, and $\varepsilon \geq 0$. An \emph{$\varepsilon$-matching} between $\B(F)$ and $\B(G)$ is the data of a collection of partial matchings $(\sigma^j, \mathcal{X}^j, \mathcal{Y}^j): \B^j(F) \to  \B^j(G)$, satisfying the following, for all $j \in \Z$:
    
    \begin{enumerate}
        \item for all $I \in \mathcal{X}^j$ such that $I = [a,b)$ with $a$ and $b$ in $\R \cup \{\pm \infty\}$, then $\sigma^j(I) = [a',b')$ with $a'$ and $b'$ in $\R \cup \{\pm \infty\}$, and\footnote{We set $|+\infty - +\infty| = |-\infty - (- \infty)| = 0$, and for all $x\in \R$, $|\pm \infty - x| = +\infty$.} $|a-a'| \leq \varepsilon $ and $|b-b'| \leq \varepsilon$, 
        
        \item for all $I =[a,b) \in \B^j(F) \backslash \mathcal{X}^j \cup \B^j(G) \backslash \mathcal{Y}^j $, then $|a-b| \leq 2 \varepsilon$.
    \end{enumerate}

\end{definition}

\begin{definition}
Let $F,G \in \Derb_{\rc,\gamma^{\circ,\,a}}(\cor_{\R})$, their bottleneck distance is defined by: $$d_B(F,G) = \inf \{\varepsilon\geq 0 \mid ~\mathbb{B}(F) ~ \text{and} ~ \mathbb{B}(G) ~\text{are $\varepsilon$-matched} \}. $$
\end{definition}

\begin{remark}
The matchings between graded barcodes of $\gamma$-sheaves are defined in the same way as between barcodes of persistence modules. Therefore, one can compute the bottleneck distance between barcodes of $\gamma$-sheaves using already existing software \cite{ tauzin2021giottotda,gudhi:urm}. It is nevertheless far from being true when removing the $\gamma$ assumption on $F$ and $G$.
\end{remark}

\begin{theorem}[{\cite[Thm 5.10]{BG18}}]\label{T:Isometry}
Let $F,G \in \Derb_{\rc , \gamma^{\circ,\,a}}(\cor_{\R})$, then the following are equivalent, for all $\varepsilon \geq 0$:

\begin{enumerate}[(i)]
    \item $F$ and $G$ are $\varepsilon$-isomorphic,
    \item there exists a $\varepsilon$-matching between $\B(F)$ and $\B(G)$,
    \item for all $j \in \Z$, $\Hn^j(F)$ and $\Hn^j(G)$ are $\varepsilon$-isormophic.
\end{enumerate}
\end{theorem}

\begin{corollary}\label{C:gradedDistance}
Let $F,G \in \Derb_{\rc , \gamma^{\circ,\,a}}(\cor_{\R})$, then 
\begin{align*}
\dist_\R (F,G) &= \max_{j\in \Z} \dist_\R(\Hn^j(F),\Hn^j(G)) \\
            &= \max_{j\in \Z} d_B (\mathbb{B}(\Hn^j(F)), \mathbb{B}(\Hn^j(G))).
    \end{align*}  
\end{corollary}

\begin{remark}
    Note that Corollary \ref{C:gradedDistance} is false if one only assumes $F,G \in \Derb_{\rc}(\cor_{\R})$.
\end{remark}

One consequence of the derived isometry theorem is the following closedness property of the convolution distance, which we recall in full generality here.

\begin{theorem}[{\cite[Thm 6.3]{BG18}}] \label{thm:closedis} \ Let $F,G \in \Derb_{\rc}(\cor_\R)$, then the following are equivalent, for all $\varepsilon \geq 0$:

\begin{enumerate}[(i)]
    \item $\dist_\R(F,G) \leq \varepsilon$;
    \item $F$ and $G$ are $\varepsilon$-isomorphic.
\end{enumerate}
In particular, $\dist_\R(F,0) = 0$ iff $F \simeq 0$.
\end{theorem}

\subsection{Fibered Barcode}

 \label{S:fiberedbarcode} One of the challenges of multi-parameter persistence is to provide a meaningful notion of distance between persistence modules which can be computed in a reasonable time complexity. Indeed, it has been shown that the usual interleaving distance between persistence modules is NP-hard to compute in the multi-parameter case \cite{bjerkevik2019computing}. To overcome this issue, the authors of \cite{Cer13} introduced the matching distance between multi-parameter persistence modules, which is by now popular in the TDA community, thanks to the software RIVET \cite{lesnick2015interactive}. We review this notion below, but formulate it in the language of $\gamma$-sheaves.

In this section, we assume that $\V = \R^n$ and consider the cone $\gamma = (-\infty, 0 ] ^n$. Let  $v = (1,...,1)^t \in \Int(\gamma^a)$ and $\|\cdot \|_v$ is defined as in equation (\ref{eq:normcool}). Then, 
\begin{equation*}
\|(x_1,...,x_n)\|_v = \max \{|x_1|, ... , |x_n|\}.
\end{equation*}

Let us define $\Lambda := \{h \in \Int(\gamma^a) , \|h\|_v = 1\}$. Given $h \in \Lambda$,  we denote by $\mathcal{L}_h$ the one dimensional subspace of $\R^n$ spanned by $h$. Recall that $\gamma$ induces on $\R^n$ the standard product order $\leq$ given by:
\begin{equation*}
x \leq y \iff x + \gamma \subseteq y + \gamma \iff x_i \leq y_i ~~~\text{for all }i.
\end{equation*}

The poset $(\mathcal{L}_h,\leq)$ is isomorphic to $(\R,\leq)$ via the isometry
\begin{equation}\label{isom:RL}
\iota_{\mathcal{L}_h} \colon \R \to \mathcal{L}_h, \; t \mapsto t \cdot h.
\end{equation}
Therefore, it has the least upper bound property.  Following \cite{lesnick2015interactive}, we define the push function $p_{\mathcal{L}_h} : \V \to \mathcal{L}_h$ by
\begin{equation*}
p_{\mathcal{L}_h}(x) = \inf \{y \in \mathcal{L}_h\mid x \leq y  \} = \inf \{y \in \mathcal{L}_h\mid x \in y + \gamma \}.
\end{equation*}

\begin{figure}
    \centering
    \includegraphics[width=0.5\textwidth]{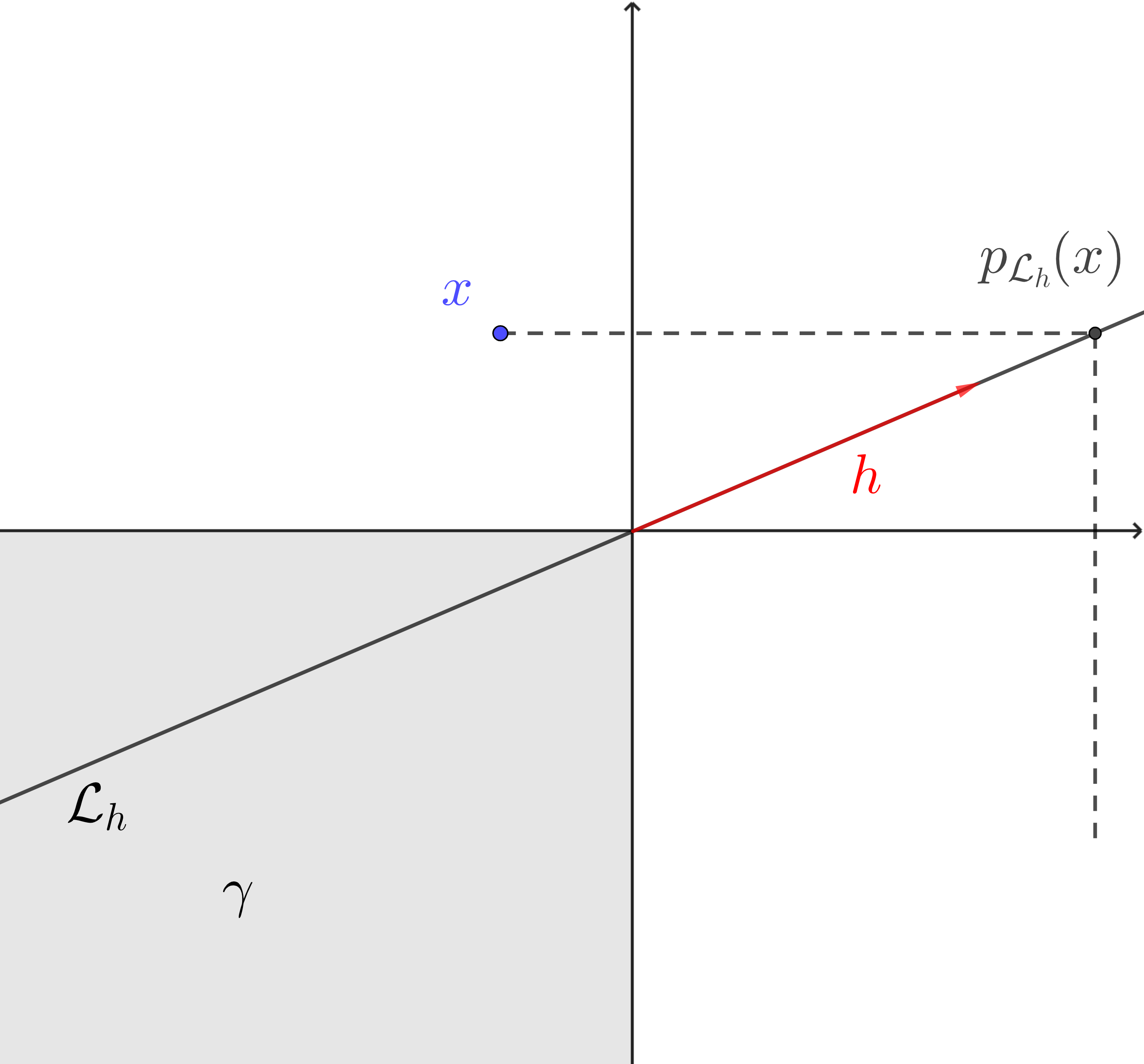}
    \caption{Illustration of the push function $p_{\mathcal{L}_h}$}
    \label{fig:push}
\end{figure}

One can prove that we have the following formula, for $x=( x_1,...,x_n) \in \V$ and $h = (h_1,...,h_n) \in \Lambda$:

\begin{equation}\label{eq:formulaP}
    p_{\mathcal{L}_h}(x) = \left ( \max_i \frac{x_i}{h_i} \right ) \cdot h.
\end{equation}

In particular, one has for $y \in \mathcal{L}_h$:

\begin{equation}\label{eq:fiberP}
    p_{\mathcal{L}_h}^{-1}(\{y\}) = y + \partial \gamma.
\end{equation}
\begin{lemma}
For $h = (h_1,...,h_n)\in \Lambda$, the map $p_{\mathcal{L}_h}$ is $\frac{1}{\min_i h_i}$-Lipschitz with respect to $\| \cdot \|_v$, and this coefficient is optimal. 
\end{lemma}

\begin{proof}
Let $x, x^\prime \in \V$, since $h_i >0$ for all $i$, one has:

$$p_{\mathcal{L}_h}(x) = \max_i \frac{x_i}{h_i} \cdot h \in \mathcal{L}_h. $$

Therefore, 

\begin{align*}
   \|p_{\mathcal{L}_h}(x) - p_{\mathcal{L}_h}(x^\prime) \|_v &=  \mid \max_i \frac{x_i}{h_i} - \max_i \frac{x^\prime_i}{h_i} \mid \\
   &\leq \max_i |\frac{x_i}{h_i} - \frac{x^\prime_i}{h_i} \mid  \\ 
   &\leq \frac{1}{\min_i h_i} \| x - x^\prime \|_v.
\end{align*}

The above inequality becomes an equality when $i_0\in \{1,...n\}$ is such that $h_{i_0} = \min_i h_i$, $x$ is the $i_0$-th element of the canonical basis of $\R^n$, and $x^\prime = 0$.
\end{proof}

\begin{definition}[\cite{Cer13}] Let $F \in \Derb_{\rc, \gamma^{\circ, a}}(\cor_\V)$, its fibered barcode is the collection of graded-barcodes $(\mathbb{B}(\opb{i_{\mathcal{L}_h}} \opb{\tau_{c}}F) )_{(h,c) \in \Lambda \times \V}$.

\end{definition}

\begin{definition}[\cite{Cer13}] Let $F,G \in \Derb_{\rc, \gamma^{\circ, a}}(\cor_\V)$, their matching distance is the possibly infinite quantity:

\begin{equation*}
d_m (F,G) = \sup_{(h,c) \in \Lambda \times \V} \left (\min_i h_i \right ) \cdot  d_B (\mathbb{B}(i_{\mathcal{L}_h}^{-1} \opb{\tau_{c}} F), \mathbb{B}(i_{\mathcal{L}_h}^{-1} \opb{\tau_{c}} G)).
\end{equation*}
\end{definition}

Recall that we denote by $\dist_\V^v$ the convolution distance associated with the norm $\|\cdot\|_v$ on $\Derb(\cor_\V)$. Then one has the following stability result:

\begin{proposition}[\cite{Cer13, Landi2018}]
    Let $F,G \in \Derb_{\rc, \gamma^{\circ, a}}(\cor_\V)$. Then,
    \begin{equation*}
    d_m(F,G) \leq \dist_\V^v(F,G)
    \end{equation*}
\end{proposition}

\noindent Therefore, the matching distance is stable with respect to the convolution distance. Under some finiteness assumptions on $F,G$. The matching distance can be computed in polynomial time (in the number of generators and relations of a finite free presentation of $F$ and $G$ seen as persistence modules). In particular, the software RIVET \cite{lesnick2015interactive} allows for an efficient computation of the matching distance when $\V$ is of dimension $2$, and gives an interactive visualization of the fibered barcode.

\section{Linear dimensionality reduction}

In this section, we study direct images of sheaves on a finite dimensional vector space by linear forms. We examine in detail the case of $\gamma$-sheaves and sublevel sets persistence modules. We obtain several vanishing results and a formula allowing to compute explicitly the pushforward by a linear form of a sublevel sets multi-parameter persistence module.

Let $(\V, \norm{\cdot})$ be a real  vector space of dimension $n$ endowed with a norm. Recall that we denote by $\Proj$ the projectivisation of $\V$. That is, we consider the $n$ dimensional projective space $\Proj^n(\V \oplus \R)$. We also write $\Proj^\ast$ for the projectivization of $\V^\ast$.  We have the open immersion 
\begin{align}\label{map:openim}
	j \colon \V \to \Proj, \quad x \mapsto [x,1].
\end{align}

\subsection{Linear direct image of sheaves}

We need the following straightforward generalization of the stability inequality. This is a version with support of the usual stability inequality, as proved in \cite{KS18}.

\begin{lemma}\label{lem:stabsupport}
	Let $X$ be a good topological space and $(\V, \norm{\cdot})$ be a real finite dimensional normed vector space. Let $F \in \Derb(\cor_X)$, $S$ be a closed subset containing $\Supp(F)$, and $f_1, f_2 : X \to \V$ be continuous maps. Then,
	\begin{align*}
	    \dist_\V(\reim{f_1}F,\reim{f_2}F) \leq \norm{f_1|_S-f_2|_S}_\infty && \dist_\V(\roim{f_1}F,\roim{f_2}F) \leq \norm{f_1|_S-f_2|_S}_\infty.
	\end{align*}
\end{lemma}

\begin{proof}
We only prove the first inequality, the proof of the second one being similar.
Remark that $F \simeq \eim{i_S} \opb{i_S}F$. Hence,
\begin{equation*}
\dist_\V(\reim{f_1}F,\reim{f_2}F)=\dist_\V(\reim{f_1}\eim{i_S} \opb{i_S}F,\reim{f_2}\eim{i_S} \opb{i_S}F).
\end{equation*}
It follows from the stability theorem that
\begin{equation*}
	\dist_\V(\reim{f_1}\eim{i_S} \opb{i_S}F,\reim{f_2}\eim{i_S} \opb{i_S}F) \leq \norm{f_1|_S-f_2|_S}_\infty.
\end{equation*}
Thus,
\begin{equation*}
	\dist_\V(\reim{f_1}F,\reim{f_2}F) \leq \norm{f_1|_S-f_2|_S}_\infty.
\end{equation*}
\end{proof}

Let $S$ be a subset of $\V$. We set $\norm{S}=\sup_{x \in S} \norm{x}$. We endow $\V^\ast$ with the operator norm $\opnorm{u} := \sup_{\|x\| = 1} \|u(x)\|$.

\begin{lemma}\label{lem:continuitylin_projbar}
Let $F \in  \Derb_{\comp}(\cor_{\V})$, $u,v \in \V^\ast$ and set $S=\Supp(F)$. Then,
\begin{equation*}
    \dist_\R(\reim{u}F,\reim{v}F) \leq  \norm{S} \opnorm{u-v}.
\end{equation*}
\end{lemma}

\begin{proof}
It follows from Lemma \ref{lem:stabsupport} that
\begin{align*}
    \dist_\R(\reim{u}F,\reim{v}F) & \leq \norm{u|_S-v|_S}_\infty 
\end{align*}
and for every $x \in S$
\begin{align*}
    |u(x)-v(x)| & \leq \opnorm{u-v} \norm{x}_\infty\\
                               & \leq  \opnorm{u-v} \norm{S} .
\end{align*}
Thus, $\norm{\tilde{u}|_S-\tilde{v}|_S}_\infty \leq \opnorm{u-v} \norm{S}$.
\end{proof}

Recall the following integral transform
\begin{equation*}
	\Phi_{\cor_A}\colon \Derb(\mathbb{P}^n) \to \Derb(\mathbb{P}^{\ast n}), \; F \mapsto \cor_A \circ F
\end{equation*}
where the set $A$ is defined by equation \eqref{eq:projdualkernel}.

We wish to extract information regarding $F$ from the data of pushforwards of $F$ by linear forms. For that purpose, we will use projective duality as it allows to construct a sheaf on $\Proj^\ast$, the projectivization of $\V^\ast$, via $\Phi_{\cor_A}$, the stalks of which are the $(\reim{u}F)_t$ where $u$ is a linear form  and $t$ an element of $\R$.
\begin{lemma}\label{lem:stalkprojduality}
Let $F \in \Derb(\cor_\V)$ and let $[u;t] \in \Proj^\ast$. Then 
\begin{align*}
\Phi_{\cor_A}(\reim{j}F)_{[u;t]}& \simeq (\reim{u}F)_t.
\end{align*}
\end{lemma}

\begin{proof}
Let $[u;t] \in \Proj^\ast$ and set 
\begin{equation*}
	B_{[u;t]}= \{ [x,y] \in \Proj \mid \, u(x)=-ty \}.
\end{equation*}
There is the following commutative diagram
\begin{equation*}
\xymatrix{
\V \ar[r]^-{j} \ar[d]_-{\id_\V \times [u;t]} & \Proj \ar[d]^{\id_\Proj \times [u;t]}
\\
\V \times \Proj^\ast \ar[r]^-{j \times \id} \ar[d]_-{q_1}& \Proj \times \Proj^\ast \ar[d]^{p_1}\\
\V \ar[r]^-{j}& \Proj,
}
\end{equation*}

\noindent with $q_1$ and $p_1$ the first projections, and $[u;t] : \V \to \Proj^\ast$ the constant map with value $[u;t]$.
\begin{align*}
\Phi_{\cor_A}(\reim{j}F)_{[u;t]}& \simeq \rsect_c(\Proj \times [u;t]; (\cor_A \otimes \opb{p_1}\reim{j}F)|_{\Proj \times [u;t]})\\
& \simeq \rsect_c(\Proj \times [u;t]; (\cor_{B_{[u;t]}\times [u;t]} \otimes \opb{(p_1  (\id_\Proj \times [u;t]))}\reim{j}F)\\
& \simeq \rsect_c(\Proj; \cor_{B_{[u;t]}} \otimes \reim{j}F)\\
& \simeq \rsect_c(\Proj ; \reim{j }(\opb{j }\cor_{B_{[u;t]}} \otimes F))\\
& \simeq \rsect_c(\Proj ; \reim{j}(\cor_{\opb{u}(t)} \otimes F))\\
& \simeq \rsect_c(\V; \cor_{\opb{u}(t)} \otimes F)\\
&\simeq \rsect_c(\opb{u}(t);F|_{\opb{u}(t)})\\
&\simeq (\reim{u}F)_t
\end{align*}
\end{proof}

\begin{proposition}\label{prop:vanishproj}
		Let $F \in \Derb(\cor_\V)$ such that for all $u \in \V^\ast$, $\reim{u}F \simeq 0$. Then, $F \simeq 0$.
	\end{proposition}
	
\begin{proof}
Let $[u;t] \in \Proj^\ast$. Then, it follows from Lemma \ref{lem:stalkprojduality} that
\begin{align*}
\Phi_{\cor_A}(\reim{j}F)_{[u;t]}& \simeq (\reim{u}F)_t.
\end{align*}
Hence, $\Phi_{\cor_A}(\reim{j}F) \simeq 0$ and 
\begin{equation*}
\widetilde{\Phi}_{\cor_A}\colon \Derb(\Proj, \dot{T}^\ast \Proj ) \to \Derb(\Proj^\ast,\dot{T}^\ast \Proj^\ast)
\end{equation*}
is an equivalence of categories. This implies that $\reim{j}F \simeq 0$ in $\Derb(\Proj, \dot{T}^\ast \Proj)$. Thus $\SSi(F) \subset T^\ast_\V \V$. It follows from \cite[Prop.~8.4.1]{KS90} that $\reim{j}F$ is a local system on $\Proj$. Now, let $[x,0] \in \Proj$. Then 
\begin{equation*}
(\reim{j}F)_{[x;0]} \simeq \rsect_c(\opb{j}([x;0]),F |_{\opb{j}([x;0])}) \; \textnormal{and} \; \opb{j}([x;0])=\emptyset.
\end{equation*}
Thus, $  (\reim{j}F)_{[x;0]} \simeq 0$. This implies that $\reim{j}F \simeq 0$, which leads to $F \simeq 0$.
\end{proof}

\subsection{Linear direct image of \texorpdfstring{$\gamma$}{gamma}-sheaves}

In this subsection, we study direct images and direct images with proper supports of $\gamma$-sheaves by linear forms. 

\begin{proposition}\label{prop:oimgamma}
	Let $F \in \Derb_{\gamma^{\circ,a}}(\cor_\V)$ and $ u \in \V^\ast$.
	\begin{enumerate}[(i)]
		\item If $u \in \gamma^{\circ,a}$ (resp. $\gamma^{\circ}$), then $\reim{u} F \in \Derb_{\lambda^{\circ,a}}(\cor_\V)$ with $\lambda=\R_-$ (resp. $\lambda=\R_+$). 
		\item If $u \notin \gamma^\circ \cup \gamma^{\circ,a}$, then $\reim{u} F$ is a constant sheaf. If furthermore $F$ has a compact support, then $\reim{u} F \simeq 0$. 
	\end{enumerate}
	The same results $(i)$ and $(ii)$ hold for $\reim{u} F$ replaced with $\roim{u} F$.
\end{proposition}

\begin{proof}
\noindent (i) Let $F \in \Derb_{\gamma^{\circ,a}}(\cor_\V)$, $u \in \V^\ast$ and $\pi_{\R^\ast} \colon \R \times \R^\ast \to \R^\ast$. Then \cite[Corollary 1.17]{GS14} asserts that
\begin{equation}\label{eq:microprojlin}
    \SSi(\reim{u}F) \subset \opb{\pi_{\R^\ast}} \opb{(u^t)}(\gamma^{\circ,a}).
\end{equation}
Moreover,
\begin{equation}\label{eq:microestim}
\opb{(u^t)}(\gamma^{\circ,a})=\{\xi \in \R^\ast \mid \xi(1) u \in \gamma^{\circ,a} \}.   
\end{equation}
Assume that $u \in \gamma^{\circ, a} \setminus \{0 \}$ and let $\xi \in \opb{(u^t)}(\gamma^{\circ,a})$. Since the interior of $\gamma$ is non-empty, there exists $x_0 \in \gamma$ such that $u(x_0)<0$ and $\xi(1)u(x_0) \leq 0$. Thus $\xi(1) \geq 0$ which is equivalent to $\xi \in \lambda^{\circ, a}$ with $\lambda=\R_{-}$. Equation \eqref{eq:microprojlin} becomes
\begin{equation*}
    \SSi(\reim{u}F) \subset \R \times \lambda^{\circ, a}
\end{equation*}
which proves the claim. The case where $u \in \gamma^\circ$ is treated in a similar way.

\noindent (ii) Since $u \notin \gamma^\circ \cup \gamma^{\circ,a}$, there exist $x,y \in \gamma$ such that $u(x)<0$ and $u(y)>0$. Hence, $\opb{\pi_\R} \opb{(u^t)}(\gamma^{\circ,a})=\R \times \{0\}$. It follows from \cite[Prop.~8.4.1]{KS90} that $\reim{u}F$ (resp. $\roim{u}F)$ is a constant sheaf (possibly equal to zero) as $\V$ is contractible. If the support of $F$ is compact, then the support of $\reim{u}F$ is different from $\R$ and thus $\reim{u} F \simeq 0$.
\end{proof}

\begin{corollary}
Let $F \in \Derb_{\gamma^{\circ,a},\comp}(\cor_\V)$. Then $\rsect(\V;F) \simeq 0$.
\end{corollary}

\begin{proof}

Let $u \notin \gamma^\circ \cup \gamma^{\circ,a}$. Then, by Proposition \ref{prop:oimgamma}, $\roim{u}F \simeq 0$ and
\begin{equation*}
    \rsect(\V;F) \simeq \roim{a_\R}\roim{u} F \simeq 0
\end{equation*}
where $a_\R$ is the constant map $\V \to \{ \pt \}$.
\end{proof}

\begin{lemma}\label{lem:polarpropre}

Let $K \subset \V$ be a compact set and $u \in \Int(\gamma^\circ)$. Then $u|_{K + \gamma}$ is proper.
\end{lemma}

\begin{proof}
Since the interior of $\gamma$ is nonempty, $K$ is compact and $u$ is linear, we can assume without loss of generality that $K \subset \gamma$. Therefore, it is sufficient to prove that for any two real numbers $a<b$, the set $\Sigma := u^{-1}([a,b]) \cap \gamma$ is compact. Since $u \in \Int(\gamma^\circ)$, it follows that for every $x \in \gamma \setminus \{0\}$, $u(x)>0$, which shows that $\ker(u) \cap \gamma = \{0\}$. We set $v \in \Int(\gamma) $ such that $u(v) = 1$. Consequently, one has the direct sum decomposition $\V =  \R v \oplus \ker(u)$, and one has : \begin{equation*}
   \Sigma  = \{t \cdot v + h \mid t \in [a,b], h \in \ker(u)\} \cap \gamma.
\end{equation*}

Since $\gamma$ is closed and convex, so is $\Sigma$. Therefore, $\Sigma$ is unbounded if and only if there exists $x \in \Sigma $ and $y \in \V \setminus \{0\}$ such that for all $t \geq 0$, $x + t\cdot y \in \Sigma$. Let assume that there exists such $x$ and $y$. Then, for all $t\geq 0$, $u(x + t\cdot y) \in [a,b]$, which implies that $y \in \ker(u)$. Moreover, for all $t> 0$, 

\begin{equation*} \frac{1}{t} \cdot \left ( x + t\cdot y \right ) = \frac{x}{t} + y \in \gamma. \end{equation*} 

Since $\gamma$ is closed, we obtain making $t$ going to $+\infty$ that $y \in \gamma \cap \ker(u) = \{0\}.$ This is absurd by hypothesis, so $\Sigma$ is bounded. Since $\V$ is finite dimensional, it is compact.

\end{proof}

\begin{lemma}\label{lem:oimcone}
	Let $u \in \Int(\gamma^\circ)$ (resp. $\Int(\gamma^{\circ, a})$). Then 
	$\roim{u} \cor_\gamma \simeq \cor_{\R_+}$ (resp. $\roim{u} \cor_{\gamma} \simeq \cor_{\R_-}$).
\end{lemma}

\begin{proof}
Since $u \in \gamma^\circ$, it follows that $\gamma \subset \opb{u}(\R_+)$ and these two sets are closed. 
Hence, there is a canonical map $\opb{u} \cor_{\R_+}\simeq\cor_{\opb{u}(\R_+)} \to \cor_\gamma$. By adjunction, this provide a map $\alpha \colon \cor_{\R_+} \to \roim{u} \cor_\gamma$. We compute the stalks of this map. Since $u \in \Int(\gamma^\circ)$, it is proper on $\gamma$. It follows that for $x \in \R$ $(\roim{u} \cor_\gamma)_x \simeq \rsect_c(\opb{u}(x),\cor_{\gamma \cap \opb{u}(x)})$.
	\begin{itemize}
		\item If $x \notin \R_+$, then $\opb{u}(x) \cap \gamma = \emptyset$. Thus $\rsect_c(\opb{u}(x),\cor_{\gamma \cap \opb{u}(x)}) \simeq 0$ and $\alpha_x$ is an isomorphism.
		\item If $x \in \R_+$, then $\opb{u}(x) \cap \gamma$ is non-empty and it is a compact convex subset of $\opb{u}(x)$. This implies that $\rsect_c(\opb{u}(x),\cor_{\gamma \cap \opb{u}(x)})\simeq \cor$ and $\alpha_x$ is again an isomorphism.
	\end{itemize}
\end{proof}

 We recall the notion of characteristic cone following \cite{Klee55, Stoker40}. Let $C$ be a closed convex subset of $\V$ and let $p \in C$. Consider the set \[ C_p=\{x \in \V \mid \forall \lambda \geq 0, \lambda (x-p)+p \in C \}.\] Hence $C_p$ is the union of half-lines contained in $C$ emanating from $p$. The set $C_p$ is an affine convex cone and if $q \in C$ then $C_p=C_q + (p-q)$. The cone $C_p$ is called a characteristic cone of $C$.

\begin{lemma}\label{lem:cohomubdconv} Let $\V$ be a finite dimensional real vector space and let $C$ be a closed convex set such that a characteristic cone of $C$ is not an affine subspace  of $\V$. Then
\begin{equation*}
    \rsect_c(\V,\cor_C) \simeq 0.
\end{equation*}
\end{lemma}

\begin{proof}

We first establish the lemma when $\V=\R$ and $C=[0,+\infty[$. Using a stereographic projection from the north pole N of $\mathbb{S}^1 \subset \R^2$, we notice that $\R_+$ is homeomorphic to a half-circle minus the north pole that is to $[0,1[$. We consider the following exact triangle

\begin{equation*}
    \rsect_c([0,1];\cor_{[0,1[}) \to \rsect_c([0,1]; \cor_{[0,1]}) \to \rsect_c([0,1]; \cor_{\{1\}}) \stackrel{+1}{\to}.
\end{equation*}

Since, the morphism $\rsect_c([0,1]; \cor_{[0,1]}) \to \rsect_c([0,1]; \cor_{\{1\}})$ is an isomorphism, it follows that  $\rsect_c([0,1];\cor_{[0,1[}) \simeq 0$. Thus $\rsect_c(\R,\cor_{[0,+\infty[}) \simeq 0$.

Now, assume that $\V$ is a real vector space of dimension $n$ and $C$ is closed convex set of $\V$ such that a characteristic cone of $C$ is not an affine subspace  of $\V$. By  \cite[\S 5]{Klee55}, it follows that $C$ is homeomorphic to $[0,1]^d \times [0,+\infty[$ for some $0 \leq d \leq n-1$. Hence,
\begin{align*}
    \rsect_c(\V,\cor_C) & \simeq \rsect_c(C,\cor_C)\\
                        & \simeq \rsect_c([0,1]^d \times [0,+\infty[,\cor_{[0,1]^d \times [0,+\infty[})\\
                        & \simeq \rsect_c([0,1]^d,\cor_{[0,1]^d}) \otimes \rsect_c( [0,+\infty[,\cor_{[0,+\infty[})\\
                        & \simeq 0.
\end{align*}
\end{proof}

\begin{lemma}\label{lem:affsubcone}
Let $\V$ be a vector space. A closed proper cone $\gamma$ of $\V$ does not contain any non-trivial affine subspace of $\V$.
\end{lemma}

\begin{proof}
Let $E$ be an affine subspace of $\V$ contained in $\gamma$ and let $p \in E$ and $v \in \V$, an element of the vector space associated with $E$. The set $L=\{x \in \V \mid x=p+ \lambda v, \lambda \in \R \}$ is an affine subspace of $E$. Since $\gamma$ is a cone, for every $n \in \N^\ast$, the points $q_n=\frac{1}{n} (p+ n v)$ and $r_n=\frac{1}{n} (p- n v)$ belongs to $\gamma$. Since $\gamma$ is closed, the limits $\lim_{n \to \infty} p_n = v$ and $\lim_{n \to \infty} r_n = -v$ belongs again to $\gamma$ which is proper. Hence, $v=0$ and $L=\{p\}$. This implies that $E$ is trivial.
\end{proof}

\begin{proposition}\label{prop:boundvanish}
Let $u \notin \Int(\gamma^\circ) \cup \Int(\gamma^{\circ, \, a})$. Then $\reim{u} \cor_\gamma \simeq 0$.
\end{proposition}

\begin{proof}
Let $y \in \R$. Then
\begin{align*}
    (\reim{u} \cor_\gamma)_y &\simeq \rsect_c(\opb{u}(y); \cor_{\gamma \cap \opb{u}(y)}).
\end{align*}
We set $C=\gamma \cap \opb{u}(y)$. Then, either $C=\emptyset$ and $\rsect_c(\opb{u}(y); \cor_{C})\simeq 0$ or $C \neq \emptyset$ and it is a closed convex set. Let $x_0 \in C$. Since $u \in \partial \gamma^\circ \cup \partial \gamma^{\circ, \, a}$, it follows that $(\ker(u) \cap \gamma) \setminus (0) \neq \emptyset$. As $\gamma$ is a cone this implies that $(\ker(u) \cap \gamma)$ contains a half-line $h$. Since $x_0 \in \gamma$ and $\gamma$ is a convex cone then $x_0+h \subset \gamma$. Moreover, $h \subset \ker(u)$, thus $u(x_0+h)=\{y\}$. Hence, $x_0+h$ is contained in $C$ which implies that the characteristic cone $C_{x_0}$ contains at least a ray and is not reduced to a point. Let $z \in C$. We have the following inclusion of characteristic cone $C_z \subset \gamma_z= \gamma+z$. Moreover, it follows from Lemma \ref{lem:affsubcone} that $\gamma$ does not contain any non-trivial affine subspace of $\V$. Hence $C_z$ is not an affine subspace of $\V$. Then, applying Lemma \ref{lem:cohomubdconv}, we get that $(\reim{u} \cor_\gamma)_y \simeq \rsect_c(\V, \cor_C) \simeq 0$.
\end{proof}

\subsection{Sublevel sets persistence}

In \cite{KS18}, Masaki Kashiwara and Pierre Schapira provide a sheaf-theoretic construction of sublevel sets multi-parameter persistence. The aim of this section is to prove that the sheaf encoding the sublevel sets multi-parameter persistence of a pair $(S,f)$ where $S$ is a good compact topological space and $f \colon S \to \V$ is a continuous map, is $\gamma$-compactly generated. 

\begin{definition}\label{df:gammacompact}
A sheaf $F \in \Derb(\cor_\V)$ is $\gamma$-compactly generated if there exists $G \in \Derb_\comp(\cor_\V)$ such that $F \simeq G \npsconv \cor_{\gamma^a}$. We denote by $\Derb_{\gcg}(\cor_\V)$ the full subcategory of $\Derb(\cor_\V)$ spanned by $\gamma$-compactly generated sheaves.
\end{definition}

\begin{remark}
Any compactly supported $\gamma$-sheaves is $\gamma$-compactly generated. Also, with the same notations, one has $F \simeq G \star \cor_{\gamma^a}$.
\end{remark}

We now recall Kashiwara-Schapira's construction of the sublevel sets multi-parameter persistence module associated to a pair $(S,f)$ where $S$ is a good topological space and $f \colon S \to \V$ is a continuous map.  We denote by $\Gamma_f$ the graph of $f$ and $\gamma$ is a cone satisfying hypothesis \eqref{hyp:cone}. We set 
\begin{align*}
	\Gamma_f^\gamma&=\{(x,y) \in S \times \V \mid f(x)-y \in \gamma\}\\
	&=\Gamma_f + \gamma^a.
\end{align*}
We write $p \colon S \times \V \to \V, \; (x,v) \mapsto v$ for the projection onto $\V$ and $s \colon \V  \times \V \to \V, (v,w) \mapsto v+w$. We notice that $s \circ (p \times \id_{\V})=p \circ (\id_S \times s)$.

\begin{definition}
The sublevel sets persistent sheaf of the pair $(S,f)$ is defined by :  \[\PH(f) := \roim{p}\cor_{\Gamma_f^\gamma} \in \Der (\cor_\V) .\] 
We use the notation  $\PH^i(f) := \Hn^i(\roim{p}\cor_{\Gamma_f^\gamma})$, for $i \in \Z$.
\end{definition}

\begin{remark} Let $M$ be a real analytic manifold (for instance $M=\R^n$) and let $S$ be a good topological space. We assume that $\gamma$ is subanalytic and that we have the data of $i_S \colon S \to M$ a closed immersion whose image is a  subanalytic subset of $M$ and that $f \colon S \to \V$ is continuous and subanalytic in $M$ i.e. the graph  $(i_S \times \id_\V)(\Gamma_f)$ is subanalytic in $M \times \V$. Following \cite{KS18}, if we also assume that
\begin{equation}\label{hyp:contruct_peristence}
	\textnormal{For each $K \subset \V$ compact, the set $\{x \in S \mid f(x) \in K+ \gamma\}$ is compact.}
\end{equation}
Then the sheaf $\roim{p}\cor_{\Gamma_f^\gamma}$ is constructible. Indeed, writing $p_M \colon M \times \V \to \V$ for the projection we have
\begin{align*}
	\roim{p}\cor_{\Gamma_f^\gamma} \simeq \roim{p_M} (\roim{i_S}\cor_{\Gamma_f^\gamma}) \simeq \roim{p_M}\cor_{i_S(\Gamma_f^\gamma)}
\end{align*}
and the result follows immediately from \cite[Thm.~1.11]{KS18}. Furthermore, if $M$ is compact, then $\roim{p}\cor_{\Gamma_f^\gamma}$ is constructible up to infinity.
\end{remark}

Let $f \colon \V \to S$ be a continuous map. Remark that $\Gamma_f \times \gamma^a\subset \opb{(\id_S \times s)} (\Gamma_f^\gamma)$. This provides a canonical map 
\begin{equation*}
\phi:\opb{(\id_S \times s)} \cor_{\Gamma_f^\gamma} \to \cor_{\Gamma_f \times \gamma^a}. 
\end{equation*}
Precomposing the map $\roim{(\id_S \times s)} \phi$ with the morphism 
\begin{equation*}
\cor_{\Gamma_f^\gamma} \to \roim{(\id_S \times s)}\opb{(\id_S \times s)}\cor_{\Gamma_f^\gamma} 
\end{equation*}
induced by the unit of the adjunction $(\opb{(\id_S \times s)},\roim{(\id_S \times s)})$ leads to the map
	\begin{equation}\label{map:isoepisum}
		\cor_{\Gamma_f^\gamma} \to \roim{(\id_S \times s)}\opb{(\id_S \times s)}\cor_{\Gamma_f^\gamma} \stackrel{\roim{(\id_S \times s)}\phi}{\longrightarrow} \roim{(\id_S \times s)} \cor_{\Gamma_f \times \gamma^a}.
	\end{equation}
\begin{lemma}
The morphism \eqref{map:isoepisum} is an isomorphism.
\end{lemma}
\begin{proof}
We show that the morphism \eqref{map:isoepisum} is an isomorphism by checking it at the level of the stalks. The map $\id_S \times s$ is proper on $\Gamma_f \times \gamma^a$ and induces a bijection $\id_S \times s \colon \Gamma_f \times \gamma^a \to \Gamma_f^\gamma$. Let $(x,y) \in S \times \V$. We notice that 
	\begin{equation*}
	\left(\roim{(\id_S \times s)} \cor_{\Gamma_f \times \gamma^a} \right)_{(x,y)}\simeq \rsect \left (\opb{\left (\id_S \times s \right)} \left (x,y \right );\cor_{\Gamma_f \times \gamma^a \cap \opb{(\id_S \times s)} (x,y)}\right ).
	\end{equation*}

First, if $(x,y) \notin \Gamma_f^\gamma$, then $\Gamma_f \times \gamma^a \cap \opb{(\id_S \times s)} (x,y)=\emptyset$ and the stalk at $(x,y)$ of the morphism \eqref{map:isoepisum} is an isomorphism. Second, if $(x,y) \in \Gamma_f^\gamma$, then $\Gamma_f \times \gamma^a \cap \opb{(\id_S \times s)} (x,y)=(x,f(x),y-f(x))$. Thus, \[\rsect(\opb{(\id_S \times s)} (x,y),\cor_{\Gamma_f \times \gamma^a \cap \opb{(\id_S \times s)} (x,y)})\simeq \cor.\] Morphism \eqref{map:isoepisum} induces a non-zero map from $\cor$ to $\cor$ which is an isomorphism.
\end{proof}
\begin{proposition} \label{prop:bassublevel}
	Assume that $f$ is proper. Then
	\begin{equation*}
		\roim{p}\cor_{\Gamma_f^\gamma} \simeq (\roim{f} \cor_S) \npsconv \cor_{\gamma^a}.
	\end{equation*}
\end{proposition}

\begin{proof}

Applying the functor $\roim{p}$ to the isomorphism \eqref{map:isoepisum}, we get
	\begin{align*}
		\roim{p} \cor_{\Gamma_f^\gamma} & \simeq \roim{p} \roim{(\id_S \times s)} \cor_{\Gamma_f \times \gamma^a}.
	\end{align*}
Moreover,	
	\begin{align*}	
		\roim{p} \roim{(\id_S \times s)} \cor_{\Gamma_f \times \gamma^a} & \simeq \roim{s} \roim{(p \times \id_\V)} \cor_{\Gamma_f \times \gamma^a}\\
		&\simeq \roim{s} \reim{(p \times \id_\V)} \cor_{\Gamma_f \times \gamma^a} \quad \textnormal{(Properness of $f$)}\\
		&\simeq \roim{s}(\reim{p} \cor_{\Gamma_f} \boxtimes \cor_{\gamma^a})\quad \textnormal{(Künneth isomorphism)}\\
		&\simeq \roim{s}(\roim{p} \cor_{\Gamma_f} \boxtimes \cor_{\gamma^a})\quad \textnormal{(Properness of $f$)}\\
		&\simeq \roim{s}(\roim{f}\cor_{S} \boxtimes \cor_{\gamma^a}).
	\end{align*}
\end{proof}

\begin{corollary}
Let $S$ be a good compact topological space and $f \colon S \rightarrow \V$ be a continuous map. The sheaf $\PH(f)$ is $\gamma$-compactly generated.
\end{corollary}

\subsection{Properties of \texorpdfstring{$\gamma$}{gamma}-compactly generated sheaves}

In this subsection, we study the properties of $\gamma$-compactly generated sheaves and deduce from them results for sublevel sets multi-parameter persistent sheaves.

\begin{lemma}\label{lem:comsublevel} Let $u \in \Int(\gamma^\circ)$.
\begin{enumerate}[(i)]
\item If $F \in \Derb_{\comp}(\cor_{\V})$. Then, $\roim{u}(F \npsconv  \cor_{\gamma^a}) \simeq (\roim{u} F) \npsconv  \cor_{\R_-}$.
\item Let $S$ be a good compact topological space, $f \colon S \to \V$ be a continuous map. Then, $\roim{u}\roim{p}\cor_{\Gamma_f^\gamma} \simeq (\roim{(uf)} \cor_S) \npsconv  \cor_{\R_-}$. In other words, 
\begin{equation*}
\roim{u}\PH(f)\simeq \PH(u \circ f).
\end{equation*}
\end{enumerate}
\end{lemma}

\begin{proof} 
\noindent (i)  We have
\begin{align*}
    \roim{u}(F \npsconv  \cor_{\gamma^a}) &\simeq \roim{u} \roim{s} (F \boxtimes \cor_{\gamma^a})\\
    & \simeq \roim{s} (\roim{(u \times u)} (F \boxtimes \cor_{\gamma^a})).
\end{align*}
Since $u \in \Int(\gamma^\circ)$, it follows from Lemma \ref{lem:polarpropre} that $u$ is proper on $\gamma^a$. As $\Supp(F)$ is compact, this implies that $u \times u$ is proper on $\Supp(F) \times \gamma^a$. Hence, applying Künneth formula
\begin{align*}
\roim{s} (\roim{(u \times u)} F \boxtimes \cor_{\gamma^a})& \simeq \roim{s} (\roim{u}F \boxtimes \roim{u}\cor_{\gamma^a})\\
& \simeq \roim{s} (\roim{u}F \boxtimes \cor_{\R_-}) \quad \textnormal{(Lemma \ref{lem:oimcone})}\\
&\simeq (\roim{u} F) \npsconv  \cor_{\R_-}.
\end{align*}

\noindent (ii) Since $S$ is compact, the support of $\roim{f}\cor_S$ is compact. Hence, applying (i) with $F:=\roim{f}\cor_S$, we have
\begin{align*}
    \roim{u}((\roim{f} \cor_S) \npsconv  \cor_{\gamma^a}) &\simeq (\roim{(uf)} \cor_S) \npsconv  \cor_{\R_-}.
\end{align*}
\end{proof}

\begin{remark}
We emphasize that the pushforward by a linear form $u$ of the module $\PH(f)$ is in general different of the module $\PH(u\circ f)$. Let $X$ and $Y$ be good topological spaces and consider two functions $f \colon X \to \V$ and $g \colon Y \to \V$. Then, the quantity $\sup_{ \opnorm{u} \leq 1} \dist_\R(\roim{u}\PH(f)$,$ \roim{u} \PH(g))$ is a lower bound of the interleaving distance $\dist_\V(\PH(f),\PH(g))$ whereas this is in general not the case for $\dist_\R(\PH(u\circ f),\PH(u \circ g))$ as shown by the following proposition.
\end{remark}
Consider $(\R^2, \norm{\cdot}_\infty)$ equipped with the usual cone $\gamma = \R_{\leq 0}^2$.
Let $X = \{(-1,1) \cdot t \mid t \in [0,1]\}$ and for $s \geq 2$, $Y_s = \text{Conv}( X \cup (0,s))$, where $\text{Conv}$ stands for the the convex hull of a set. Let $f : X \to \R^2$ and $g_s : Y_s \to \R^2$ be the inclusions. We also set $p : \R^2 \to \R$ defined by $p(x,y) = \frac{y-x}{2}$, which is a $1$-Lipschitz linear form $(\R^2, \|\cdot\|_\infty) \to (\R,|\cdot|)$.

\begin{proposition}\label{p:cexpositivity}
The following holds:

\[\dist_{\R^2}(\PH(f), \PH(g_s)) = 0 ~~~\text{and}~~~\dist_{\R}(\PH(p\circ f), \PH(p\circ g_s)) = \left |\frac{s}{2} - 1 \right|. \]

In particular, one has:

\[\dist_{\R}(\PH(p\circ f), \PH(p\circ g_s)) \xrightarrow[s \to \infty]{} \infty \quad \textnormal{while} \quad \forall s \geq 2, \dist_{\R^2}(\PH(f), \PH(g_s)) = 0\]
\end{proposition}

\begin{proof}

For $s\geq 0$, since $Y_s$ is compact, $g_s$ is proper and one has \begin{align*}
    \PH(g_s) \simeq (\roim{g_s} \cor_{Y_s} ) \npsconv \cor_{\gamma^a} \simeq \cor_{Y_s} \npsconv \cor_{\gamma^a}. 
\end{align*}

Moreover, the sum map is proper on $Y_s \times \gamma^a$, so that $\cor_{Y_s} \npsconv \cor_{\gamma^a} \simeq \cor_{Y_s}  \star \cor_{\gamma^a}$. Similarly, one proves that $\PH(f) \simeq \cor_X \star \cor_{\gamma^a}$. Since $X$ and $Y_s$ are both compact and convex, by \cite[Ex II.20]{KS90}, one has $\cor_{Y_s} \star \cor_{\gamma^a} \simeq \cor_{Y_s + \gamma^a} \simeq \cor_{X + \gamma^a} \simeq  \cor_{X} \star \cor_{\gamma^a}$. Therefore, we have proved that $\PH(f) \simeq \PH(g_s)$, which in particular implies that $\dist_{\R^2}(\PH(f), \PH(g_s)) = 0$.

Similarly as above, one has $\PH(p\circ f) \simeq (\roim{(p\circ f)} \cor_X) \star \cor_{\R^+} \simeq \cor_{[1, \infty)}$ and $\PH(p\circ g_s) \simeq \cor_{[\frac{s}{2}, \infty)}$. By the derived isometry theorem \cite[Theorem 5.10]{BG18}, one concludes that $\dist_{\R}(\PH(p\circ f), \PH(p\circ g_s)) = \left |\frac{s}{2} - 1 \right|$.

\end{proof}

\begin{figure}
    \centering
    \includegraphics[width = 0.45\textwidth]{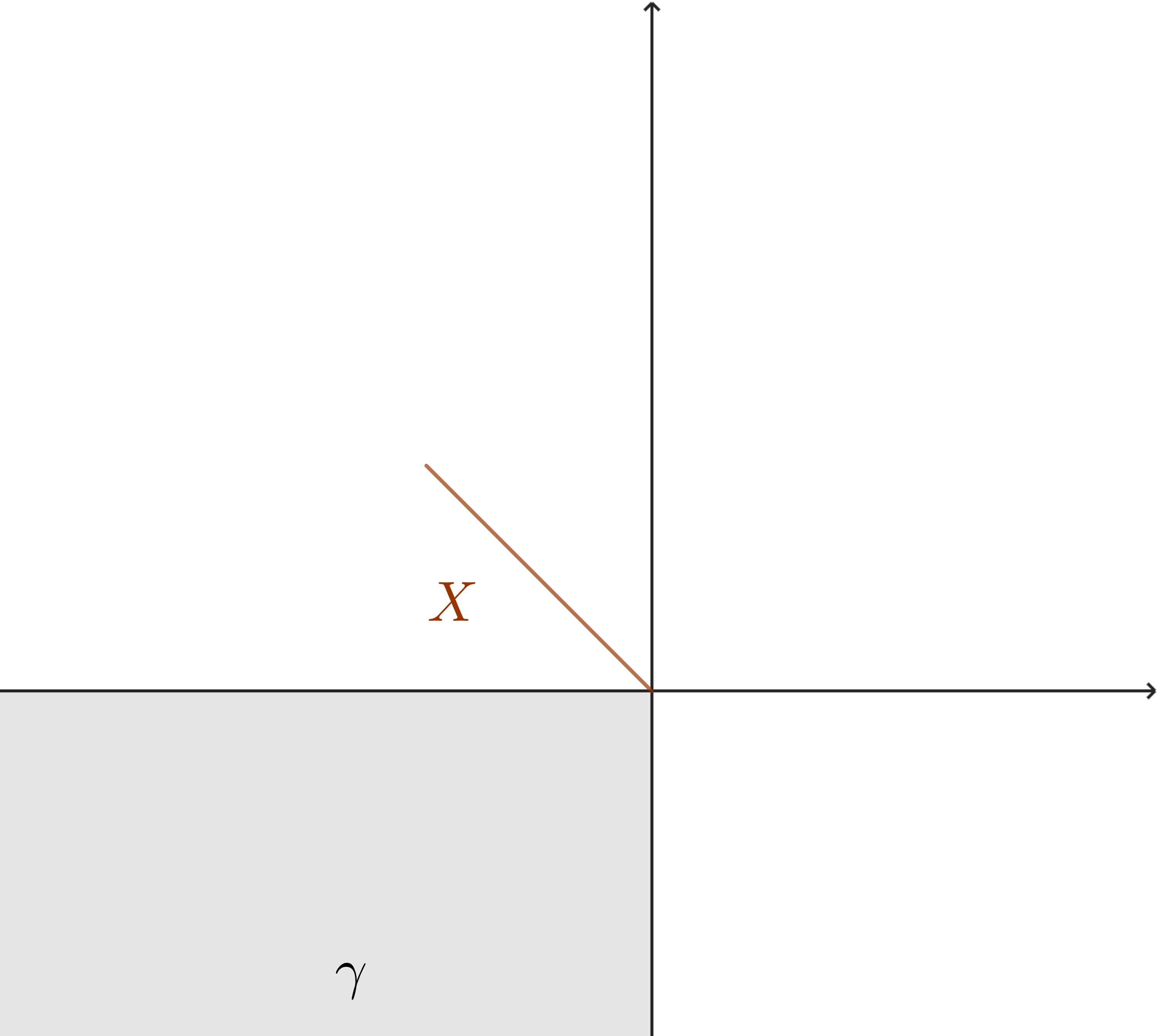}
    \includegraphics[width = 0.45\textwidth]{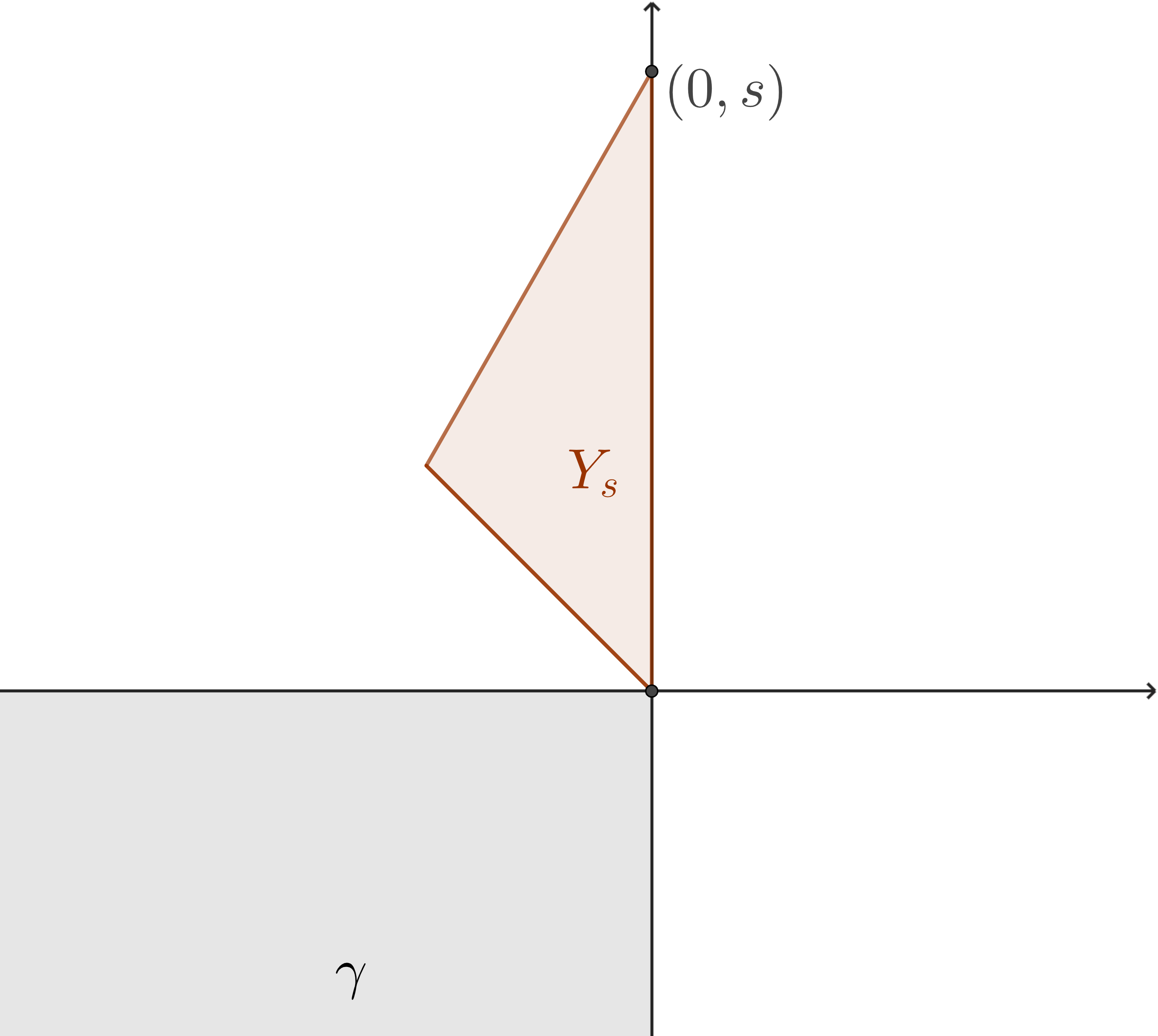}
    \caption{The sets $X$ (left) and $Y_s$ (right)}
\end{figure}

\begin{lemma}\label{lem:gammacomvanishing}
Let $u \notin \Int(\gamma^\circ) \cup \Int(\gamma^{\circ, a})$ and $F$ a $\gamma$-compactly generated sheaf. Then $\reim{u}F \simeq 0$
\end{lemma}

\begin{proof}
Since $F$ is $\gamma$-compactly generated there exist $F^\prime \in  \Derb_{\comp}(\cor_\V)$ such that $F\simeq F^\prime \star \cor_{\gamma^a}$. Then $\reim{u} F \simeq \reim{u} F^\prime \star \reim{u} \cor_{\gamma^a} \simeq 0$ where the last isomorphism follows from Proposition \ref{prop:boundvanish}. 
\end{proof}

\begin{corollary}
Let $F$ be a $\gamma$-compactly generated sheaf. Then $\rsect_c(\V,F) \simeq 0$.
\end{corollary}

\begin{proposition}\label{prop:gammavanishproj}
		Let $F \in \Derb(\cor_\V)$. Assume that $F$ is $\gamma$-compactly generated and that for all $u \in \Int(\gamma^\circ)$, $\reim{u}F \simeq 0$. Then $F \simeq 0$.
\end{proposition}

\begin{proof}
By hypothesis, for every $u \in \Int(\gamma^\circ) \cup \Int(\gamma^{\circ, \, a}) $, $\reim{u}F \simeq 0.$, Moreover, for every $u \notin  \Int(\gamma^\circ) \cup \Int(\gamma^{\circ, a})$, it follows from Lemma \ref{lem:gammacomvanishing}  that $\reim{u}F \simeq 0$. 
Thus, for every $u \in \mathbb{V}^\ast$, $\reim{u} F \simeq 0$. It follows from Corollary \ref{prop:vanishproj} that $F \simeq 0$.
\end{proof}

\begin{lemma}\label{ineq:Lipschitzgamma}
Let $F \in  \Derb_{\comp}(\cor_{\V})$, $u,v \in \Int(\gamma^{a,\circ})$ (resp. $\Int(\gamma^{\circ})$) and set $S=\Supp(F)$. Then,
\begin{equation*}
    \dist_\R(\reim{u}(F \star \cor_{\gamma^a}),\reim{v}(F \star \cor_{\gamma^a})) \leq  \norm{S} \opnorm{u-v}.
\end{equation*}
\end{lemma}

\begin{proof}
It follows from Lemma \ref{lem:comsublevel} that
\begin{equation*}
  \dist_\R(\reim{u}(F \star \cor_{\gamma^a}),\reim{v}(F \star \cor_{\gamma^a}))= \dist_\R(\reim{u}F \star \cor_{\R_+},\reim{v}F \star \cor_{\R_+}) 
\end{equation*}
Using Proposition \ref{prop:dualdis} (ii) followed by Lemma \ref{lem:continuitylin_projbar}, we obtain
\begin{align*}
\dist_\R(\reim{u}F \star \cor_{\R_+},\reim{v}F \star \cor_{\R_+}) & \leq \dist_\R(\reim{u}F,\reim{v}F)\\
& \leq \norm{S} \opnorm{u-v}
\end{align*}
which complete the proof.
\end{proof}

\begin{proposition}\label{prop:multi_to_single}
Let $F,G \in  \Derb_{\comp}(\cor_{\V})$ and $p \in \gamma^\circ \cup \gamma^{\circ, a}$. Then,
\begin{equation*}
\dist_\R(\reim{p}F \star \cor_\lambda, \reim{p}G \star \cor_\lambda) \leq \opnorm{p} \, \dist_\V(F \star \cor_{\gamma^a}, G \star \cor_{\gamma^a})
\end{equation*}
where $\opnorm{\cdot}$ is the operator norm associated with the norm $\norm{\cdot}$ on $\V$ and $\lambda=[0, +\infty [$ if $p \in \gamma^{\circ,a}$ and $\lambda=[-\infty,0 [$ if $p \in \gamma^{\circ}$.
\end{proposition}
\begin{proof}
We proceed in two steps.\\
\noindent (i) If $p \in \Int(\gamma^{\circ,a}) \cup \Int(\gamma^{\circ})$, we have
\begin{align*}
 \dist_\R(\reim{p}F \star \cor_\lambda, \reim{p}G \star \cor_\lambda)
 &=\dist_\R(\reim{p}(F \star \cor_{\gamma^a}), \reim{p}(G \star \cor_{\gamma^a})) \quad \textnormal{Lem. \ref{lem:comsublevel}} \\
 & \leq \opnorm{p} \, \dist_\V(F \star \cor_{\gamma^a}, G \star \cor_{\gamma^a}) \quad \textnormal{Thm.  \ref{thm:lipstability}}
\end{align*}

\noindent (ii) We now assume that $p \in \partial \gamma^\circ \cup \partial \gamma^{\circ, a}$. Without loss of generality, we can further assume that $p \in \partial \gamma^\circ$. Thus there exists a sequence $(p_n)_{n \in \N}$ of elements of $\Int(\gamma^{\circ})$ such that $\lim_{n \to \infty}p_n=p$. Then
\begin{align*}
\dist_\R(\reim{p}F \star \cor_\lambda, \reim{p}G \star \cor_\lambda) & \leq \dist_\R(\reim{p}F \star \cor_\lambda, \reim{p_n}F \star \cor_\lambda)\\ 
& \quad + \dist_\R(\reim{p_n}F \star \cor_\lambda, \reim{p_n}G \star \cor_\lambda)\\
& \quad + \dist_\R(\reim{p_n}G \star \cor_\lambda, \reim{p}G \star \cor_\lambda).
\end{align*}
 Moreover, it follows from the first step that
\begin{align*}
 \dist_\R(\reim{p_n}F \star \cor_\lambda, \reim{p_n}G \star \cor_\lambda) & \leq \opnorm{p_n} \, \dist_\V(F \star \cor_{\gamma^a}, G \star \cor_{\gamma^a}) 
\end{align*}
and
\begin{align*}
\dist_\R(\reim{p}F \star \cor_\lambda, \reim{p_n}F \star \cor_\lambda) & \leq \dist_\R(\reim{p}F, \reim{p_n}F)\\
& \leq \norm{\Supp(F)} \opnorm{p_n-p}.
\end{align*}
This implies that
\begin{align*}
  \dist_\R(\reim{p}F \star \cor_\lambda, \reim{p}G \star \cor_\lambda) &\leq \opnorm{p_n} \, \dist_\R(F \star \cor_{\gamma^a}, G \star \cor_{\gamma^a})\\ 
  & \quad + \norm{\Supp(F)} \opnorm{p_n-p} \\
  & \quad+\norm{\Supp(G)} \opnorm{p_n-p}
\end{align*}

Taking the limit when $n \to \infty$, we get
\begin{align*}
  \dist_\R(\reim{p}F \star \cor_\lambda, \reim{p}G \star \cor_\lambda) &\leq \opnorm{p} \, \dist_\R(F \star \cor_{\gamma^a}, G \star \cor_{\gamma^a}).
\end{align*}
\end{proof}

\begin{corollary}\label{cor:sensitivitymultipers}
Let $X$ and $Y$ be compact good topological spaces and $f \colon X \to \R^n$, $x \mapsto (f_1(x),\ldots,f_n(x))$ and $g \colon Y \to \R^n$, $y \mapsto (g_1(y),\ldots,g_n(y))$ be continuous maps. Consider the cone $\gamma=]-\infty, 0 ]^n$. Then for every $1 \leq i \leq n$
\begin{equation*}
    \dist_\R(\PH(f_i),\PH(g_i)) \leq \dist_{\R^n}(\PH(f),\PH(g)).
\end{equation*}
\end{corollary}

\begin{proof}
Let $p_i \colon \R^n \to \R$ be the projection on the  $i$-th coordinate. Notice that $p_i \in \partial \gamma$ and $\opnorm{p_i}_\infty=1$. Hence, applying Lemma \ref{lem:comsublevel} and Proposition \ref{prop:multi_to_single}, we get 
\begin{align*}
    \dist_\R(\PH(f_i),\PH(g_i)) = \dist_\R(\PH(p_i\circ f),\PH(p_i \circ g))\leq \dist_\V(\PH(f),\PH(g)).
\end{align*}
\end{proof}
\begin{remark}
The above corollary shows that passing from persistence to multi-persistence tend to increase the distance between persistence modules. A possible interpretation is that multipersistence has a better sensitivity than persistence but it cannot be claimed that it is more robust to outliers. 
\end{remark}

\section{Projected barcodes}
In this section, we introduce the notion of projected barcodes. The projected barcodes of a multi-parameter persistence module $F$ on $\V$ (that is a $\gamma$-sheaf) is the family of barcodes obtained by considering the direct images of $F$ by various maps from $\V$ to $\R$. While the fibered barcode is obtained by pulling back a multi-parameter module on an affine line, we propose instead to study the pushforwards of a $\gamma$-sheaf onto $\R$. We start by providing an example showing that two non isomorphic multi-parameter modules that have the same fibered barcodes can have different projected barcodes. We then formally introduce the notion of $\fF$-projected barcodes and study its properties.

\subsection{Motivations and example}
\label{s:cexamplefibered}

The fibered barcode has been successfully used in a variety of machine learning tasks as a summary of multi-parameter persistence modules \cite{CarriereBlumberg}. Nevertheless,  it is easy to build examples of $\gamma$-sheaves (hence persistence modules) with the same fibered barcodes (hence at matching distance zero) though they are not isomorphic and are at a strictly positive convolution/interleaving distance. Let us describe below one of this well-known situation.

 In this example, we consider $\V = \R^2$, endowed with $\gamma = (-\infty, 0]^2$. We set $v = (1,1)$, therefore $\| (x,y) \|_v = \max(|x| , |y|) = \|(x,y)\|_\infty$. We keep the notation $\Lambda := \{h \in \Int(\gamma^a) , \|h\|_v = 1\}$. We let $a >2$ and  define the two sheaves $F$ and $G$ in $\Derb_{\rc, \gamma^{\circ,a}}(\cor_{\V})$ as follows: $F = \cor_{[1,a) \times [0,a)} \oplus \cor_{[0,a) \times [1,a)} $ and $G = \cor_A \oplus \cor_{[1,a) \times [1,a)} $, where 

\begin{equation*}
A = [0,a) \times [0,a) \setminus \left ([0,1) \times [0,1)\right ).
\end{equation*}
The sheaf $\cor_A$ fits into the following short exact sequence
\begin{equation*}
    0 \to \cor_{[0,1)\times [0,1)} \to \cor_{[0,a) \times [0,a)} \to \cor_A \to 0.
\end{equation*}
\begin{center}
	\begin{figure}
		\centering
		\includegraphics[width=\textwidth]{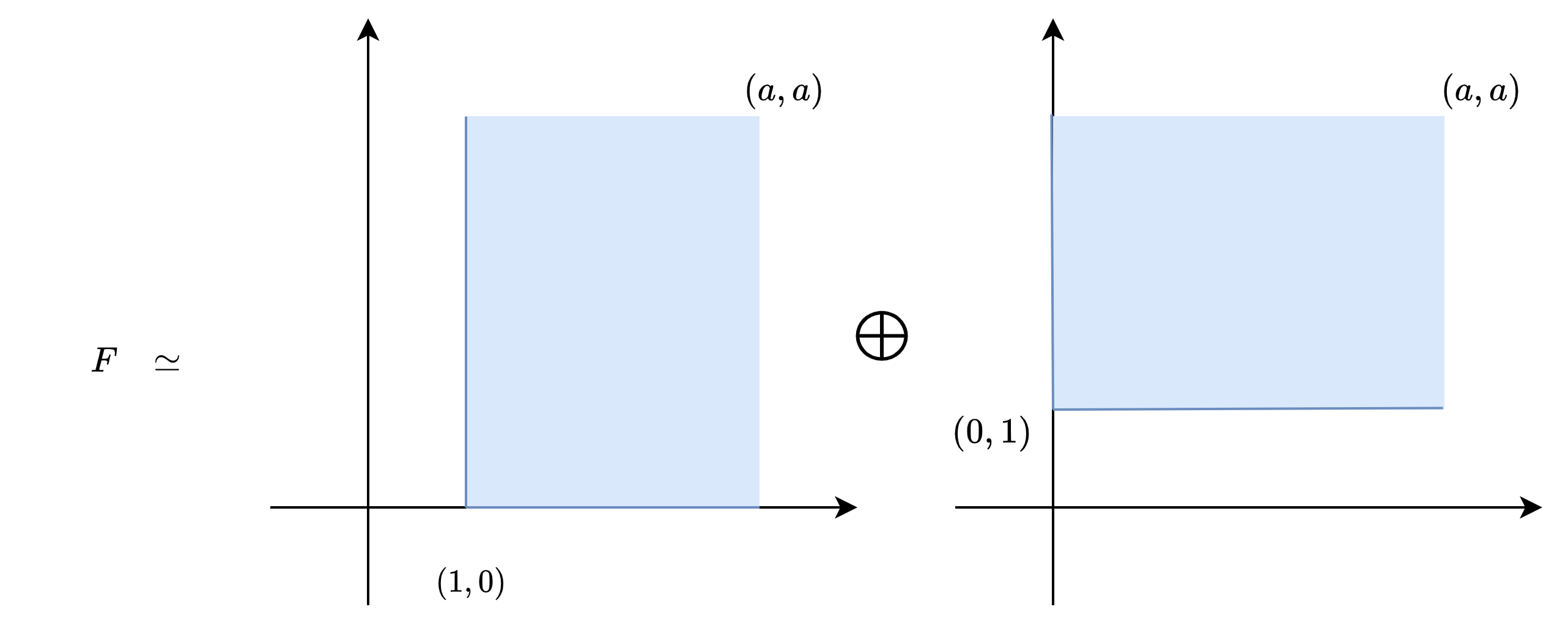}
		\includegraphics[width=\textwidth]{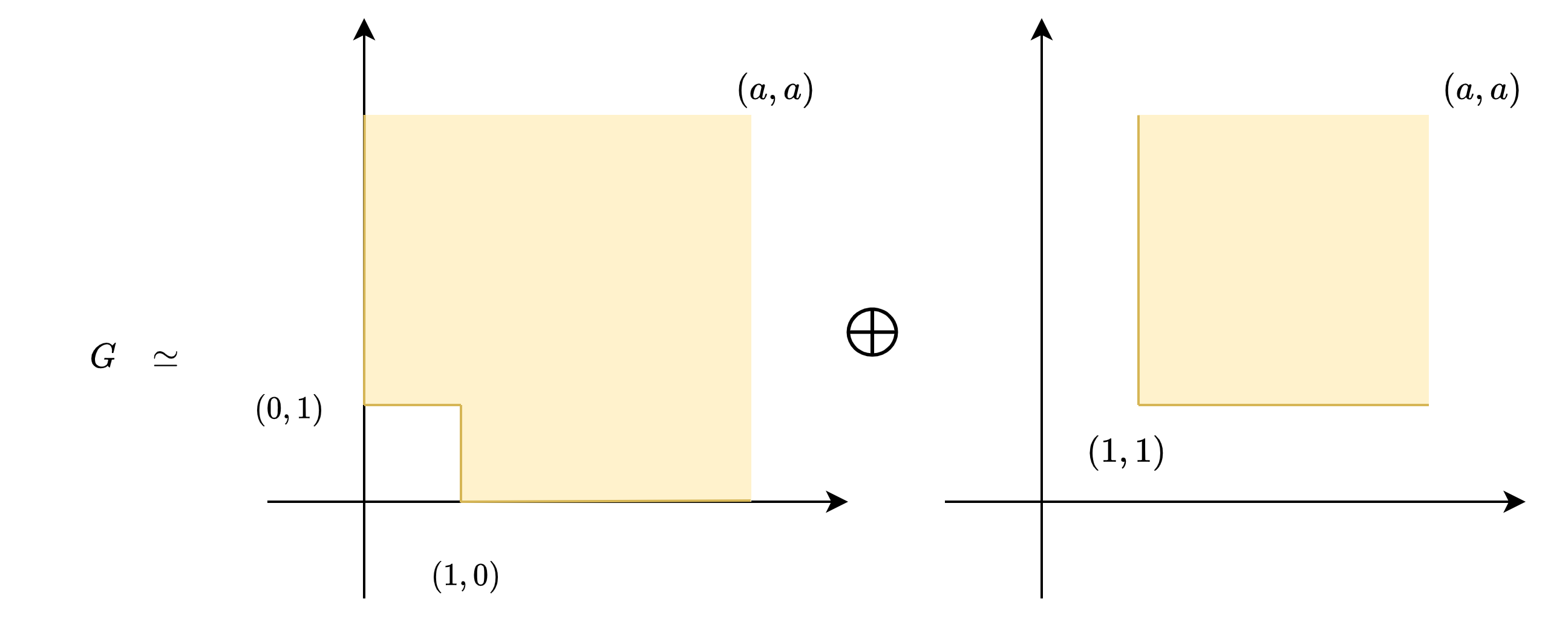}
		\caption{The sheaves $F$ and $G$ in $\Derb_{\rc}(\cor_{\R^2})$.}
		\label{fig:cefibered}
	\end{figure} 
\end{center}

Some classical computations (see for instance \cite[Example 2.1]{vipond2020local}) give $\opb{i_{\mathcal{L}_h}} \opb{\tau_{c}}F \simeq \opb{i_{\mathcal{L}_h}} \opb{\tau_{c}}G$ for all $(h,c) \in \Lambda \times \V$. Therefore, one cannot distinguish $F$ and $G$ by restrictions to one dimensional affine subspaces. However, one has $\dist_\V(F,G) = 1$.

We now show that it is possible to produce barcodes out of $F$ and $G$ which are different. Our idea is to study the barcode decomposition obtained after applying the direct image functor of a map from $\V$ to $\R$, rather than restrictions, which correspond to the inverse image functor of injections from $\R$ to $\V$.  Let  $p \colon \V \to \R$ be defined by the formula $p(x,y) = \frac{(x+y)}{2}$. Then, we have
\begin{equation*}
\roim{p} F \simeq \left ( \cor_{[\frac{1}{2}, a )    } \right )^2, ~~~ \roim{p} G \simeq  \cor_{[\frac{1}{2}, a )} \oplus \cor_{[\frac{1}{2}, 1 ) }  \oplus \cor_{[1, a)}.
\end{equation*}
According to Definition \ref{D:matching}, the partial matching between $\B(\roim{p} F)$ and $\B(\roim{p} G)$ given by
\begin{equation*}
\left [ \frac{1}{2}, a  \right) \leftrightarrow \left [\frac{1}{2}, a  \right),~~\left [\frac{1}{2}, a \right ) \leftrightarrow \left [1, a \right ),~~\emptyset \leftrightarrow \left [\frac{1}{2}, 1 \right )
\end{equation*}
is a $\frac{1}{2}$-matching, and a simple computations shows that for every $\frac{1}{2}>\varepsilon>0$, there is no $\varepsilon$-matching between $\B(\roim{p} F)$ and  $\B(\roim{p} G)$. Using Theorem \ref{T:Isometry}, one deduces that

\begin{equation*}
\dist_\R (\roim{p} F, \roim{p} G ) = d_B(\mathbb{B}(\roim{p} F),\mathbb{B}(\roim{p} G) ) = \frac{1}{2}.
\end{equation*}

Therefore, it is possible to distinguish $F$ from $G$ by using barcodes obtained from pushforwards, rather than pullbacks.

\subsection{\texorpdfstring{$\fF$}{F}-Projected barcodes}

\subsubsection{Generalities}
Recall  that $\mathcal{SC}(\V_\infty)$ denotes the set of continuous maps from $\V$ to $\R$ which are subanalytic up to infinity and let $\fF$ be a subset of $\mathcal{SC}(\V_\infty)$. We regard $\fF$ as a discrete category.

\begin{definition}\label{def:projectedbarcode}
The $\fF$-projected barcodes is the functor
\begin{align}\label{map:projbarcode}
    \projbar^\fF \colon \fF \times \Derb_{\rc}(\cor_{\V_\infty}) \to \Derb_{\rc}(\cor_{\R_\infty}), \quad (u,F) \mapsto \reim{u} F.
\end{align}
and the non-proper projected barcodes is the functor
\begin{align}\label{map:npprojbarcode}
    \npprojbar^\fF \colon \fF \times \Derb_{\rc}(\cor_{\V_\infty}) \to \Derb_{\rc}(\cor_{\R_\infty}), \quad (u,F) \mapsto \roim{u} F.
\end{align}

\noindent When the context is clear, we will omit $\fF$ from the notation.
\end{definition}
We will often study restriction of the (non-proper) projected barcodes to specific subcategories of $\Derb_{\rc}(\cor_{\V_\infty})$. For instance, the category of compactly supported constructible sheaves or the category of $\Derb_{\rc, \gamma^{\circ,a}}(\cor_{\V_\infty})$.
\begin{Notation}
Let $F \in \Derb_{\rc}(\cor_{\V_\infty})$. We set $\projbar_F^\fF:=\projbar^\fF(\cdot,F)$ and $\npprojbar_F^\fF:=\npprojbar^\fF(\cdot,F)$.
\end{Notation}

\begin{remark}
The proper and the non-proper projected barcodes are related via the following formula which follows from Corollary \ref{cor:roimtoreim}.
Let $F \in \Derb_{\rc}(\cor_{\V_\infty})$, Then,
\begin{equation*}
    \npprojbar_F^\fF=\dual_\R(\projbar^\fF_{\dual_\V(F)}).
\end{equation*}
\end{remark}

We now provide a few examples of projected barcodes of interest.

\subsubsection{Linear projected barcodes}
 Let $(\V,\norm{\cdot})$ be a real finite dimensional vector space and $\V^\ast$ its dual endowed with the operator norm $\opnorm{\cdot}$.

 We set $\mathbb{S}^\ast=\{u \in \V^\ast \mid \, \opnorm{u}=1\}$. The linear projected barcodes is the functor
 \begin{align*}\label{map:projbarcode}
    \projbar^{\mathbb{S}^\ast} \colon \mathbb{S}^\ast \times \Derb_{\rc}(\cor_{\V_\infty}) \to \Derb_{\rc}(\cor_{\R_\infty}), \quad (u,F) \mapsto \reim{u} F.
\end{align*}

\begin{proposition}\label{P:nullitytest} The linear projected barcodes has the following properties.
\begin{enumerate}[(i)]
\item Let $F \in \Derb_\rc(\cor_{\V_\infty})$. If $\projbar^{\mathbb{S}^\ast}_F \simeq 0$, then $F \simeq 0$,
\item The map $\projbar^{\mathbb{S}^\ast} \colon \mathbb{S}^\ast \times \Derb_{\rc,\comp}(\cor_{\V_\infty}) \to \Derb_{\rc}(\cor_{\R_\infty})$, $(u,F)\mapsto \projbar_F^{\mathbb{S}^\ast}(u)$ is continuous.
\item Let $u \in \mathbb{S}^\ast$. The map $\Derb_{\rc,\comp}(\cor_{\V_\infty}) \to \Derb_{\rc}(\cor_{\R_\infty})$, $F\mapsto \projbar_F^{\mathbb{S}^\ast}(u)$ is Lipschitz.
\end{enumerate}
\end{proposition}

\begin{proof}
\noindent (i) follows from Proposition \ref{prop:vanishproj}.\\
\noindent (ii) Let us show that $\projbar$ is continuous. Let $(u,F) \in \mathbb{S}^\ast \times \Derb_{\rc,\comp}(\cor_{\V_\infty})$ and let $\varepsilon > 0$. For every $v \in \mathbb{S}^\ast$ such that $\opnorm{u-v}< \varepsilon / (2 \norm{\Supp(F)})$ and $G \in \Derb_{\rc,\comp}(\cor_{\V_\infty}) $ such that $\dist_\V(F,G) \leq \varepsilon/2$, we have
\begin{align*}
    \dist_\R(\projbar^{\mathbb{S}^\ast}_F(u),\projbar^{\mathbb{S}^\ast}_G(v)) &\leq \dist_\R(\projbar^{\mathbb{S}^\ast}_F(u),\projbar^{\mathbb{S}^\ast}_F(v)) + \dist_\R(\projbar^{\mathbb{S}^\ast}_F(v),\projbar^{\mathbb{S}^\ast}_G(v))\\
    &\leq  \norm{\Supp(F)} \opnorm{u-v} + \opnorm{v} \dist_\V(F,G)\\
    & \leq  \norm{\Supp(F)} \opnorm{u-v} + \dist_\V(F,G) \leq \varepsilon
\end{align*}
\noindent (iii) This follows from the above inequality by taking $u=v$.
\end{proof}

\begin{remark}
The first point of the above proposition expresses that the linear projected barcodes can be used as a nullity test of persistence modules, where the space of test parameters is $\mathbb{S}^*$, thus compact. This is a fundamental difference with the fibered barcode, whose parameter test space is \emph{not} compact.
\end{remark}

\subsubsection{$\gamma$-Linear projected barcode} We introduce a notion of projected barcode tailored for $\gamma$-sheaves, that is, for persistence modules. For sublevel sets persistence modules, we explain in example \ref{ex:softwarecomputation} how it can be computed with standard one-parameter persistence software.

Since this version of the projected barcode is aimed at $\gamma$-sheaves, it is natural, in view of Proposition \ref{prop:oimgamma} and Lemma \ref{lem:polarpropre}, to set $\fF = \Int(\gamma^{\circ, \,a}) \cap \mathbb{S}^\ast$ where $\mathbb{S}^\ast$ is the unit sphere of $\V^\ast$ for the norm $\opnorm{\cdot}$ . This version of the projected barcode is called $\gamma$-linear projected barcode. For the sake of brevity, we set $Q_\gamma=\Int(\gamma^{\circ,a }) \cap \mathbb{S}^\ast$.

\begin{definition}
The $\gamma$-linear projected barcode is the functor
\begin{align*}
    \projbar^\gamma \colon Q_\gamma \times \Derb_{\rc, \gamma^{\circ, a}}(\cor_{\V_\infty}) \to \Derb_{\rc, \lambda^{\circ,a}}(\cor_{\R_\infty}), \quad (u,F) \mapsto \reim{u} F.
\end{align*}
where $\lambda=\R_-$.
\end{definition}

\begin{remark}
We focus our attention on the properties of the $\gamma$-projected barcode for $\gamma$-sheaves which are $\gamma$-compactly generated. If $F$ is $\gamma$-compactly generated,  then for $u \in \Int(\gamma^{\circ,a})$, one has $\reim{u}F \simeq \roim{u} F$. This is why we only discuss the $\gamma$-projected barcode and omit the study of the non-proper $\gamma$-linear projected barcode.
\end{remark}

\begin{proposition} \label{P:gammalinearnullity} The $\gamma$-linear projected barcode has the following properties
\begin{enumerate}[(i)]
    \item Let $F \in \Derb_{\rc, \gamma^{\circ, a}}(\cor_{\V_\infty})$ and assume that $F$ is $\gamma$-compactly generated. If $\projbar^\gamma_F \simeq 0$, then $F \simeq 0$.
    \item  Let $(u,F) \in Q_\gamma \times \Derb_{\rc, \gamma^{\circ, a}}(\cor_{\V_\infty})$ and assume that $F$ is $\gamma$-compactly generated. The map $\projbar^{\gamma} \colon Q_\gamma \times \Derb_{\rc}(\cor_{\V_\infty}) \to \Derb_{\rc}(\cor_{\R_\infty})$ is continuous in $(u,F)$.
\end{enumerate}
\end{proposition}

\begin{proof}
\noindent (i) This is a direct consequence of Proposition \ref{prop:gammavanishproj}.\\
\noindent (ii) Let $(u,F) \in Q_\gamma  \times \Derb_{\rc,\gcg}(\cor_{\V_\infty})$ and let $\varepsilon > 0$. Since $F$ is $\gamma$-compactly generated there exists a sheaf $F^\prime$ with compact support such that $F \simeq F^\prime \star \cor_{\gamma^a}$. For every $v \in Q_\gamma$ such that $\opnorm{u-v}< \varepsilon / (2 \norm{\Supp(F^\prime)})$ and $G \in \Derb_{\rc}(\cor_{\V_\infty}) $ such that $\dist_\V(F,G) \leq \varepsilon/2$, we have
\begin{align*}
    \dist_\R(\projbar^{\gamma}_{F}(u),\projbar^{\gamma}_G(v)) &\leq \dist_\R(\projbar^{\gamma}_{F}(u),\projbar^{\gamma}_{F}(v)) + \dist_\R(\projbar^{\gamma}_F(v),\projbar^{\gamma}_G(v))\\
    & \leq  \dist_\R(\projbar^{\gamma}_{F^\prime}(u),\projbar^{\gamma}_{F^\prime}(v)) + \opnorm{v} \dist_\V(F,G)\\
    &\leq  \norm{\Supp(F^\prime)} \opnorm{u-v} + \opnorm{v} \dist_\V(F,G)\\
    & \leq  \norm{\Supp(F^\prime)} \opnorm{u-v} + \dist_\V(F,G) \leq \varepsilon
\end{align*}
\end{proof}

\begin{example}\label{ex:softwarecomputation}
    Here, we specialize the situation to the case where $\V=\R^2$ endowed with the norm $\norm{(x,y)}_\infty=\max(|x|,|y|)$ and $\gamma=(-\infty,0]^2$. Then $(\V^\ast,\opnorm{\cdot}_\infty)$ is isometric to  $(\R^2, \norm{\cdot}_1)$ where $\norm{(a,b)}_1=|a|+|b|$. It follows that $Q_\gamma=\{ (a,b) \in \R_{>0} \times \R_{>0} \, ; a+b=1 \}$. If for instance $S$ is a compact real analytic manifold and $f \colon S \to \R^2$ is a subanalytic map with $f=(f_1,f_2)$. Then by Lemma \ref{lem:comsublevel}, the projected barcode $u \mapsto \projbar^\gamma_{\PH(f)}(u)$ is the collection of barcode associated to the family of sublevel sets persistence modules $\left( \PH(a \, f_1+(1-a) \, f_2) \right)_{a \in (0,1)}$.
\end{example}

\subsection{Fibered versus Projected Barcodes}
In this subsection, we prove that the fibered barcode is a special instance of projected barcode.

We get back to the setting of Section \ref{S:fiberedbarcode}, with $\V = \R^n$, $\gamma = (-\infty, 0]^n$ and $v = (1,...,1)$ and denote by $j_{\mathcal{L}_h}$ the inverse of morphism \eqref{isom:RL}. Note that for $h\in \Lambda$, $\lambda := \mathcal{L}_h \cap \gamma$ is a cone of the one-dimensional real vector space $\mathcal{L}_h$ satisfying hypothesis (\ref{hyp:cone}) as a cone of $\mathcal{L}_h$. We denote, $\mathcal{L}_{h, \lambda}$ the topological space $\mathcal{L}_h$ endowed with the $\lambda$-topology. We also consider $\R$ with the cone topology associated with the cone $\lambda_0=(-\infty; 0]$. In particular, $\iota_{\mathcal{L}_h}(\lambda_0)=\lambda$.  

\begin{proposition}
    Let $h\in \Lambda$, the following holds: 
    
    \begin{enumerate}
        \item $p_{\mathcal{L}_h} \circ i_{\mathcal{L}_h} = \iota_{\mathcal{L}_h}$ as maps of sets,
        
        \item $p_{\mathcal{L}_h}$, $i_{\mathcal{L}_h}$ and $j_{\mathcal{L}_h}$ are continuous for both the norm and the cone topologies.
    \end{enumerate}
\end{proposition}

\begin{proof}
    \noindent (i) is clear.\\
    \noindent (ii) We have already noticed that $p_{\mathcal{L}_h}$ is Lipschitz, therefore continuous for the norm topology. Observe that for $x \in \mathcal{L}_h$, $p^{-1}_{\mathcal{L}_h}(x + \Int(\lambda)) = x + \Int(\gamma)$. Moreover, for any $\lambda$-open subset $U\subset {\mathcal{L}_h}$, $U= \bigcup_{x\in U} x + \text{Int}(\lambda)$ by \cite[Lemma 2.1]{BP21}. Hence, $p^{-1}_{\mathcal{L}_h}(U) = \bigcup_{x\in U} x + \text{Int}(\gamma)$. This proves the statement.  
\end{proof}

Let us precise some notations. When the context is clear, we drop the $h$ of the maps $p_{\mathcal{L}_h}$, $j_{\mathcal{L}_h}$  and $i_{\mathcal{L}_h}$.  From now on,  $p_{\mathcal{L}}$, $j_{\mathcal{L}}$ and $i_{\mathcal{L}}$ refer to the morphisms of topological spaces equipped with the norm topologies, whereas $p_{\mathcal{L}}^\gamma$, $j_{\mathcal{L}_h}^\gamma$ and $i_{\mathcal{L}}^\gamma$ refers to the morphisms of topological spaces equipped with the cone topologies.

We now construct an  isomorphism of functors between $\opb{{i^{\gamma}_{\mathcal{L}}}}$ and  $\oim{(j^\gamma_{\mathcal{L}} \, p_{\mathcal{L}}^\gamma)}$.

\begin{proposition}
   There is an isomorphism of (non derived) functors $\oim{(j^\gamma_{\mathcal{L}} \, p_{\mathcal{L}}^\gamma)} \simeq (i^{\gamma}_{\mathcal{L}})^{-1}$ in $\textnormal{Fun}(\Mod(\cor_{\V_\gamma}), \Mod(\cor_{\R_{\lambda_0}}))$.

\end{proposition}

\begin{proof}
     Let $F \in \Mod(\cor_{\V_\gamma})$ and $U \subset \R$ be a non empty $\lambda_0$-open set. Since $\R$ is one dimensional, there exists $t\in \R\cup \{+\infty\}$ such that $U = t + \Int(\lambda_0)$. Hence, $\oim{(j_{\mathcal{L}}^\gamma \, p_{\mathcal{L}}^\gamma)} F (U) = F((p_\mathcal{L}^{\gamma})^{-1} \opb{(j_{\mathcal{L}}^\gamma)}(U)) = F(t \cdot h + \Int( \gamma))$. By definition, $(i^{\gamma}_\mathcal{L})^{-1}F$ is the sheaffification of the presheaf $(i_\mathcal{L}^{\gamma})^{\dag} F: V \mapsto \varinjlim_{V \subset W } F(W)$. For every $\lambda_0$-open $U \subset \R$, the colimit $\varinjlim_{U \subset W } F(W)$ is canonically isomorphic to $F(t\cdot h + \Int(\gamma))$. This leads to an isomorphism of pre-sheaves $ (i_\mathcal{L}^{\gamma})^{\dag} F \to \oim{(j^\gamma_{\mathcal{L}} \, p_{\mathcal{L}}^\gamma)} F$. Since $\oim{(j^\gamma_{\mathcal{L}} \, p_{\mathcal{L}}^\gamma)} F$ is a sheaf, we conclude that so is $(i_\mathcal{L}^{\gamma})^{\dag} F$ and $(i_\mathcal{L}^{\gamma})^{\dag} F \simeq \oim{(j^\gamma_{\mathcal{L}} \, p_{\mathcal{L}}^\gamma)} F \simeq (i^{\gamma}_\mathcal{L})^{-1}F$. Moreover, this isomorphism is functorial in $F$ by functoriality of colimits.
\end{proof}

 We denote by $\phi : \V \to \V_\gamma$, $\phi^\prime : \mathcal{L}_h \to \mathcal{L}_{h,\lambda}$ and $\phi^{\prime \prime} : \R \to \R_{\lambda_0}$ the morphisms of topological spaces induced by the identity maps. We have the following relation $\phi^\prime \, p_{\mathcal{L}} =p^\gamma_{\mathcal{L}} \phi$. For every $G \in \Der[+](\cor_\V)$, we have
\begin{align*}
    \Hom(\opb{\phi}G,\opb{\phi}G)& \to \Hom(\opb{\phi}\opb{{p^\gamma_{\mathcal{L}}}} \roim{p^\gamma_{\mathcal{L}}}G,\opb{\phi}G)\\
        &\simeq \Hom( \opb{p_{\mathcal{L}}}\opb{{\phi^\prime}}\roim{p^\gamma_{\mathcal{L}}}G,\opb{\phi}G)\\
        &\simeq \Hom( \opb{{\phi^\prime}}\roim{p^\gamma_{\mathcal{L}}}G, \roim{p_{\mathcal{L}}}\opb{\phi}G).
\end{align*}
 The image of the identity map by the above sequence of morphisms provides a map
 \begin{equation}\label{map:invplgamma}
     \phi^{' -1} \roim{p^\gamma_{\mathcal{L}}} G \to \roim{p_{\mathcal{L}}} \phi^{-1}G.
 \end{equation}
 
 \begin{lemma}
 Let $G \in \Der[+](\cor_{\V_\gamma})$, the morphism \eqref{map:invplgamma} is an isomorphism.
 \end{lemma}

\begin{proof}
 To prove that the morphism \eqref{map:invplgamma} is an isomorphism, it is sufficient to prove that it induces an isomorphism $\rsect(U; \phi^{' -1} \roim{p^\gamma_{\mathcal{L}}} G) \stackrel{\sim}{\longrightarrow} \rsect(U; \roim{p_{\mathcal{L}}} \phi^{-1}G)$, for every open subset $U \subset \mathcal{L}_h$ of the form $U = \{t \cdot h \mid t \in (a,b) \}$, for $a < b$ in $\R$. Let $\mathcal{I}$ be a bounded from below complex of injective (in particular flabby) sheaves on $\V_\gamma$, which is quasi-isomorphic to $G$.\\
 It follows from equation \eqref{eq:fiberP} that
\begin{equation*}
p_{\mathcal{L}}^{-1}(U) = U + \partial \gamma =  \Int(b \cdot h +\gamma) \setminus (a \cdot h +\gamma).
\end{equation*}
In the canonical basis of $\R^n$, $b \cdot h=(b_1,\ldots,b_n)$ and $a\cdot h=(a_1,\ldots,a_n)$. Then 
\begin{align*}
b \cdot h+\Int(\gamma)=\bigcap_{i=1}^n \{x_i < b_i\}\\
a \cdot h+\gamma=\bigcap_{i=1}^n \{x_i \leq a_i\}
\end{align*}

Hence,
\begin{align*}
    \Int(b \cdot h +\gamma) \setminus (a \cdot h +\gamma) & =  \left( \bigcap_{i=1}^n \{x_i < b_i\} \right) \cap \left( \bigcup_{i=1}^n \{x_i > a_i\} \right)\\
    &= \bigcup_{j=1}^n \left[\left( \bigcap_{i=1}^n \{x_i < b_i\} \right) \cap \{x_j > a_j\} \right]
\end{align*}

For $1 \leq j \leq n$, we set $R_j=\left(\bigcap_{i=1}^n \{x_i < b_i\} \right) \cap \{x_j > a_j\}$. The $R_j$ are convex open subsets of $\V$. For $I \subset \{1,...,n\}$, we define $R_I = \cap_{i \in I} R_i$. In particular, $R_{\{1,\ldots,n\}}=\bigcap_{j=1}^n R_j =  \prod_{i=1}^n (a_i,b_i)$. For all $I \subset \{1,...,n\}$, $R_I$ is a convex open subset, satisfying $R_I + \gamma = p_{\mathcal{L}}^{-1}(U) + \gamma$, see figure \ref{fig:R}.

\begin{figure}
    \centering
    \includegraphics[width=0.5\textwidth]{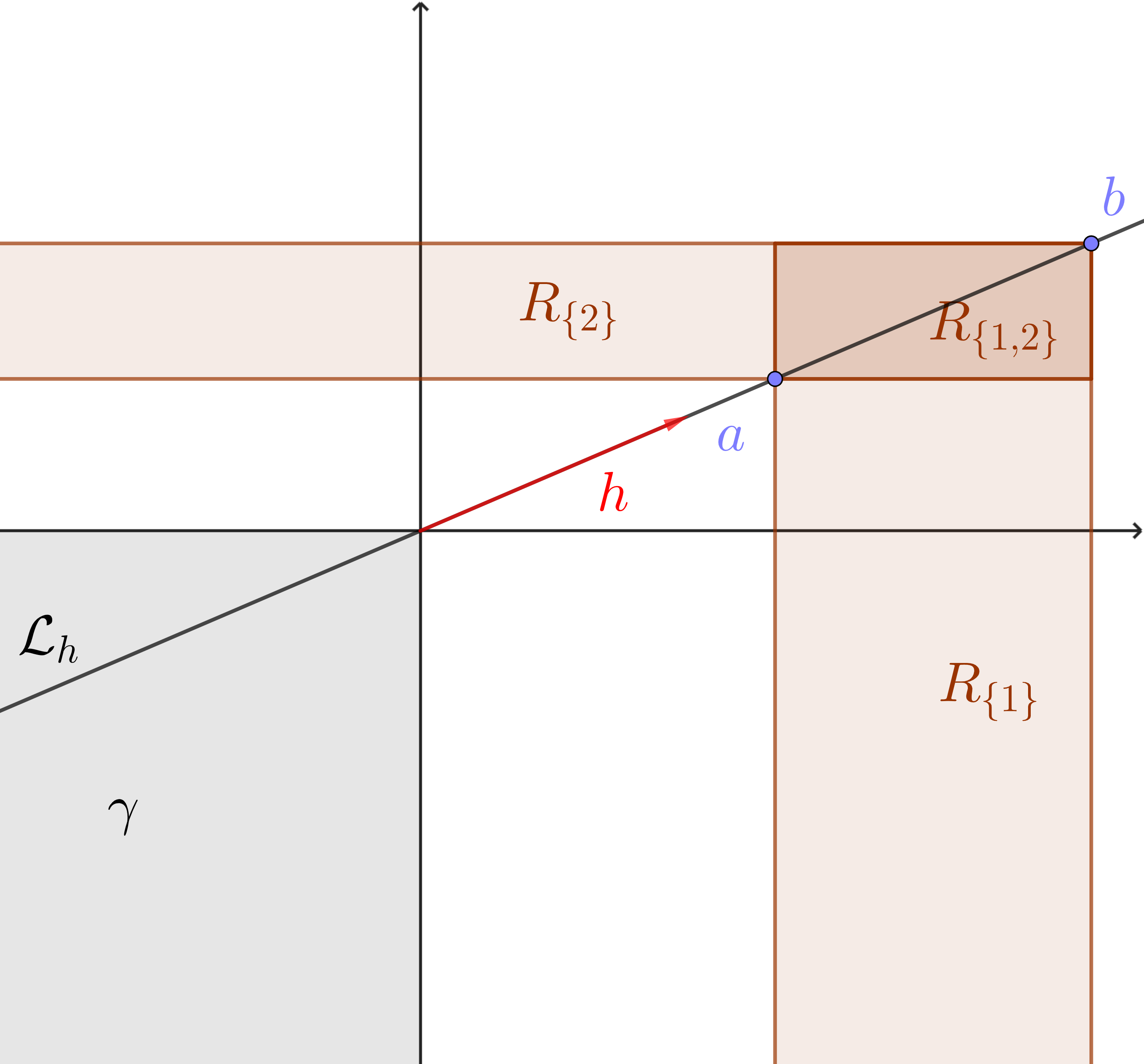}
    \caption{Illustration of the sets $R_I$}
    \label{fig:R}
\end{figure}

Therefore, it follows from the proof of \cite[Prop.~3.5.3]{KS18}, there are isomorphisms of chain complexes in $\mathrm{C}(\Mod(\cor))$
\begin{equation*}
\sect(R_I + \gamma;  \mathcal{I}  ) = \sect(p_{\mathcal{L}}^{-1}(U) + \gamma; \mathcal{I} )  \stackrel{\sim}{\longrightarrow} \sect(R_I ; \phi^{-1} \mathcal{I} ),
\end{equation*}
which commute with the restriction morphisms  $\sect(R_J ; \phi^{-1} \mathcal{I}  ) \longrightarrow \sect(R_I ; \phi^{-1} \mathcal{I}  )$, for $J  \subset I$. Thus, one has the isomorphisms in $\mathrm{C}(\Mod(\cor))$ :

\begin{equation*}
\sect(p_{\mathcal{L}}^{-1}(U) + \gamma; \mathcal{I} ) \simeq \varprojlim_I \sect(R_I ; \phi^{-1} \mathcal{I}  ) \simeq \sect(p_{\mathcal{L}}^{-1}(U) ; \phi^{-1} \mathcal{I}).
\end{equation*}
From the above, we conclude that in $\Der[+](\Mod(\cor))$:
\begin{align*}
    \rsect(U; \phi^{' -1} \roim{p^\gamma_{\mathcal{L}}} G) 
    & \simeq \rsect(U + \lambda;  \roim{p^\gamma_{\mathcal{L}}} G) \\
    & \simeq \rsect(p_{\mathcal{L}}^{-1}(U + \lambda) ; G) \\
    & = \rsect(p_{\mathcal{L}}^{-1}(U) + \gamma ; G) \\
    & \simeq \Gamma(p_{\mathcal{L}}^{-1}(U) + \gamma ; \mathcal{I}) \\
    & \simeq \Gamma(p_{\mathcal{L}}^{-1}(U) ; \phi^{-1} \mathcal{I}) \\
    & \simeq \rsect(p_{\mathcal{L}}^{-1}(U) ; \phi^{-1} G) \\
    & \simeq \rsect(U; \roim{p_{\mathcal{L}}} \phi^{-1}G).
\end{align*}

\end{proof}

We now prove the result announced at the  beginning of the subsection. Namely, that the fibered barcode is a special instance of projected barcode.

\begin{proposition}\label{P:directinverse}
Let $F \in \Der[+]_{\gamma^{\circ,a}}(\cor_{\V})$, then there is a functorial isomorphism $\oim{(j_{\mathcal{L}} \, p_{\mathcal{L}})} F \simeq i_\mathcal{L}^{-1}F$. 
\end{proposition}

\begin{proof} We have the following commutative diagrams of topological spaces:

\begin{equation*}
\xymatrix{
\R \ar[r]^{i_\mathcal{L}} \ar[d]_{\phi''} & \V \ar[d]^{\phi} \ar[r]^{p_\mathcal{L}} &  \mathcal{L}_h  \ar[d]^{\phi'} \ar[r]^{j_\mathcal{L}}  &  \R  \ar[d]^{\phi^{\prime \prime}}\\
\R_{\lambda_0}  \ar[r]^{i_\mathcal{L}^\gamma}  & \V_\gamma \ar[r]^{p_\mathcal{L}^\gamma} & \mathcal{L}_{h,\lambda} \ar[r]^{j_\mathcal{L}^\gamma}   & \R_{\lambda_0} }
\end{equation*}
 
Therefore, one has the following isomorphisms:
\begin{align*}
    i_\mathcal{L}^{-1} F & \simeq i_\mathcal{L}^{-1} \phi^{-1} \roim{\phi} F \\
    & \simeq (\phi \circ i_\mathcal{L} )^{-1} \roim{\phi} F \\
    & \simeq (i^\gamma_{\mathcal{L}} \circ \phi'')^{-1} \roim{\phi} F \\
    & \simeq \phi^{'' -1} \roim{j_{\mathcal{L}}^\gamma} \roim{p_{\mathcal{L}}^\gamma} \roim{\phi} F \\
    & \simeq \phi^{'' -1} \roim{j_{\mathcal{L}}^\gamma} \roim{\phi '} \roim{p_\mathcal{L}} F \\
    & \simeq \phi^{'' -1} \roim{\phi ''} \roim{j_{\mathcal{L}}}  \roim{p_\mathcal{L}} F \\
    & \simeq \roim{j_{\mathcal{L}}} \roim{p_\mathcal{L}} F.
\end{align*}
\end{proof}

\begin{corollary}
Let $F \in \Derb_{\rc, \gamma^{\circ, a}}(\cor_\V)$, Then 
\begin{equation*}
\opb{i_{\mathcal{L}_h}} \opb{\tau_{c}}F \simeq \roim{j_{\mathcal{L}_h}} \roim{p_{\mathcal{L}_h}} \roim{\tau_{-c}} F
\end{equation*}
In other words, the fibered barcode of $F$ is the $\mathfrak{F}$-projected barcode of $F$ with $\mathfrak{F}=\{ j_{\mathcal{L}_h} p_{\mathcal{L}_h} \tau_{-c} \,|  (h,c) \in \Lambda \times \V\}$.
\end{corollary}

\section{Integral sheaf metrics}

\subsection{Generalities}
In this section, we elaborate on our study of projected barcodes to introduce a family of pseudo-metrics on categories of sheaves inspired by integral probability metrics \cite{Mull97}. Here, the probability measures are replaced by sheaves and the integration of real-valued functions against the probability measure by the pushforward of the sheaves by such functions. In this subsection, we study the general properties of such metrics.

Let $X$ be a good topological space. We denote by $\Cont{0}(X)$ the algebra of continuous functions from $X$ to $\R$ and by capital fraktur letters such that $\mathfrak{F}$, $\mathfrak{B}$, $\mathfrak{R}$ subsets of continuous functions. Given a class $\mathfrak{F}$ of functions and writing $\dist_\R$ for the convolution distance on $\Derb(\cor_\R)$, we introduce the following pseudo-metrics on $\Derb(\cor_X)$
\begin{equation}\label{def:ism}
	\isme{\mathfrak{F}}(F,G)=\sup_{f \in \mathfrak{F}}( \dist_\R(\reim{f}F,\reim{f} G)),
\end{equation}

\begin{equation} \label{def:ismnp}
	\ism{\mathfrak{F}}(F,G)=\sup_{f \in \mathfrak{F}}( \dist_\R(\roim{f}F,\roim{f} G)).
\end{equation}

These pseudo-metrics are called \textit{integral sheaf metrics} (ISM) and $\mathfrak{F}$ is called a generator of $\isme{\mathfrak{F}}$ (resp. $\ism{\mathfrak{F}}$).

\begin{remark}
\begin{enumerate}[(a)]
    \item In the above formula, the convolution distance can be replaced by any distance on $\Derb(\cor_\R)$ as for instance the Wasserstein distance. Here, we focus on the convolution distance.
	\item If $F$, $G \in \Derb(\cor_X)$ have compact supports, then we have $\isme{\mathfrak{F}}(F,G)=\ism{\mathfrak{F}}(F,G)$.
\end{enumerate}
\end{remark}

\begin{proposition}
	The mappings $\isme{\mathfrak{F}}, \ism{\mathfrak{F}} \colon \Ob(\Derb(\cor_X)) \times \Ob(\Derb(\cor_X)) \to \R_+ \cup \lbrace \infty \rbrace$ are pseudo-metrics.
\end{proposition}

\begin{proof}
This follows from the fact that $\dist_\R$ is a pseudo-distance.
\end{proof}


\begin{definition}
	Let $\mathfrak{F} \subset \Cont{0}(X)$. The set $\maxgen{\mathfrak{F}}$  of all functions $f \in \Cont{0}(X)$ such that
	\begin{equation}
		\dist_\R(\reim{f}F,\reim{f} G) \leq \isme{\mathfrak{F}}(F,G) ; \textnormal{for all} \; F, \, G \in \Derb(\cor_X)
	\end{equation}
 is called the maximal generator of $\isme{\mathfrak{F}}$.
Similarly, we define for $\ism{\mathfrak{F}}$  the set  $\maxgeno{\mathfrak{F}}$ and call it the maximal generator of $\ism{\mathfrak{F}}$ .
\end{definition}

We have the following straightforward lemma.

\begin{lemma} Let $\mathfrak{F} \subset \mathfrak{D}$ and $F$, $G \in \Derb(\cor_X)$. Then,
	\begin{enumerate}[(i)]
		\item $\isme{\mathfrak{F}}(F,G) \leq \isme{\mathfrak{D}}(F,G)$,
		\item $\maxgen{\mathfrak{F}} \subset \maxgen{\mathfrak{D}}$,
		\item if $\mathfrak{D} \subset \maxgen{\mathfrak{F}}$, then $\isme{\mathfrak{F}}=\isme{\mathfrak{D}}$.
	\end{enumerate}
Similar results hold for $\maxgeno{\mathfrak{F}}$  and $\maxgeno{\mathfrak{D}}$.
\end{lemma}

\begin{proof}
The points (i) and (ii) are clear. For (iii), it follows from (i) that for every $F$ and $G$ in $\Derb(\cor_X)$ $\isme{\mathfrak{F}}(F,G) \leq \isme{\mathfrak{D}}(F,G)$. Applying (i) again we get
\begin{equation*}
    \isme{\mathfrak{D}}(F,G) \leq \isme{\mathfrak{\maxgen{\mathfrak{F}}}}(F,G)=\isme{\mathfrak{F}}(F,G).
\end{equation*}
Hence, $\isme{\mathfrak{F}}=\isme{\mathfrak{D}}$.
\end{proof}

\begin{proposition} Let $\mathfrak{F}$ be a generator of $\isme{\mathfrak{F}}$.
	\begin{enumerate}[(i)]
		\item $f \in \maxgen{\mathfrak{F}}$ implies $a f +b \in \maxgen{\mathfrak{F}}$ for all $a \in \lbrack -1,1\rbrack$ and $b \in \R$
		\item if the sequence $(f_n)_{n \in \N} \subset \maxgen{\mathfrak{F}}$ converges uniformly to $f$, then $f \in \maxgen{\mathfrak{F}}$.
	\end{enumerate}
	
\noindent Results $(i)$ and $(ii)$ hold with $\ism{\mathfrak{F}}$ and $\maxgeno{\fF}$ instead of $\isme{\mathfrak{F}}$ and $\maxgen{\mathfrak{F}}$.
\end{proposition}

\begin{proof}
\begin{enumerate}[(i)]
\item Follows from Lemma \ref{lem:transaffine}.
\item Let $(f_n)_{n \in \N}$ be a sequence of continuous functions of $\maxgen{\mathfrak{F}}$ converging uniformly to a function $f$ and let $F, \, G \in \Derb(\cor_X)$. Then,
\begin{align*}
\dist_\R(\reim{f}F,\reim{f} G) \leq& \dist_\R(\reim{f}F,\reim{f_n} F)+\dist_\R(\reim{f_n}F,\reim{f_n} G)\\
&\quad +\dist_\R(\reim{f_n}G,\reim{f} G)\\
& \leq 2 \norm{f-f_n}_\infty+\isme{\mathfrak{F}}(F,G) \quad \textnormal{(Stability Theorem)}
\end{align*}
Hence, taking $n$ to infinity, the above inequality implies
\begin{equation*}
    \dist_\R(\reim{f}F,\reim{f} G) \leq \isme{\mathfrak{F}}(F,G).
\end{equation*}
\end{enumerate}
\end{proof}

\subsection{A comparison result}
In this subsection, we relate the pseudo-distances $\ism{\fF}$ and $\isme{\fF}$ when the underlying space is a real finite dimensional normed vector space $(\V, \norm{\cdot})$.

\begin{proposition}
Let $F, G \in \Derb(\cor_\V)$, then
\begin{equation*}
    \ism{\fF}(\dual_\V(F),\dual_\V(G)) \leq \isme{\fF}(F,G).
\end{equation*}
\end{proposition}

\begin{proof}
Let $f \colon \V \to \R$ be a continuous map. Then,
\begin{equation*}
    \roim{f}\dual_\V(F) \simeq \dual_\R(\reim{f}F).
\end{equation*}
Hence,
\begin{align*}
     \dist_\R(\roim{f}\dual_\V(F),\roim{f}\dual_\V(G)) &= \dist_\R(\dual_\V(\reim{f}F),\dual_\V(\reim{f}G))\\
     &\leq \dist_\R(\reim{f}F,\reim{f}G) \quad \textnormal{(by Proposition \ref{prop:dualdis} (ii))}
\end{align*}
which implies that $\ism{\fF}(\dual_\V(F),\dual_\V(G)) \leq \isme{\fF}(F,G)$.
\end{proof}

\begin{proposition}
Assume that $\fF \subset \mathcal{A}(\V_\infty)$ and that $F, \, G \in \Derb_{\rc}(\cor_{\V_\infty})$. Then,
\begin{equation*}
    \ism{\fF}(\dual_\V(F),\dual_\V(G)) = \isme{\fF}(F,G).
\end{equation*}
\end{proposition}

\begin{proof}
Let $F, G \in \Derb_{\rc}(\cor_{\V_\infty})$. By Proposition \ref{prop:opconstructinf}, $\reim{f}F$ and $\reim{f}G$ are again constructible up to infinity. Hence,
\begin{align*}
     \dist_\R(\roim{f}\dual_\V(F),\roim{f}\dual_\V(G)) &= \dist_\R(\dual_\V(\reim{f}F),\dual_\V(\reim{f}G))\\
     &= \dist_\R(\reim{f}F,\reim{f}G) \quad \textnormal{(by Lemma \ref{lem:dualiso})}.
\end{align*}
The result follows by taking the supremum over $\fF$ on both sides of the equality.
\end{proof}

\subsection{Lipschitz ISM}\label{subsec:lipism}

In order to get well behaved ISMs, it is natural to assume that the set $\fF$ in the definition of the ISM is a subset of the Lipschitz functions.

Let $(X,d)$ be a good metric space, $(\V, \norm{\cdot})$ be a normed finite dimensional real vector space. We denote the space of Lipschitz functions from $X$ to $\V$ by
\begin{equation*}
	\Lip(X,\V)=\lbrace f \colon X \to \V \mid \, f \, \textnormal{is Lipschitz on X} \rbrace.	
\end{equation*}
The space $\Lip(X,\V)$ can be equipped with the following semi-norm
\begin{equation}\label{eq:LipNorm}
	L(f)=\sup \lbrace \norm{f(x)-f(y)}/ d(x,y) \mid  x,y \in X, ~ x \neq y \rbrace.
\end{equation}
We also set $\Lip_{\leq 1}(X,\V)=\{ f \in \Lip(X,\V) \mid \; L(f) \leq 1\}$.
If $\V=\R$, we write $\Lip(X)$ instead of $\Lip(X,\V)$ and similarly for $\Lip_{\leq 1}(X, \V)$.

A pointed metric space $(X,d, x_0)$ is a metric space $(X,d)$ together with a distinguished point $x_0 \in X$. If $X$ is a vector space, we always choose $x_0=0$. For a pointed metric space $(X,d,x_0)$, we set
\begin{equation*}
	\Lip_0(X,\R)=\lbrace f \colon X \to \R \mid \, f \, \textnormal{is Lipschitz on X and} \, f(x_0)=0 \rbrace.
\end{equation*}

The following proposition is immediate in view of Theorem \ref{thm:lipstability}. It asserts that Lipschitz ISMs provide lower bounds for the convolution distance for sheaves on a good metric space.

\begin{proposition}\label{prop:lowerbound}
Let $(X,d)$ be a good metric space and $F, G \in \Derb(\cor_X)$. Assume that $\fF \subset \Lip_{\leq 1}(X)$. Then
\begin{enumerate}[(i)]
    \item $\isme{\fF}(F,G) \leq \dist_X(F,G)$,
    \item if furthermore, $X$ is a real finite dimensional vector space endowed with the euclidean norm, $\ism{\fF}(F,G) \leq \dist_X(F,G)$.
\end{enumerate}
\end{proposition}

The following proposition is a direct corollary of Proposition \ref{prop:lowerbound} combined with Theorem \ref{thm:stability}. It states that Lipschitz ISMs are stable.

\begin{corollary}
	Let $(\V,\norm{\cdot})$ be a finite dimensional normed vector space and  $Z$ be a locally compact space and let $\mathfrak{F} \subset \Lip_{\leq 1}(\V)$. Let $F \in \Derb(\cor_Z)$  and $f_1, f_2 \colon Z \to \V$ be two continuous maps. Then
	\begin{align*}
		\isme{\mathfrak{F}}(\reim{f_1}F,\reim{f_2}F) \leq \norm{f_1-f_2}_\infty && \isme{\mathfrak{F}}(\roim{f_1}F,\roim{f_2}F) \leq \norm{f_1-f_2}_\infty.
	\end{align*}
If furthermore the norm $\norm{\cdot}$ is  euclidean, then
\begin{align*}
		\ism{\mathfrak{F}}(\reim{f_1}F,\reim{f_2}F) \leq \norm{f_1-f_2}_\infty && \ism{\mathfrak{F}}(\roim{f_1}F,\roim{f_2}F) \leq \norm{f_1-f_2}_\infty.
	\end{align*}
\end{corollary}

\begin{lemma}\label{lem:boulesphere}
	Let  $\mathfrak{F} \subset \lbrace f \in \Lip_0(X,\R) \mid \, L(f) \leq 1 \rbrace$ and $\mathfrak{D}=\lbrace f \in \mathfrak{F}| \, L(f) = 1 \rbrace \cap \mathfrak{F}$. Assume that for every  $f \in \mathfrak{F} \setminus \{0\}$, $f/L(f) \in  \mathfrak{F}$. Then $\isme{\fF}=\isme{\fD}$ and
	$\ism{\fF}=\ism{\fD}$.
\end{lemma}

\begin{proof}

We only prove the statement for  $\isme{\fF}$ and $\isme{\fD}$ as the proof is similar for $\ism{\fF}$ and $\ism{\fD}$. Since $\mathfrak{D} \subset \mathfrak{F}$, $\isme{\fF} \geq \isme{\fD}$.

We now prove the reverse inequality.	
Let $f \in \fF \setminus \lbrace 0\rbrace $. Then $g=f/L(f) \in \fD$ and for every $F, \, G \in \Derb(\cor_X)$
\begin{equation*}
		\dist_\R(\reim{f}F,\reim{f}G) \leq \frac{1}{L(f)}\dist_\R(\reim{f}F,\reim{f}G)=\dist_\R(\reim{g}F,\reim{g}G).
	\end{equation*}
Hence, $\isme{\fF} \leq \isme{\fD}$.
\end{proof}

We now prove some inequalities that we will use to control the regularity of Lipschitz ISM.

\begin{lemma}\label{lem:lip}
Let	$(X,d)$ be a good metric space, $f, g \in \Lip(X)$ and $F,G \in \Derb_{\comp}(\cor_X)$. Set $S=\Supp(F) \cup \Supp(G)$. Then
\begin{equation*}\label{eq:lip}
		|\dist_\R(\reim{f}F,\reim{f}G)-\dist_\R(\reim{g}F,\reim{g}G)| \leq 2 \, 
		\Diam(S) L(f-g).
\end{equation*}
\end{lemma}

\begin{proof}
Let $x_0 \in S$. The distance $\dist_\R$ is invariant by translation. Hence setting $\tilde{f}(x)=f(x)-f(x_0)$ and $\tilde{g}(x)=g(x)-g(x_0)$, we get that
\begin{align*}
    \dist_\R(\reim{f}F,\reim{f}G)=\dist_\R(\reim{\tilde{f}}F,\reim{\tilde{f}}G),\\ \dist_\R(\reim{g}F,\reim{g}G)=\dist_\R(\reim{\tilde{g}}F,\reim{\tilde{g}}G).
\end{align*}
Thus,
\begin{align*}
        |\dist_\R(\reim{f}F,\reim{f}G)-\dist_\R(\reim{g}F,\reim{g}G)| & =
		|\dist_\R(\reim{\tilde{f}}F,\reim{\tilde{f}}G)-\dist_\R(\reim{\tilde{g}}F,\reim{\tilde{g}}G)|
\end{align*}
and
\begin{align*}
		|\dist_\R(\reim{\tilde{f}}F,\reim{\tilde{f}}G)-\dist_\R(\reim{\tilde{g}}F,\reim{\tilde{g}}G)| &\leq |\dist_\R(\reim{\tilde{f}}F,\reim{\tilde{f}}G)-\dist_\R(\reim{\tilde{f}}F,\reim{\tilde{g}}G)| \\
		& \hspace{0.35cm}+ |\dist_\R(\reim{\tilde{f}}F,\reim{\tilde{g}}G)-\dist(\reim{\tilde{g}}F,\reim{\tilde{g}}G)|\\
		&\leq \dist_\R(\reim{\tilde{f}}F,\reim{\tilde{g}}F) +\dist_\R(\reim{\tilde{f}}G,\reim{\tilde{g}}G)\\
		&\leq 2 \norm{\tilde{f}|_S-\tilde{g}|_S}_\infty \\ 
		&\leq 2 \, \sup_{x \in S}(d(x,x_0)) L(f-g)\\
		&\leq 2 \Diam(S) \, L(f-g).
	\end{align*}
\end{proof}

\subsection{Distance kernel}

We now give a first example of ISM.

Let $(\V, \norm{\cdot})$ be a finite dimensional normed vector space such that the distance $(x,y) \mapsto \norm{x-y}$ is subanalytic.
We set 
\begin{equation*}
	\ell_{x_0} \colon X \to \R, \quad x \mapsto \ell_{x_0}(x)=\norm{x-x_0}.
\end{equation*}
For every $x$ in $\V$, the map $\ell_x$ is 1-Lipschitz and proper. Let $X$ be a subset of $\V$. We set $\fF_X=\{ \ell_{x}, \; x \in X \}$ and consider the Lipschitz ISM, $\delta_{\fF_X}$ generated by $\fF_X$, namely

\begin{equation*}
    \delta_\fF(F,G)=\sup_{\ell_x \in \fF_X}\dist_\R(\reim{\ell_x}F,\reim{\ell_x}G) , \quad \text{for} \quad F, G \in \Derb(\cor_\V).
\end{equation*}

\begin{remark} \noindent (a) More generally the above ISM can be defined on any real analytic manifold endowed with a subanalytic distance.\\ 
\noindent(b) Since the applications $\ell_x$ are proper, we have that $\isme{\fF}=\ism{\fF}$.
\end{remark}

We recall the definition of the local Euler-Poincaré index of an object  $F \in \Derb_\rc(\cor_X)$.
\begin{equation*}
    \chi(F) \colon X \rightarrow \Z, \quad x \mapsto \chi(F)(x):=\sum_{i \in \Z} (-1)^i \dim(\Hn^i(F_x)).
\end{equation*}

\begin{proposition}\label{prop:almost_close_distance}
Let $F, G \in \Derb_{\rc}(\cor_\V)$ and assume that $\isme{\fF_\V}(F,G)=0$. Then,
\begin{enumerate}[(i)]
	\item for every $x \in \V$ and $r>0$, $\rsect(B^\prime(x,r);F)\simeq \rsect(B^\prime(x,r);G)$,
	\item $F_x \simeq G_x$
	\item $\chi(F)=\chi(G)$
\end{enumerate} 
\end{proposition}

\begin{proof}
It is clear that (i) implies (ii) and (ii) implies (iii). Hence, we only prove (i).
Let $F, G \in \Derb_{\rc}(\cor_\V)$. It follows from the definition of $\isme{\fF_\V}$, that for every $x \in \V$ $\dist_\R(\reim{\ell_{x}}F,\reim{\ell_{x}}G)=0$. Moreover since for every $x$, $\ell_{x}$ is proper and subanalytic, $\reim{\ell_x}F$ and $\reim{\ell_x}G$ are constructible. Hence, by \cite[Thm.~6.3]{BG18}, we have that
\begin{equation}\label{iso:geodesique}
    \reim{\ell_{x_0}}F \simeq \reim{\ell_{x_0}}G.
\end{equation}
Moreover,
 \begin{align*}
 \rsect([0,r[,\reim{\ell_{x_0}}F) &\simeq \RHom[\cor_\V](\cor_{[0,r[},\reim{\ell_{x_0}}F)\\
 & \simeq \RHom[\cor_\V](\opb{\ell_{x_0}}\cor_{[0,r[},F)\\
 & \simeq \RHom[\cor_\V](\cor_{\opb{\ell_{x_0}}([0,r[)},F)\\ 
 & \simeq \rsect(B^\prime(x_0,r),F).
 \end{align*}
Hence, Applying the functor $\rsect([0,r[,\cdot)$ to both sides of the isomorphism \eqref{iso:geodesique} we get that $\rsect(B^\prime(x_0,r),F)\simeq\rsect(B^\prime(x_0,r),G)$.
\end{proof}

\subsection{Linear ISM}
We assume that $(\V,\norm{\cdot})$ is a real finite dimensional normed vector space and denote by $\V^\ast$ its dual. Again, by analogy with integral probability metrics and in view of Subsection \ref{subsec:lipism} , it is natural to consider the set $\fL=\lbrace u \in \V^\ast \mid L(u) \leq 1 \rbrace$ and introduce the distance $\isme{\fL}$ i.e. for $F,\; G \in \Derb(\cor_\V)$
\begin{equation*}
	\isme{\fL}(F,G)=\sup_{\{u \in \V^\ast \mid L(u) \leq 1 \}} \dist_\R(\reim{u}F,\reim{u}G).
\end{equation*}

 Remark that for a linear map $u \in \V^\ast$, $L(u)$ is equal to the operator norm $\opnorm{u}$ of $u$. This, together with Lemma \ref{lem:boulesphere} implies that
 \begin{equation*}
     \isme{\fL}(F,G)=\sup_{u \in \mathbb{S}^\ast } \dist_\R(\reim{u}F,\reim{u}G),
 \end{equation*}
 where $\mathbb{S}^\ast$ is the unit sphere in $(\V^\ast,\opnorm{\cdot})$.

\begin{proposition} \label{prop:vanishinglinism}
Let $F \in \Derb_{\rc }(\cor_{\V_\infty})$. If $\isme{\fL}(F,0)=0$, then $F \simeq 0$.
\end{proposition}

\begin{proof}
	Let $F  \in \Derb_{\rc }(\cor_{\V_\infty})$ such that $\isme{\fL}(F,0)=0$. Since $F \in \Derb_{\rc }(\cor_{\V_\infty})$, it follows that $\reim{u}F \in \Derb_{\rc }(\cor_{\V_\infty})$. Applying Theorem \ref{thm:closedis}, we deduce that for every $u \in \V^\ast$, $\reim{u}F=0$. Now, Proposition \ref{prop:vanishproj} implies that $F \simeq 0$.
\end{proof}

\begin{corollary} \label{cor:contdist}
	Let $u, v \in \V^\ast$ and $F,G \in \Derb_{\comp}(\cor_\V)$. Set $S=\Supp(F) \cup \Supp(G)$. Then
\begin{equation*}
		|\dist(\reim{u}F,\reim{u}G)-\dist(\reim{v}F,\reim{v}G)| \leq 2 \, 
		\Diam(S) \, \opnorm{u-v}.
\end{equation*}
\end{corollary}

\begin{proof}
  For a linear map $u \in \V^\ast$, $L(u)$ is equal to the operator norm $\opnorm{u}$ of $u$. Then, the inequality follows from Lemma \ref{lem:lip}.
\end{proof}

The above proposition implies that for $F,G \in \Derb_{\comp}(\cor_\V)$ the map
\begin{equation*}
	\Upsilon_{F,G} \colon \V^\ast \to \R, \; u\mapsto \dist_\R(\reim{u}F,\reim{u}G)
\end{equation*}
is Lipschitz on $\V^\ast$. It follows from the Rademacher Theorem that the map $\Upsilon_{F,G}$ is $\Cont{1}$ almost everywhere. Since it is continuous and $\lbrace u \in \V^\ast \mid L(u) \leq 1 \rbrace$ is compact $\Upsilon_{F,G}$ reaches its supremum that is for $F,G \in \Derb_{\comp}(\cor_\V)$
\begin{equation*}
	\isme{\fL}(F,G)= \max_{u \in \fL} \Upsilon_{F,G}(u).
\end{equation*}

\subsection{Sliced convolution distance}\label{subsec:sliced}
The result of the preceding section allows to introduce the notion of sliced convolution distance for sheaves in $\Derb_\comp(\cor_\V)$. For the sake of simplicity, we assume that $\V$ is endowed with the euclidean norm $\norm{\cdot}_2$ and let $\dist_\V$ be the convolution distance associated with $\norm{\cdot}_2$.  Corollary \ref{cor:contdist} implies that for $F, G \in \Derb_\comp(\cor_\V)$ the map $\Upsilon_{F,G} \colon \V^\ast \to \R, \; u\mapsto \dist_\R(\reim{u}F,\reim{u}G)$ is measurable. Hence, for $p \in \N^\ast$, we set
\begin{equation} \label{def:sliced}
     \sliced_p(F,G) := \frac{1}{\vol(\mathbb{S}^\ast)} \left(\int_{\mathbb{S}^\ast} \Upsilon_{F,G}(u)^p \, du \right)^\frac{1}{p}
\end{equation}
where $\mathbb{S}^\ast$ is the Euclidean sphere of radius 1 of $\V^\ast$ and $du$ is the canonical volume form on $\mathbb{S}^\ast$. This is the  \emph{$p$-th sliced convolution distance}.

The following proposition is clear.

\begin{proposition}
The application $\sliced_p \colon  \Ob(\Derb_\comp(\cor_\V)) \times \Ob(\Derb_\comp(\cor_\V)) \to \R_+$ is a pseudo-distance and for $F, \, G \in \Derb_\comp(\cor_{\V})$
\begin{equation*}
    \sliced_p(F,G) \leq \dist_\V(F,G).
\end{equation*}
\end{proposition}

\begin{lemma}
Let $F \in \Derb_{\rc,\comp}(\cor_{\V_\infty})$. If $\sliced_p(F,0)=0$, then $F \simeq 0$.
\end{lemma}

\begin{proof}
 It follows from the hypothesis that the function $\Upsilon_{F,0}$ is zero almost everywhere on $\mathbb{S}^\ast$ and as it is continuous, $\Upsilon_{F,0}=0$ on $\mathbb{S}^\ast$. Lemma \ref{lem:transaffine} (ii) implies that $\Upsilon_{F,0}=0$ on $\V^\ast$. Then, the result follows from Proposition \ref{prop:vanishinglinism}.  
\end{proof}

\subsection{A Question}

The following question as well as the study of integral sheaf metrics was already mentioned in \cite{HDRpetit}. Here, we recall this question and remark a few facts.

\begin{question} \label{conj:distance} Let $F, G \in \Derb(\cor_\V)$. When does the following equality hold
	\begin{equation*}
		\dist_\V(F,G) \stackrel{?}{=}\sup_{f \in \Lip_{\leq 1}}(\dist_\R(\roim{f}F,\roim{f}G))
	\end{equation*}
Is it sufficient to assume either $\dist_\V(F,G) < \infty$ or $F, G$ are constructible up to infinity and $\sup_{f \in \Lip_{\leq 1}}(\dist_\R(\roim{f}F,\roim{f}G)) < \infty$?
\end{question}

\begin{remark} The fact that we are working over a vector space, hence contractible, is essential for the question. Indeed, let us provide a counterexample where the space is no longer contractible. Let $\mathbb{S}^1 = \{z \in \C \mid |z| = 1\}$ equipped with its standard riemannian structure. Let $m : \mathbb{S}^1 \to \mathbb{S}^1$ defined by $m(z) = z^2$. Let $F = \cor_{\mathbb{S}^1}$ and $G = \text{R} m_\ast F$. It is clear that $F$ and $G$ are non-isomorphic local systems on $\mathbb{S}^1$. Therefore, by \cite[Proposition 2.3.7]{PS20}, $\dist_{\mathbb{S}^1}(F,G) = \infty$. Let $p : \mathbb{S}^1 \to \R$ be a $1$-Lipschitz map.  Then one has that $\text{supp}(\text{R} p_\ast F) , \text{supp}(\text{R} p_\ast G) \subset \text{Im} (p) $, in particular, the supports of $\text{R} p_\ast F$ and $\text{R} p_\ast G$ are compact. Moreover, since $p$ is 1-Lipschitz, $\text{diam}(\text{Im}(p)) \leq \text{diam}(\mathbb{S}^1) = 2$. Since $\text{R} p_\ast F$ and $\text{R} p_\ast G$ have isomorphic global sections, one deduces by \cite[Remark 2.5 (i)]{KS18} that $\dist_{\R}(\text{R} p_\ast F,\text{R} p_\ast G) \leq 2$. To summarize, we have:
\begin{equation*}
\dist_{\mathbb{S}^1}(F,G) = \infty ~~~~~\text{and}~~~~~ \sup_{f\in \Lip_{\leq 1}}(\dist_\R(\roim{f}F,\roim{f}G)) \leq 2,
\end{equation*}
which contradicts the question over $\mathbb{S}^1$.
\end{remark}

\section{Metric for multi-parameter persistence modules}

In this section, we make use of the structure of $\gamma$-sheaves to construct metrics which are efficiently computable for sublevel sets persistence modules by relying on software dedicated to one-parameter persistence modules and recent advances on optimization of topological functionals \cite{poulenard:hal-02953321}. One of these distances is an ISM whereas the second is a sliced distance. We study the properties of these two metrics.

\subsection{\texorpdfstring{$\gamma$}{gamma}-linear ISM}

We now introduced an ISM tailored for $\gamma$-sheaves. In view of Proposition \ref{prop:oimgamma}, it is natural to consider the following pseudo-distance on $\Derb_{\gamma^{\circ, a}}(\cor_\V)$ (Note that though this pseudo-distance is well defined on $\Derb(\cor_X)$, it is mostly interesting $\Derb_{\gamma^{\circ, a}}(\cor_\V)$).

\begin{equation} \label{dist:gammaism}
	\isme{\gamma^{\circ}}(F,G)=\sup_{\{u \in \gamma^\circ \mid L(u) \leq 1 \}} \dist_\R(\reim{u}F,\reim{u}G)
\end{equation}

\begin{lemma} The following equality holds
\begin{equation*}
    \isme{\gamma^{\circ}}(F,G)=\sup_{\{u \in \Int(\gamma^\circ) \mid L(u) \leq 1 \}} \dist_\R(\reim{u}F,\reim{u}G).
\end{equation*}
\end{lemma}

\begin{proof}
This follows from Lemma \ref{prop:boundvanish}.
\end{proof}
We recall that $\fL=\lbrace u \in \V^\ast \mid  L(u) \leq 1 \rbrace$.
\begin{proposition}\label{prop:egalismlin} Let $F,\; G \in \Derb_{\gamma^{\circ, a}}(\cor_\V)$ and assume that $\isme{\fL}(F,G)<\infty$ or $\isme{\gamma^{\circ}}(F,G)=\infty$. Then $\isme{\fL}(F,G)=\isme{\gamma^{\circ}}(F,G)$.
\end{proposition}

\begin{proof}
    Since $\isme{\gamma^{\circ}}$ is a lower bound of $\isme{\fL}$, the case $\isme{\gamma^{\circ}}(F,G)=\infty$ is clear.  Hence, we assume that $\isme{\fL}(F,G)<\infty$. 
    
    Let $u \notin \gamma^\circ \cup \gamma^{\circ,a}$ such that $L(u) \leq 1$.  It follows from Proposition \ref{prop:oimgamma} that $\reim{u}F$ is a constant sheaf and similarly for $\reim{u}G$. Moreover, $\dist_\R(\reim{u} F, \reim{u}G) < \infty$ since $\isme{\fL}(F,G)<\infty$. Thus it follows from \cite[Prop.~2.3.7]{PS20}, that $\reim{u}F \simeq \reim{u}G$ which implies that $\dist_\R(\reim{u} F, \reim{u}G)=0$. Morever, if $u \in \gamma^{\circ,a}$, then $v=-u \in \gamma^\circ$ and it follows from Lemma \ref{lem:transaffine} (ii) that $\dist_\R(\reim{u} F, \reim{u}G)=\dist_\R(\reim{v} F, \reim{v}G)$.  Hence, 
    \begin{align*}
       \isme{\fL}(F,G)&=\sup_{\{u \in \V^\ast \mid L(u) \leq 1 \}} \dist_\R(\reim{u}F,\reim{u}G)\\
       &= \sup_{\{u \in \gamma^\circ \mid L(u) \leq 1 \}} \dist_\R(\reim{u}F,\reim{u}G)\\
       &=\isme{\gamma^{\circ}}(F,G).
    \end{align*}
\end{proof}

\begin{corollary}
Let $F,\; G \in \Derb_{\gamma^{\circ, a}}(\cor_\V)$ with compact supports. Then $\isme{\fL}(F,G)=\isme{\gamma^{\circ}}(F,G)$.
\end{corollary}

\begin{proof}
Let $F,\; G \in \Derb_{\gamma^{\circ, a}}(\cor_\V)$ with compact supports. Since they are $\gamma$-sheaves with compact support $\rsect(\V;F) \simeq \rsect(\V;G)  \simeq 0$. Then \cite[Example 2.4]{KS18}, implies that $\dist_\V(F,G)< \infty$. Thus, $\isme{\fL}(F,G)< \infty$ and applying Proposition \ref{prop:egalismlin}, the result follows.
\end{proof}

We denote by $\Derb_{\rc , \gamma^{\circ,a} }(\cor_{\V})$ the full triangulated subcategory of $\Derb(\cor_{\V})$ spanned by the objects of $\Derb_{\gamma^{\circ,a}}(\cor_{\V}) \cap \Derb_{\rc}(\cor_{\V}) $

\begin{proposition}\label{prop:vanishingISMgamma}
Let $F \in \Derb_{\rc , \gamma^{\circ,a} }(\cor_{\V})$. Assume that $F$ is $\gamma$-compactly generated. If $\isme{\gamma^\circ}(F,0)=0$, then $F \simeq 0$.
\end{proposition}

\begin{proof}
Since $F$ is constructible, it follows from Theorem \ref{thm:closedis} that for every $u \in \Int(\gamma^\circ)$, $\reim{u} F \simeq 0$. Then Proposition \ref{prop:gammavanishproj} implies that $F \simeq 0$.
\end{proof}

\begin{proposition}\label{prop:cohodisgamma}
Let $F,G \in \Derb_{\rc , \gamma^{\circ,\,a}}(\cor_{\V_{\infty}})$, then 
\begin{align*}
\delta_{\gamma^\circ}(F,G)&= \max_{j\in \Z} \sup_{\{u \in \gamma^\circ \mid L(u) \leq 1 \}}d_B (\mathbb{B}(\Hn^j(\reim{u}F)), \mathbb{B}(\Hn^j(\reim{u}G))).
\end{align*}
\end{proposition}

\begin{proof}
Let $F,G \in \Derb_{\rc , \gamma^{\circ,\,a}}(\cor_{\V})$.
\begin{align*}
\delta_{\gamma^\circ}(F,G)&=\sup_{\{u \in \gamma^\circ \mid L(u) \leq 1 \}} \dist_\R (\reim{u}F,\reim{u}G) \\ 
&= \sup_{\{u \in \gamma^\circ \mid L(u) \leq 1 \}} \max_{j\in \Z} \dist_\R(\Hn^j(\reim{u}F),\Hn^j(\reim{u}G)) \quad \textnormal{ (Corollary (\ref{C:gradedDistance})}\\
&= \max_{j\in \Z} \sup_{\{u \in \gamma^\circ \mid L(u) \leq 1 \}} \dist_\R(\Hn^j(\reim{u}F),\Hn^j(\reim{u}G)) \\
&= \max_{j\in \Z} \sup_{\{u \in \gamma^\circ \mid L(u) \leq 1 \}}d_B (\mathbb{B}(\Hn^j(\reim{u}F)), \mathbb{B}(\Hn^j(\reim{u}G))).
\end{align*}
\end{proof}

\begin{corollary}
Let $S$ and $S^\prime$ be compact good topological spaces, $f \colon S \to \V$ and $ g\colon S^\prime \to \V$  be continuous maps. Assume that $\PH(f)$ and $\PH(g)$ are constructible. Then
\begin{align*}
    \delta_{\gamma^\circ}(\PH(f),\PH(g))&= \max_{j\in \Z} \sup_{\{u \in \gamma^\circ \mid L(u) \leq 1 \}}(d_B (\mathbb{B}(\PH^j(u\circ f)), \mathbb{B}(\PH^j(u\circ g)))).
\end{align*}
\end{corollary}

\begin{proof}
It follows from Proposition \ref{prop:cohodisgamma} that
\begin{equation*}
\delta_{\gamma^\circ}(\PH(f),\PH(g))= \max_{j\in \Z} \sup_{\{u \in \gamma^\circ \mid L(u) \leq 1 \}}(d_B (\mathbb{B}(\Hn^j(\reim{u}\PH(f)),\mathbb{B}(\Hn^j(\reim{u}\PH(g)))).
\end{equation*}
Let $u \in \Int(\gamma^\circ)$. Then $u$ is proper on $f(S) + \gamma^a$ and on $g(S^\prime) + \gamma^a$. Hence, it is proper on $\Supp(\PH(f))$ and $\Supp(\PH(g))$. Then by Lemma \ref{lem:comsublevel} (ii), $\reim{u}\PH(f) \simeq \PH(u \circ f)$ and similarly $\reim{u}\PH(g) \simeq \PH(u \circ g)$. This implies that
\begin{equation*}
    d_B (\mathbb{B}(\Hn^j(\reim{u}\PH(f)),\mathbb{B}(\Hn^j(\reim{u}\PH(g)))=d_B (\mathbb{B}(\PH^j(u\circ f)), \mathbb{B}(\PH^j(u\circ g)))
\end{equation*}
which conclude the proof.
\end{proof}

The above results leads naturally to introduce the following pseudo-metrics that we call truncated $\gamma$-integral sheaves metrics.

\begin{definition}
Let $F,G \in \Derb_{\rc , \gamma^{\circ,\,a}}(\cor_{\V})$ and $p\leq q \in \Z$. The $(p,q)$-truncated $\gamma$-integral sheaf metric is
\begin{equation*}
   \delta_{\gamma^\circ}^{p,q}(F,G)=\max_{p \leq j \leq q}(\sup_{\{u \in \gamma^\circ \mid L(u) \leq 1 \}} \dist_\R(\Hn^j(\reim{u}F), \Hn^j(\reim{u}G)))
\end{equation*}
\end{definition}

\begin{proposition}
The map    
\begin{equation*}
\delta_{\gamma^\circ}^{p,q} \colon \Ob(\Derb_{\rc , \gamma^{\circ,\,a}}(\cor_{\V})) \times \Ob( \Derb_{\rc , \gamma^{\circ,\,a}}(\cor_{\V})) \to \R_+ \cup \{ \infty \}
\end{equation*}
is a pseudo-metric and
\begin{equation*}
    \delta_{\gamma^\circ}^{p,q}(F,G) \leq \delta_{\gamma^\circ}(F,G) \leq \dist_\V(F,G).
\end{equation*}
\end{proposition}

\subsubsection{Gradient computation}

In this section, we compute explicitly the gradient of the functional $u \mapsto \dist_\R (\reim{u}F, \reim{u}G)$ when F and G are $\gamma$-sheaves arising from sublevel set persistence. We do so in order to approximate $\delta_{\gamma^\circ}$ by gradient ascent.

Let $F, G \in \Der_\gcg(\cor_\V)$. We first study the regularity of the application
\begin{equation}\label{map:upsigamma}
	\Upsilon_{F,G} \colon \Int(\gamma^{\circ}) \to \R, \; u\mapsto \dist_\R(\reim{u} F,\reim{u} G )
\end{equation}
\begin{proposition}\label{prop:Up_gammacomp_Lip}
Let $u, v \in \Int(\gamma^\circ)$ and $F,G \in \Derb_{\gcg}(\cor_\V)$. Then there exists a constant $C_{F,G}$ in $\R_{\geq 0}$ depending on $F$ and $G$ such that

\begin{equation*}
		|\Upsilon_{F,G}(u)-\Upsilon_{F,G}(v)| \leq C_{F,G} \opnorm{u-v}.
\end{equation*}
\end{proposition}

\begin{proof}

The proof is similar to the one of Lemma \ref{lem:lip}. Since $F$ and $G$ are $\gamma$-compactly generated, there exists $F^\prime$ and $G^\prime \in \Derb_\comp(\cor_\V)$ such that  $F \simeq F^\prime \npsconv \cor_{\gamma^a}$ and $G \simeq G^\prime \npsconv \cor_{\gamma^a}$. We set $S=\Supp(F^\prime) \cup \Supp(G^\prime)$. Since $S$ is compact, there exists $(x_0,x_1)$ in $S \times S$ such that $\Diam S= d(x_0,x_1)$. Moreover $\dist_\R$ is invariant by translation. Hence, setting $\tilde{u}(x)=u(x)-u(x_0)=\tau_{u(x_0)} \circ u$ and $\tilde{v}(x)=v(x)-v(x_0)=\tau_{v(x_0)} \circ u$ we get that
\begin{equation*}
\dist_\R(\reim{u} (F^\prime \npsconv \cor_{\gamma^a}),\reim{u} (G^\prime \npsconv \cor_{\gamma^a}))=\dist_\R(\reim{\tilde{u}} (F^\prime \npsconv \cor_{\gamma^a}),\reim{\tilde{u}} (G^\prime \npsconv \cor_{\gamma^a})).
\end{equation*}

Moreover,
\begin{align*}
   \roim{\tilde{u}} (F^\prime \npsconv \cor_{\gamma^a}) &\simeq \roim{\tau_{u(x_0)}} \reim{u} (F^\prime \npsconv \cor_{\gamma^a})\\
   &\simeq \roim{\tau_{u(x_0)}} (\roim{u}F^\prime \npsconv \cor_{\R_-}) \quad \textnormal{by Lemma \ref{lem:comsublevel}}\\
   &\simeq (\roim{\tau_{u(x_0)}}\roim{u})F^\prime \npsconv \cor_{\R_-}\\
   & \simeq \roim{\tilde{u}}F^\prime \npsconv \cor_{\R_-}
\end{align*}
and similarly with $\tilde{u}$ replace by $\tilde{v}$ and $F^\prime$ by $G^\prime$.
Hence
\begin{equation*}
    \Upsilon_{F,G}(u) = \dist_\R( \roim{\tilde{u}}F^\prime \npsconv \cor_{\R_-},\roim{\tilde{u}} G^\prime \npsconv \cor_{\R_-}).
\end{equation*}
Now using Proposition \ref{prop:dualdis} (ii), we get
\begin{align*}
    |\Upsilon_{F,G}(u)-\Upsilon_{F,G}(v)|&\leq \dist_\R( \roim{u}F^\prime \npsconv \cor_{\R_-},\roim{v} F^\prime \npsconv \cor_{\R_-}) + \\
    & \hspace{1cm}\dist_\R( \roim{\tilde{u}}G^\prime \npsconv \cor_{\R_-},\roim{\tilde{v}} G^\prime \npsconv \cor_{\R_-})\\
    & \leq \dist_\R( \roim{\tilde{u}}F^\prime ,\roim{\tilde{v}} F^\prime)+\dist_\R( \roim{\tilde{u}}G^\prime ,\roim{\tilde{v}} G^\prime ) \\
    & \leq 2 \norm{\tilde{u}|_S-\tilde{v}|_S}_\infty \leq  2 \Diam(S) \opnorm{u-v}.
\end{align*}
\end{proof}

The Rademacher Theorem combined with the above inequality implies the following corollary.

\begin{corollary}
Let $F, G \in \Derb_\gcg(\cor_\V)$. The application $\Upsilon_{F,G} \colon \Int(\gamma^\circ) \to \R$ (see \eqref{map:upsigamma}) is differentiable Lebesgue almost everywhere on $\Int(\gamma^\circ)$.
\end{corollary}

We choose a basis $(e_1,...,e_n)$ of $\V$, and denote $(e_1^\ast,...,e_n^\ast)$ the associated dual basis of $\V^\ast$. Given $u\in \Int{\gamma^\circ}$, we will denote $u_1,...,u_n \in \R$ the coordinates of $u$ in the basis $(e_1^\ast,...,e_n^\ast)$. Let $K_1$ and $K_2$ be two finite simplicial complexes with geometric realizations $|K_1|$ and $|K_2|$, $f : |K_1| \to \V$ and $g : |K_2| \to \V$ be PL maps.  We define: 
\begin{align*}
\mathcal{T}(u) = \dist_\R (\PH(u \circ f), \PH(u \circ g)).
\end{align*}
With $\mathcal{T}^j(u) := \dist_\R (\PH^j(u \circ f), \PH^j(u \circ f))$, we deduce from \eqref{C:gradedDistance} that: 
\begin{align*}
    \mathcal{T}(u)
    &= \max_{j\in \mathbb{Z}} \dist_\R (\PH^j(u \circ f), \PH^j(u \circ g)) \\
    &= \max_{j\in \mathbb{Z}} \mathcal{T}^j(u).
\end{align*} 

Let $\mathcal{F}$ denote the bottleneck distance function between two barcodes. For $j \in \mathbb{Z}$, let $\mathbb{B}_1^j(u) = \{[b^j_1,d^j_1)\}$ (resp. $\mathbb{B}_2^j(u) = \{[b'^j_2,d'^j_2)\}$) be the barcode of $ \PH^j(u \circ f)$ (resp. $\PH^j(u \circ g)$). Therefore, we deduce by \ref{C:gradedDistance} that:
\begin{align*}
    \mathcal{T}^j(u) &=  d_B(\mathbb{B}_1^j(u), \mathbb{B}_2^j(u))\\
   &= \mathcal{F}(\mathbb{B}_1^j(u), \mathbb{B}_2^j(u)).
\end{align*}

We now follow the exposition of \cite[section 13.2.2]{dey2022computational} for the gradient computation. Let $E^j_1 \subset \R$ (resp. $E_2^j \subset \R$) be the set of endpoints of the interval appearing in $\mathbb{B}_1^j$ (resp. $\mathbb{B}_2^j$), which we assume to be all distinct. Therefore, there exists an inverse function $\rho^u_1 : \bigsqcup_j E_1^j \to K_1 $  (resp. $\rho^u_2 : \bigsqcup_j E_2^j \to K_2 $) mapping each birth and death in $\mathbb{B}^j_1$ (resp. $\mathbb{B}^j_2$) to the corresponding simplex in $K_1$ (resp. $K_2$) generating or annihilating this bar. Moreover, this function is locally constant in $u$ in a precise sense (see \cite[section 5]{poulenardcontinuousshapematching}).

Therefore, writing $f_u = u\circ f $  and $g_u = u \circ g$ for brevity, one has the following computations:

\begin{align*}
    \frac{\partial \mathcal{T}^j }{\partial e_i^\ast}(u) &= \sum_{b \in E_1^j \sqcup E_2^j } \frac{\partial \mathcal{F} }{\partial b}(\mathbb{B}_1^j(u), \mathbb{B}_2^j(u)) \frac{\partial b }{\partial e_i^\ast}(u) \\
    &= \sum_{b \in E_1^j  } \frac{\partial \mathcal{F} }{\partial b}(\mathbb{B}_1^j(u), \mathbb{B}_2^j(u)) \frac{\partial f_u }{\partial e_i^\ast}(\rho_1^u(b)) + \sum_{b \in E_2^j  } \frac{\partial \mathcal{F} }{\partial b}(\mathbb{B}_1^j(u), \mathbb{B}_2^j(u)) \frac{\partial g_u }{\partial e_i^\ast}(\rho_2^u(b)) \\
    &= u_i \left ( \sum_{b \in E_1^j  } \frac{\partial \mathcal{F} }{\partial b}(\mathbb{B}_1^j(u), \mathbb{B}_2^j(u)) e_i^\ast (f(\rho^u_1(b)))  +  \sum_{b \in E_2^j  } \frac{\partial \mathcal{F} }{\partial b}(\mathbb{B}_1^j(u), \mathbb{B}_2^j(u))  e_i^\ast (g(\rho^u_2(b))) \right )
\end{align*}

Note that \cite[Prop.~5.3]{BG18} asserts that in the context of $\gamma$-sheaves, the function $\mathcal{F}$ coincides with the usual bottleneck distance of persistence. Therefore, one can use \cite[Lemma 3]{poulenardcontinuousshapematching} to compute the quantity $\frac{\partial \mathcal{F} }{\partial b}(\mathbb{B}_1^j(u), \mathbb{B}_2^j(u))$ which is either worth $-1$,  $0$ or $1$. We will study more in-depth the optimization of $\mathcal{T}$ in a forthcoming work.

\subsection{\texorpdfstring{$\gamma$}{gamma}-sliced convolution distance}\label{subsec:gammasliced}

In this subsection, we introduce a version of the sliced convolution distance tailored for $\gamma$-compactly generated sheaves. For the sake of simplicity, we assume that $\V=\R^n$ and that it is endowed with the norm $\norm{\cdot}_\infty$ and that $\gamma=(-\infty,0]^n$.
We endow $\V^\ast$ with the dual basis of the canonical basis of $\R^n$. We also equip $\V^\ast$ with its canonical scalar product. Then, the affine plane of $\V^\ast$ defined by the equation $\sum_{i=1}^n u_i=1$ is a Riemannian submanifold of $\V^\ast$. Hence, we equip it with its canonical Riemannian measure that we denote $du$. We define $\mathcal{Q}_\gamma= \{u \in \Int(\gamma^\circ) \mid  \opnorm{u}_\infty = 1\}$. In other words $\mathcal{Q}_\gamma=\{(u_1,\ldots,u_n) \in \R_{> 0}^n \mid \sum_{i=1}^n u_i=1 \}$. We consider the restriction of $du$ to $\mathcal{Q}_\gamma$ and set $\vol(\mathcal{Q}_\gamma)=\int_{Q_\gamma} du$. 

Let $F,G \in \Derb_{\gcg}(\cor_\V)$. It follows from Proposition \ref{prop:Up_gammacomp_Lip} that the map
\begin{equation*}
	\Upsilon_{F,G} \colon \Int(\gamma^\circ) \to \R, \quad u \mapsto\dist_\R(\reim{u}F,\reim{u}G)
\end{equation*}
is Lipschitz on $Q_\gamma$, hence measurable.

For $p \in \N^\ast$, we define the $p^\mathrm{th}$ $\gamma$-sliced convolution distance between $F$ and $G$ by
\begin{equation}\label{def:gammasliced}
\sliced_{\gamma,p}(F,G)=\frac{1}{\vol(\mathcal{Q}_\gamma)} \left( \int_\mathcal{Q_\gamma} \dist_\R(\reim{u}F,\reim{u}G)^p \, du \right)^{\frac{1}{p}}.
\end{equation}

\begin{proposition}
The application $\sliced_{\gamma,p} \colon  \Ob(\Derb_{\gcg}(\cor_\V)) \times \Ob(\Derb_{\gcg}(\cor_\V)) \to \R_+$ is a pseudo-distance and for $F, \, G \in \Derb_{\gcg}(\cor_{\V})$
\begin{equation*}
    \sliced_{\gamma,p}(F,G) \leq \dist_\V(F,G).
\end{equation*}
\end{proposition}

\begin{lemma}
Let $F \in \Derb_{\rc,\gcg}(\cor_{\V_\infty})$. If $\sliced_{\gamma,p}(F,0)=0$, then $F \simeq 0$.
\end{lemma}

\begin{proof}
It follows from the hypothesis that the function $\Upsilon^\gamma_{F,0}=0$ almost everywhere on $\mathcal{Q}_\gamma$ and as it is continuous, $\Upsilon^\gamma_{F,0}=0$ on $Q_\gamma$. Lemma \ref{lem:transaffine} implies that $\Upsilon^\gamma_{F,0}=0$ on $\gamma^\circ$. Then, the result follows from Proposition \ref{prop:vanishingISMgamma}.  
\end{proof}

\subsection{Some aspects of the computation of the \texorpdfstring{$\gamma$}{gamma}-linear ISM and the \texorpdfstring{$\gamma$}{gamma}-sliced convolution distance}

As a proof of concept, we showcase here the results of our implementation of the estimation of the $\gamma$-linear ISM in the following context. We endow $\R^2$ with the cone $\gamma = (-\infty, 0]^2$ and the norm $\|\cdot\|_\infty$. Therefore, with $p_1$ and $p_2$ the coordinate projections, we have $Q_\gamma = \{p_t = t \cdot p_1 + (1-t) \cdot p_2 \mid t\in [0,1]\}$. 

We sampled $300$ points uniformly on a circle of radius $1$, to which we added a radial uniform noise in $[0, 0.1]$. We call this first dataset $X$. The second dataset $Y$ was obtained by using the same sampling process than for $X$, then appending to it $10$ outlier points drawn uniformly in $[-1,1]^2$.

\begin{figure}[h]
    \centering
    \includegraphics[width = 0.45\textwidth]{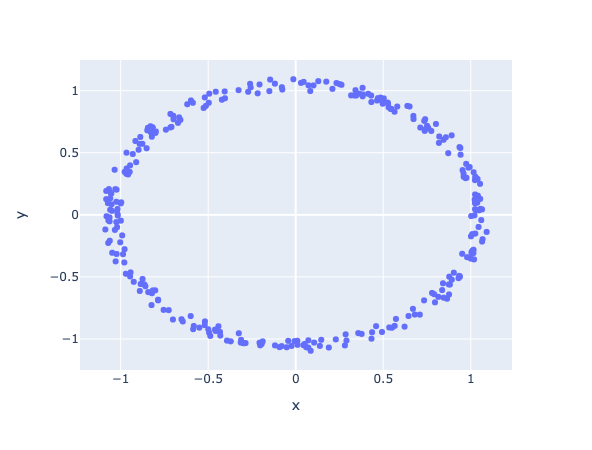}
    \includegraphics[width = 0.45\textwidth]{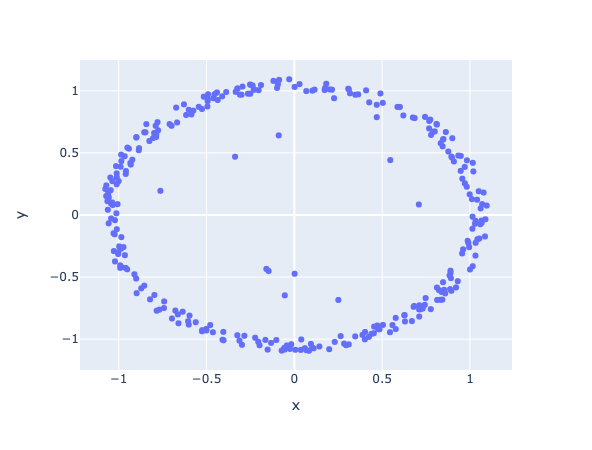}
    \caption{The datasets $X$ (left) and $Y$ (right)}
\end{figure}

We then performed a Kernel Density Estimation for $X$ and $Y$, outputting two functions $\kappa_X, \kappa_Y : [-1,1]^2 \to \R^+$, leading to two bi-filtrations of $[0,1]^2$: $f_X = (d_X, - \kappa_X)$ and $f_Y = (d_Y, - \kappa_Y)$, where $d_A$ denotes the distance function to the set $A$.  We computed the bottleneck distance between the projected barcodes $\mathbb{B}(\PH(p_t \circ f_X)$ and $\mathbb{B}(\PH(p_t \circ f_Y)$, as the parameter $t$ of the $1$-Lispchitz projection varies.  The $\gamma$-linear ISM is the maximum of these distances, while the $p$-sliced distance correspond to the $p$-th root of the integral of the $p$-th power of this function (See figure \eqref{fig:ISM-plot}).

\begin{figure}[h]
\centering
 \begin{subfigure}{0.85\textwidth}
    \centering
    \includegraphics[width =\textwidth]{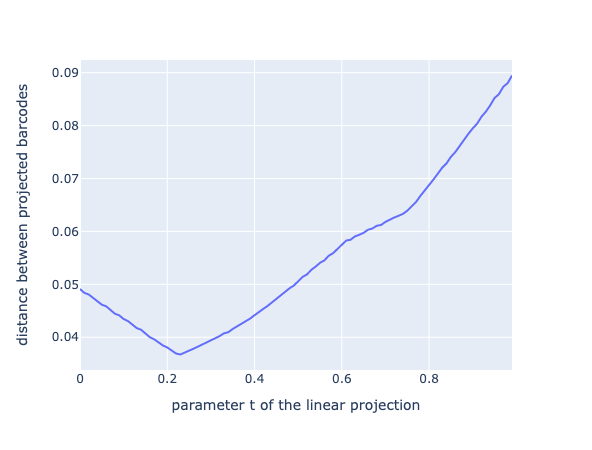}
    \caption{Bottleneck distance between projected barcodes, in function of the parameter of the $1$-Lipschitz projection. Note that $t=0$ corresponds to the $1$-filtration induced by the (co)density estimator on each pointcloud filtration, while $t = 1$ corresponds to the $1$-filtration associated to the distance function to the point cloud. }
    \label{fig:ISM-plot}
 \end{subfigure}
 
 \begin{subfigure}{0.85\textwidth}
    \centering
    \includegraphics[width =\textwidth]{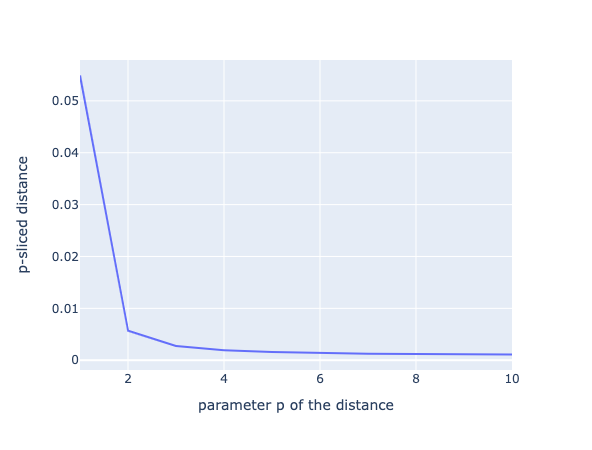}
    \caption{Value of $\mathcal{S}_{\gamma,p}(\PH(f_X), \PH(f_Y))$, as $p$ varies.}
 \end{subfigure}
 \caption{}
\end{figure}

\begin{table}[h]
    \centering
    \begin{tabular}{|l||c|c|c|c|c|}
         \hline
        distance type & $\mathcal{S}_{\gamma,1}$ & $\mathcal{S}_{\gamma,2}$ & $\mathcal{S}_{\gamma,3}$ & $\mathcal{S}_{\gamma,4}$ & $\delta_{\gamma^\circ}$ \\ \hline
         value & 0.055 & 0.0057 & 0.0027 & 0.0019 & 0.090 \\ \hline 
         
    \end{tabular}
    
    \caption{Values of the different distances between $\PH(f_X)$ and $\PH(f_Y)$}.
    \label{tab:value_distances}
\end{table}

\end{document}